\documentclass[11pt]{article}
\usepackage{enumitem}
\usepackage[font=small,format=plain,labelfont=bf,up]{caption}
\usepackage[a4paper,nohead,ignorefoot,%
twoside,scale=0.852,
twoside,scale=0.86, %
bindingoffset=1.25cm]{geometry}
\usepackage{amssymb,amsfonts,amsmath,amsthm}
\usepackage[]{graphicx}
\usepackage{color}
\graphicspath{{figures/}}
\newcommand{\url}[1]{#1} 
\usepackage{hyperref}%
\definecolor{gray}{rgb}{0.2,0.2,.2}
\hypersetup{%
  unicode=true,          
  colorlinks=true,       
  linkcolor=gray,          
  citecolor=gray      
}
\DeclareMathOperator{\mhRe}{\mathrm{Re}}
\DeclareMathOperator{\mhIm}{\mathrm{Im}}

\newcommand{\iu}{\mathtt{i}}
\newcommand{\mhexp}[1]{{{\mathtt{e}}^{#1}}}

\newcommand{\fspace}[1]{{\mathsf{#1}}}
\newcommand{\fspaceL}{\fspace{L}}
\newcommand{\fspaceH}{\fspace{H}}
\newcommand{\fspaceC}{\fspace{C}}

\newcommand{\ol}[1]{{\overline{#1}}}

\newcommand{\Rset}{{\mathbb{R}}}
\newcommand{\Zset}{{\mathbb{Z}}}
\newcommand{\Cset}{{\mathbb{C}}}

\newcommand{\ocinterval}[2]{(#1,\,#2]}%
\newcommand{\cointerval}[2]{[#1,\,#2)}%
\newcommand{\oointerval}[2]{(#1,\,#2)}%
\newcommand{\ccinterval}[2]{[#1,\,#2]}%

\newcommand{\DO}[1]{{O\at{#1}}}
\newcommand{\Do}[1]{{o\at{#1}}}
\newcommand{\nDO}[1]{{O\nat{#1}}}
\newcommand{\nDo}[1]{{o\nat{#1}}}
\newcommand{\bDO}[1]{{O\bat{#1}}}

\newcommand{\odd}{{\rm \,odd}}
\newcommand{\even}{{\rm \,even}}

\newcommand{\tdots}{{...}}%
\newcommand{\supp}{{\rm supp}}

\newcommand{\loc}{{\rm loc}}

\newcommand{\floor}[1]{{\left\lfloor{#1}\right\rfloor}}

\newcommand{\pair}[2]{{\left({#1},\,{#2}\right)}}
\newcommand{\skp}[2]{{\left\langle{#1},\,{#2}\right\rangle}}

\newcommand{\bskp}[2]{{\big\langle{#1},\,{#2}\big\rangle}}

\newcommand{\npair}[2]{{({#1},\,{#2})}}
\newcommand{\bpair}[2]{{\big({#1},\,{#2}\big)}}
\newcommand{\Bpair}[2]{{\Big({#1},\,{#2}\Big)}}
\newcommand{\at}[1]{{\left({#1}\right)}}
\newcommand{\nat}[1]{(#1)}
\newcommand{\bat}[1]{{\big(#1\big)}}
\newcommand{\Bat}[1]{{\Big(#1\Big)}}

\newcommand{\ato}[1]{{\left[{#1}\right]}}
\newcommand{\nato}[1]{[#1]}
\newcommand{\bato}[1]{{\big[#1\big]}}
\newcommand{\Bato}[1]{{\Big[#1\Big]}}

\newcommand{\triple}[3]{{\left({#1},\,{#2},\,{#3}\right)}}

\newcommand{\btriple}[3]{{\big({#1},\,{#2},\,{#3}\big)}}
\newcommand{\ul}[1]{\underline{#1}}

\newcommand{\norm}[1]{\left\|{#1}\right\|}
\newcommand{\nnorm}[1]{\|{#1}\|}
\newcommand{\bnorm}[1]{\big\|{#1}\big\|}

\newcommand{\abs}[1]{\left|{#1}\right|}
\newcommand{\babs}[1]{\big|{#1}\big|}
\newcommand{\Babs}[1]{\Big|{#1}\Big|}
\newcommand{\nabs}[1]{|{#1}|}

\newcommand{\dint}[1]{\,\mathrm{d}#1}

\newcommand{\al}{{\alpha}}
\newcommand{\be}{{\beta}}

\newcommand{\eps}{{\varepsilon}}

\newcommand{\ka}{{\kappa}}
\newcommand{\la}{{\lambda}}

\newcommand{\om}{{\omega}}

\newcommand{\calC}{\mathcal{C}}
\newcommand{\calD}{\mathcal{D}}
\newcommand{\calE}{\mathcal{E}}
\newcommand{\calF}{\mathcal{F}}

\newcommand{\calK}{\mathcal{K}}
\newcommand{\calL}{\mathcal{L}}
\newcommand{\calM}{\mathcal{M}}

\newcommand{\calR}{\mathcal{R}}
\newcommand{\calS}{\mathcal{S}}

\makeatother
\theoremstyle{plain}
\newtheorem{theorem}             {Theorem}[]
\newtheorem{corollary}  [theorem]{Corollary}

\newtheorem{lemma}      [theorem]{Lemma}

\newtheorem{result}{Main result}

\newtheorem{assumption} [theorem]{Assumption}

\numberwithin{figure}{section}
\numberwithin{table}{section}
\numberwithin{equation}{section}
\usepackage{float}
\begin{document}
%
%
\title{\vspace{-0.025\textheight}Stability of high-energy solitary waves \\ in Fermi-Pasta-Ulam-Tsingou chains}
\date{\today}
\author{%
Michael Herrmann%
\footnote{Technische Universit\"at Braunschweig, Germany, {\tt michael.herrmann@tu-braunschweig.de}  }
\and
Karsten Matthies%
\footnote{University of Bath, United Kingdom, {\tt k.matthies@bath.ac.uk}}
} %
\maketitle
\vspace{-0.025\textheight}%
%
%
%
%
%
\begin{abstract}
The dynamical stability of solitary lattice waves in non-integrable FPUT chains is a longstanding open problem and has  been solved so far only in a certain asymptotic regime, namely by Friesecke and Pego for the KdV limit, in which the waves propagate with near sonic speed, have large wave length, and carry low energy. In this paper we derive a similar result in a complemen\-tary asymptotic regime related to fast and strongly localized waves with high energy. In particular, we show that the spectrum of the linearized FPUT operator contains asymptotically no unstable eigenvalues except for the neutral ones that stem from the shift symmetry and the spatial discreteness. This ensures that high-energy waves are linearly stable in some orbital sense, and the corresponding nonlinear stability is granted by the general, non-asymptotic part of the seminal Friesecke-Pego result   and the extension by Mizumachi.
\par
Our analytical work splits into two principal parts. First we refine
two-scale \mbox{techniques} that relate high-energy wave to a nonlinear asymptotic shape ODE and provide \mbox{accurate} approxi\-mation formulas. In this way we establish the existence, local uniqueness, smooth \mbox{parameter} depen\-dence, and exponential localization of fast lattice waves for a wide class of interaction \mbox{potentials} with algebraic singularity. Afterwards we study the crucial eigenvalue problem in exponen\-tially weighted spaces, so that there is   no   unstable essential spectrum. Our key \mbox{argument} is that all proper eigenfunctions can asymptotically be linked to the unique bounded and \mbox{normalized} solution of the linearized shape ODE, and this finally enables us to disprove the existence of unstable eigenfunctions in the symplectic complement of the neutral ones.
\end{abstract}
%
%
%
\quad\newline\noindent%
\begin{minipage}[t]{0.15\textwidth}%
 Keywords:
\end{minipage}%
\begin{minipage}[t]{0.8\textwidth}%
\emph{Hamiltonian lattices waves}, \emph{high-energy limit of FPU or FPUT chains,}\\
\emph{nonlinear orbital stability}, \emph{asymptotic analysis},
\end{minipage}%
\medskip
\newline\noindent
\begin{minipage}[t]{0.15\textwidth}%
MSC (2010): %
\end{minipage}%
\begin{minipage}[t]{0.8\textwidth}%
37K60,  
37K40,  
70H14, 	
74H10  	
\end{minipage}%
%
%
%
%
\setcounter{tocdepth}{3}
\setcounter{secnumdepth}{3}{\scriptsize{\tableofcontents}}
%
%
%
%
\section{Introduction}
\label{sect:Intro}
%
%

This paper deals with the long-time dynamics in a prototypical example of
Hamiltonian mass-spring systems, namely the nonlinear atomic chain introduced by
Fermi, Pasta, Ulam, and Tsingou in \cite{FPU55}. A key question in this context is whether coherent states persist or to which extent energy is transferred to higher modes leading to the process of thermalization. Various types of coherent motion exist for Fermi-Pasta-Ulam-Tsingou (FPUT) chains (exact breathers, different types of traveling waves) but it is in general not known whether solutions that start near such coherent states remain coherent for all times. Pioneering work for solitary waves was done by Friesecke and Pego in a series of papers \cite{FP99,FP02,FP04a,FP04b}. They set up a general framework for nonlinear orbital stability and established  the crucial linear stability for small-amplitude, long-wavelength traveling waves with speed just above the speed of sound, for which the Korteweg-de Vries (KdV) equation is the appropriate scaling limit. There are many extensions of their results as detailed below, but these still deal with situations near the KdV limit or the special situation of the Toda lattice, which is completely integrable. In this paper we study high-energy solitary waves, which are localised on an essentially single moving particle close to a hard-sphere collision, for a wide class of inter-action potentials far away from any obvious continuum approximation. We will provide orbital stability results for open sets of initial data near a universal family of high-energy traveling waves.   A crucial part is to establish linear stability for singular nonlocal operators in the limit of high velocities.

\par

The FPUT   chain describes an infinite chain of masses which
are separated by nonlinear springs and interact according to Newton's law of motion via
\begin{align}
\label{Eqn:FPU0}
\ddot{q}_j=\Phi^\prime\bat{q_{j+1}-q_j}-\Phi^\prime\bat{q_{j}-q_{j-1}}\,,\qquad j \in \Zset\,.
 \end{align}
Here, $q_j\at{t}$ denotes the position of particle $j$ at time $t$ and $\Phi$ describes the potential for the nearest neighbor forces. Using atomic distances $r_j$ and velocities $v_j$ with
\begin{align*}
r_j\at{t}:=q_{j+1}\at{t}-q_{j}\at{t}\,,\qquad v_j\at{t}:=\dot{q}_j\at{t}
\end{align*}
we obtain the first-order formulation
\begin{align}
\label{Eqn:FPU}
\dot{r}_j = v_{j+1}-v_j\,,\qquad
\dot{v}_j = \Phi^
\prime\at{r_j}-\Phi^
\prime\at{r_{j-1}}\,.
\end{align}
In this paper we study coherent behavior in the form of traveling waves, which satisfy the ansatz
\begin{align}
\label{Eqn:TWAnsatz}
r_j\at{t}=R_\omega\bat{x+\tfrac12}\,,\qquad  v_j\at{t}=V_\omega\at{x}\,,
\end{align}
where the profiles functions $R_\om$ and $V_\om$ have to be determined in dependence of the wave speed $\om>0$ and as functions in $x=j-\om t$, the space variable in the co-moving frame. Notice that the additional shift in the distance component has been chosen for convenience. It guarantees the existence of even wave profiles and allows us to write the   non-local   traveling wave equation as
\begin{align}
\label{Eqn:TWNonlIdentities.V1}
\omega \partial_x V_\omega+\nabla \Phi^\prime\at{R_\omega}=0\,,
\qquad
\omega \partial_x R_\omega+\nabla V_\omega=0\,,
\end{align}
where all spatial differences are expressed in terms of the centered nabla operator
\begin{align}
\label{Eqn:DefNabla}
\at{\nabla U}\at{x}= U\at{x+\tfrac12}-U\at{x-\tfrac12}\,.
\end{align}
Moreover, eliminating the velocities in \eqref{Eqn:TWNonlIdentities.V1} we find
the second-order advance-delay-differential equation
\begin{align}
\label{Eqn:NonlADDE2Order}
\omega^2R_\omega^{\prime\prime}=\Delta\Phi^\prime\at{R_\omega}
\end{align}
with discrete Laplacian
\begin{align}
\label{Eqn:DefLapl}
\at{\Delta U}\at{x}= U\at{x-1}+U\at{x+1}-2U\at{x}\,.
\end{align}
The existence of different types of waves (periodic, homoclinic, heteroclinic) has been proven for many potentials by several authors and using different methods. We refer to \cite{Her10,HM15}  for an overview but emphasize that rigorous results concerning the uniqueness and the dynamical stability of FPUT-type lattice waves exist only for the few integrable examples and in the universal KdV regime, see \cite{Tes01} for an overview concerning the Toda chain, and \cite{DHM06} for both the harmonic case and the hard-sphere limit. The KdV limit has been introduced formally in \cite{ZK65} and studied rigorously in \cite{FP99,FP02,FP04a,FP04b}. The stability of a single solitary wave with respect to the natural energy norms is proven in \cite{Miz09} and the linear stability part has been simplified in \cite{HW13b}. The analysis has been extended to interacting KdV waves in \cite{HW08,HW09} and $N$-solitary waves are studied in \cite{Miz11,Miz13}. Further approaches to the stability problem are given \cite{HW13a,KP17} and the orbital stability of large amplitude waves is shown in \cite{MizPe08,BHW12} for the Toda lattice. Modulation equations and the related concept of long- but finite-time stability of KdV-type solutions are discussed in \cite{SW00} for FPUT chains, and in \cite{CCPG12,GMWZ14} for other atomistic systems. Finally, more general existence proofs for KdV waves   in lattices   can be found in \cite{FM14,HML15,HW17}.
\par
Notice also that the mathematical theory of traveling waves in spatially continuous Hamiltonian systems is much more developed. First, traveling waves in such systems are determined by ODEs instead of advance-delay-differential equations and hence often known almost explicitly. Secondly, the stability investigation benefits on the linear level from the rather extensive knowledge about linear differential operators (with non-constant coefficients) and on the nonlinear side from the shift symmetry, which gives rise to a further conserved quantity according to Noether's theorem and provides easy control concerning accelerations of a given wave. Consequently, the orbital stability of solitary waves is now well-understood for a huge class of dispersive PDEs, see for instance \cite{Ang09} for an overview.
%
%
\subsection{Setting of the problem and main results}
%
%
In this paper we suppose that the interaction potential $\Phi$ possesses a repulsive algebraic singularity as illustrated in the left panel of Figure \ref{Fig:Shock} and rely on the following standing   assumptions.
\begin{assumption}[interaction potential]
\label{Ass.Pot}
The potential $\Phi$ is   a strictly convex $\fspaceC^3$-function on the semiopen interval $\ocinterval{-1}{0}$. Moreover, it   admits a global minimum at $r=0$ with
\begin{align*}
\Phi\at{0}=\Phi^\prime\at{0}=0\,,\qquad \Phi^{\prime\prime}\at{0}>0\,,
\end{align*}
and posses at $r=-1$ a normalized algebraic singularity of order $m$. The latter means
\begin{align*}
\abs{m+2+\Phi^{\prime\prime\prime}\at{r}\at{1+r}^{m+3}}\leq C \at{1+r}^k
\end{align*}
for all $r\in\oointerval{-1}{-1/4}$ and hence
\begin{align*}
\abs{1+\at{m+1}\Phi^\prime\at{r}\at{1+r}^{m+1}}+
\abs{-1+\Phi^{\prime\prime}\at{r}\at{1+r}^{m+2}}\leq C \at{1+r}^k
\end{align*}
for all $r\in\oointerval{-1}{-1/4}$ and for some real-valued exponents $m>0$ and $k>0$ and some constant $C$.  For technical reasons we also assume  $m\neq k$ as well as $\min\{k,\,m\}>1$.
\end{assumption}
Prototypical examples of Assumption~\ref{Ass.Pot} are
\begin{align}
\label{Eqn:ExPot1}
\Phi\at{r}=\frac{1}{m\at{m+1}}\at{\frac{1}{\at{1+r}^m}+mr-1}
\end{align}
with $k=m+1$ and free parameter $m>1$, as well as the Lennard-Jones-type potential
\begin{align}
\label{Eqn:ExPot2}
\Phi\at{r}=\frac{1}{2n\at{2n+1}}\at{\frac{1}{\at{1+r}^n}-1}^2
\end{align}
with free parameter $n>1/2$ corresponding to $m=2k=2n$, where the classical choice is $n=6$. Further examples and a more detailed discussion concerning  $m$ and $k$ is given below in \S\ref{sect:discussion}, and we mention that $\Phi$ can  be extended to a smooth and strictly convex  function on $\oointerval{-1}{\infty}$.
\par
We further restrict our considerations to waves that propagate with high speed $\om\gg1$ and must hence carry a large amount of energy
\begin{align}
\label{Eqn:TotalEnergy}
h_\omega :=
\int\limits_\Rset\tfrac12V_\omega\at{x}^2+\Phi\bat{R_\omega\at{x}}\dint{x}.
\end{align}
In this asymptotic regime, the wave profiles $R_\om$ and $V_\om$ from \eqref{Eqn:TWNonlIdentities.V1} can be approximated by the piecewise constant functions
\begin{align}
\label{Eqn:LimitProfiles}
R_\infty\at{x}=\min\big\{0, -1+\abs{x}\big\}\,,\qquad \hat{V}_\infty\at{x} = \chi\at{x}:=\left\{
\begin{array}{lcl} 1&&\text{for $\abs{x}<\frac12$},\\
0&&\text{else},
\end{array}
\right.
\end{align}
which are -- up to elementary transformations -- naturally related to the traveling waves in the hard-sphere model of elastically colliding particles, see Figures \ref{Fig:WavesDist} and \ref{Fig:WavesVel} for numerical examples and Figure \ref{Fig:Shock} for a schematic representation of the limit dynamics. The high-energy limit of traveling waves in singular potentials has been studied in \cite{FM02} within a variational setting and in \cite{Tre04} by fixed-point arguments. Moreover, the authors of the present paper refined the underlying asymptotic analysis  in \cite{HM15} and showed that the fine structure of the wave profiles is determined by an asymptotic shape ODE, which provides more accurate approximation formulas for the wave profiles as well as the leading order terms in the scaling laws for all other quantities of interest. Related results for non-singular but rapidly growing potentials are discussed in \cite{Her10,Her17}.
\par
However, \cite{HM15} does not guarantee that the exact wave profiles depend smoothly on the wave speed $\om$ but, roughly speaking, only that the leading order terms in the expansion around $\om=\infty$ do so. Similarly, the asymptotic estimates control the approximation error $R_\om-R_\infty$ but not the  properties of $\partial_\om R_\om$. Since both issues are essential for our spectral analysis, we  reinvestigate in \S\ref{sect:NonlWaves} the high-energy limit on the level of the nonlinear profile functions and present, as explained below, a novel and more robust asymptotic approach to the shape ODE, which also provides an analogue to the local uniqueness results from \cite{Tre04,HM17}. Our main findings are formulated in Theorem \ref{Thm:ExistenceNonlWaves}, Theorem \ref{Thm:SmoothnessNonlWaves}, Lemma \ref{Lem:ApproxVelProfile}, and Lemma \ref{Lem:Regularity}, and can informally be summarized as follows.
\begin{result}[family of high-energy waves parameterized by speed $\om$] \label{res:Ex}
Under Assumption \ref{Ass.Pot} there exists a family of solitary solutions
$(R_\om,V_\om)$  to the traveling wave equations \eqref{Eqn:TWNonlIdentities.V1} such that the following properties are satisfied for all sufficiently large $\om$:
\begin{enumerate}
\item
\emph{\ul{Qualitative behavior for fixed $\om$}}. The profiles $R_\om$ and $V_\om$ are even and exponentially decaying. Moreover, $R_\om$ takes values in $\ocinterval{-1}{0}$ and $V_\om$ is non-negative.
\item
\emph{\ul{Local uniqueness}}. The solution branch $\om\mapsto$ $\pair{R_\om}{V_\om}$  is isolated in natural function spaces.
\item
\emph{\ul{Smooth parameter dependence}}. The map $\om\mapsto \pair{R_\om}{V_\om}$ yields a smooth curve in some Sobolev space with exponentially growing weights for $x\to\pm\infty$.
\item
\emph{\ul{Accurate approximations}}. The profiles $\pair{R_\om}{V_\om}$ as well as $\pair{\partial_\om R_\om}{\partial_\om V_\om}$ and
$\pair{\partial_x R_\om}{\partial_x V_\om}$  can be approximated  by the unique solution to an asymptotic ODE initial value problem and the solution space of   the corresponding linearized equation.
\end{enumerate}
\begin{figure}[ht!] %
\centering{ %
\includegraphics[width=0.95\textwidth]{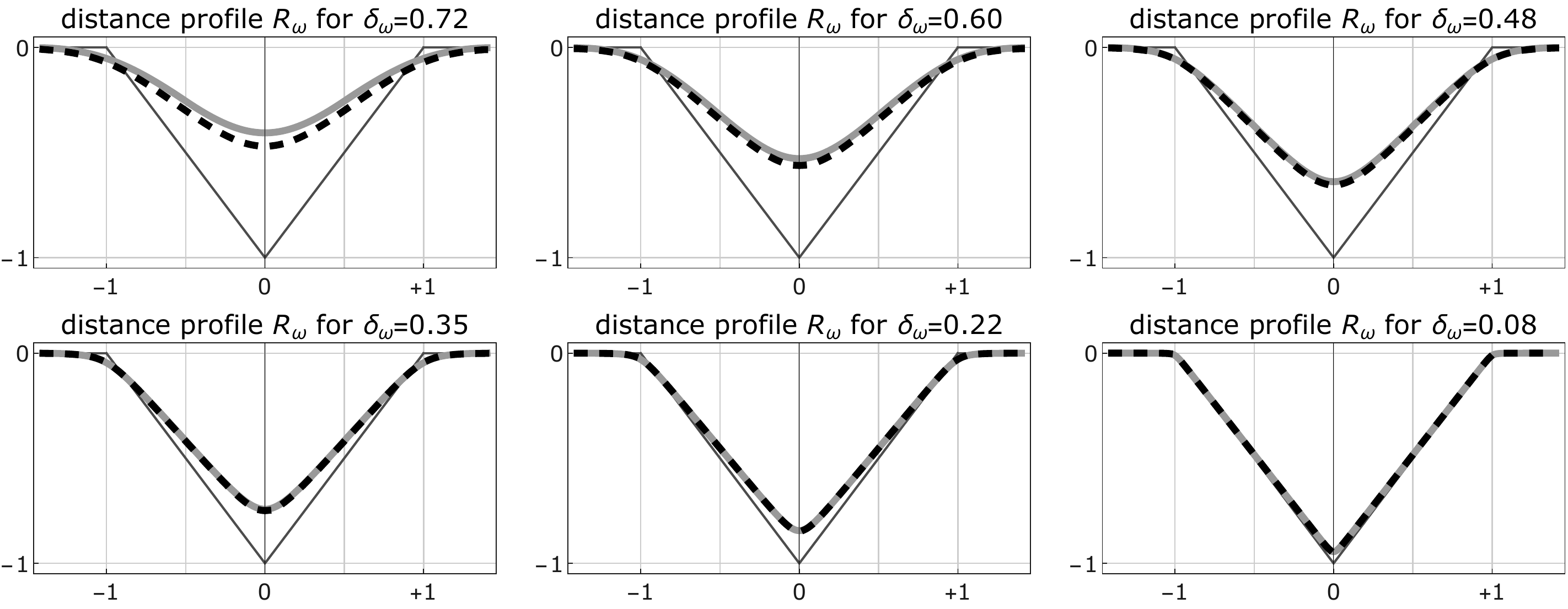}
} %
\caption{ %
Distance profiles in the high energy limit $\delta_\omega=\omega^{-2m}\to0$ for the Lennard-Jones potential \eqref{Eqn:ExPot2} with $n=2$ corresponding to $m=4$ and $k=2$ in Assumption~\ref{Ass.Pot}: Numerical solutions (black, dashed) to the nonlocal equations \eqref{Eqn:TWNonlIdentities.V2} along with the ODE approximations from Lemma~\ref{Lem:DefBreveR} (gray, solid) and the tent-shaped limit   $R_\infty$ from \eqref{Eqn:LimitProfiles}   (dark gray, thin line). All profile functions are plotted against $x$, the unscaled space variable in the comoving frame. The precise approximation results are formulated in Theorem~\ref{Thm:ExistenceNonlWaves} and Lemma~\ref{Lem:ApproxVelProfile}, where $x\in \ccinterval{-\tfrac12}{+\tfrac12}$ corresponds to $\tilde{x}\in I_\omega$, and the approximation error is plotted in Figure~\ref{Fig:WavesErr}.
} %
\label{Fig:WavesDist} %
\end{figure} %
\begin{figure}[ht!] %
\centering{ %
\includegraphics[width=0.95\textwidth]{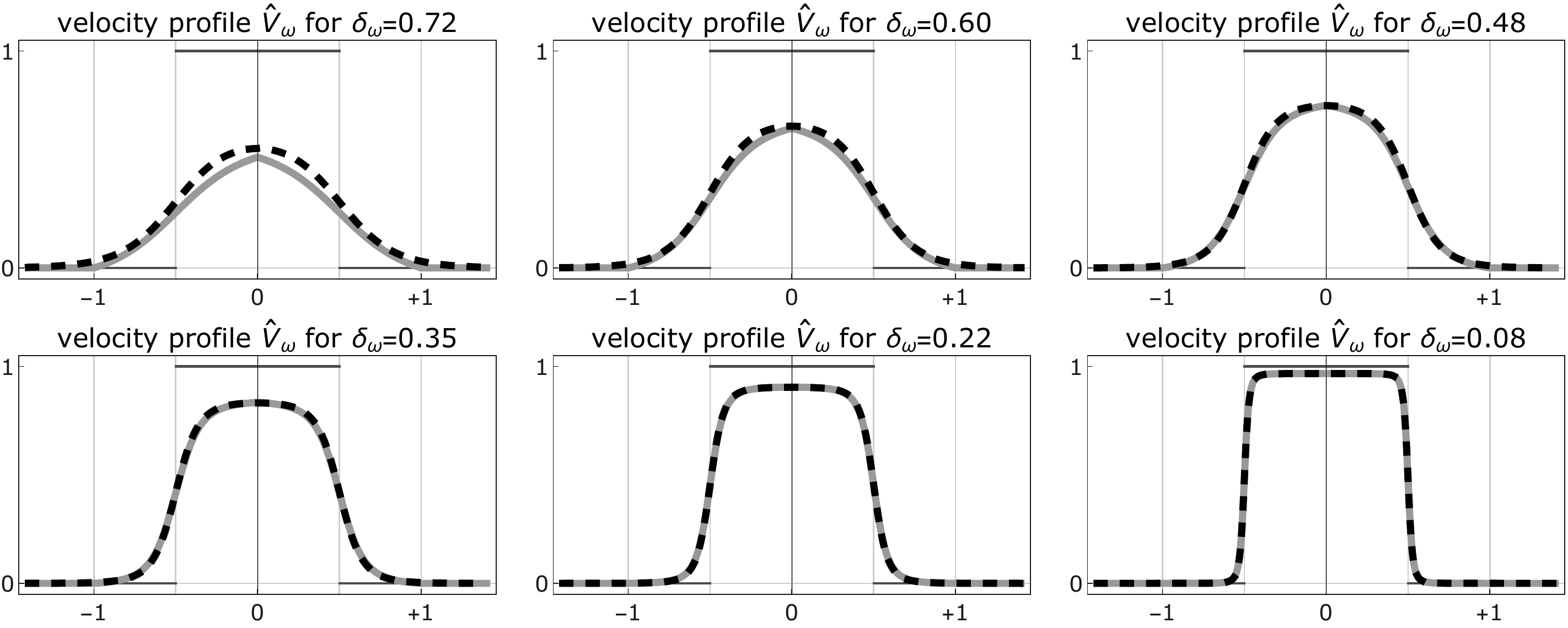}
} %
\caption{ %
Normalized velocity profiles $\hat{V}_\omega=\om^{-1} V_\omega$ corresponding to Figure~\ref{Fig:WavesDist}, which approach the indicator   function  $\hat{V}_\infty$ of the interval $\ccinterval{-1/2}{+1/2}$, see \eqref{Eqn:LimitProfiles}.   The numerical scheme to compute the nonlinear waves relies on the maximization of the potential energy under the constraint $\nnorm{\hat{V}}_2=1-\hat{\delta}$,   is easy to implement, and provides traveling waves $\triple{\omega}{V}{R}$ depending on the free parameter $\hat{\delta}$, see \cite{Her10,HM15} for the details. Our existence result \S\ref{sect:NonlWaves} is not based on constrained optimization but on explicit approximation formulas and nonlocal fixed point arguments, and yields waves parameterized by $\om$. In the high energy limit, both approaches are equivalent since we have $\omega\sim\hat{\delta}^{-m}$ and hence $\hat{\delta}\sim\delta$.
} %
\label{Fig:WavesVel} %
\end{figure}

\begin{figure}
\centering{ %
\includegraphics[width=0.275\textwidth]{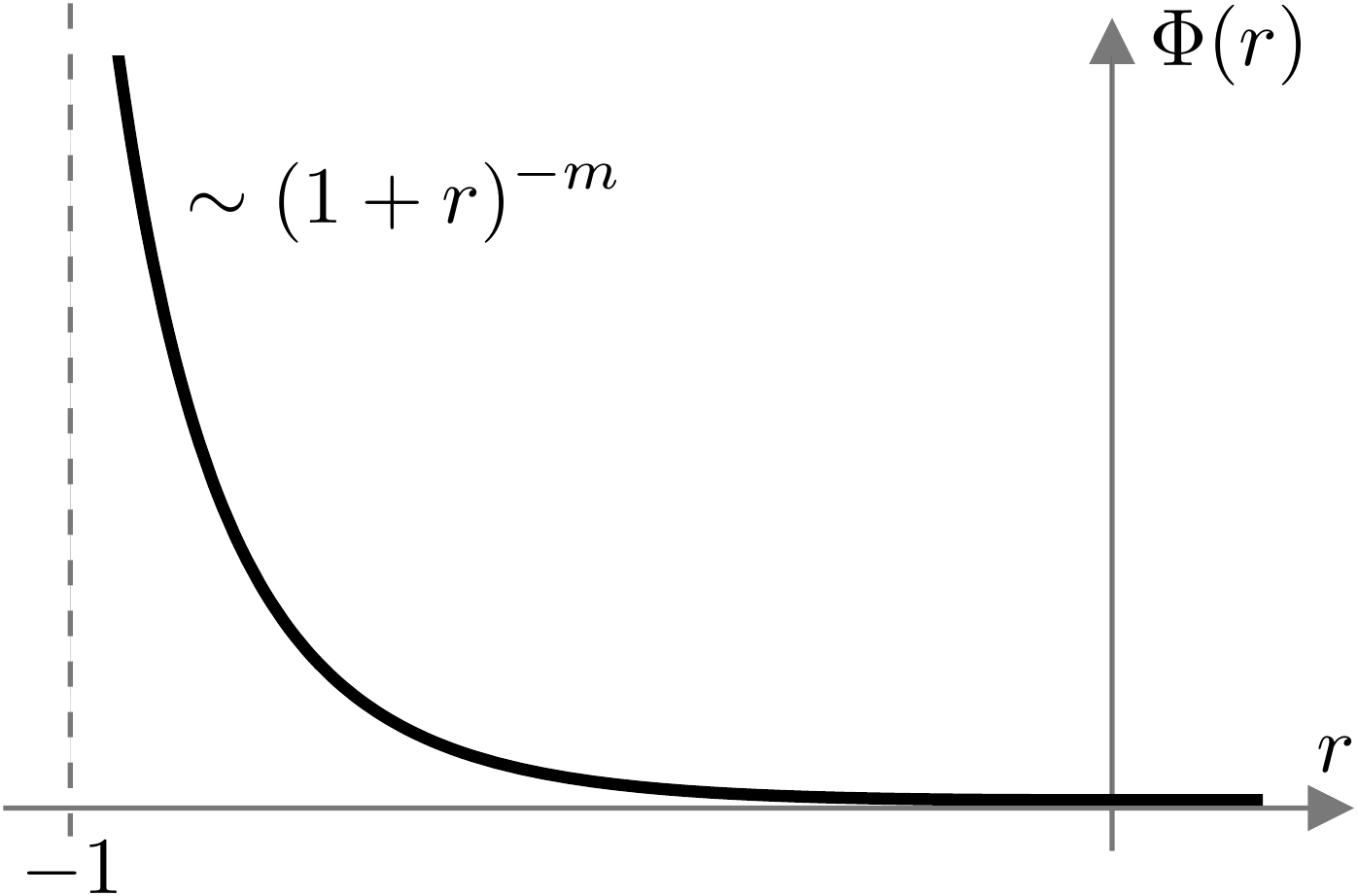} %
\hspace{.15\textwidth} %
\includegraphics[width=0.35\textwidth]{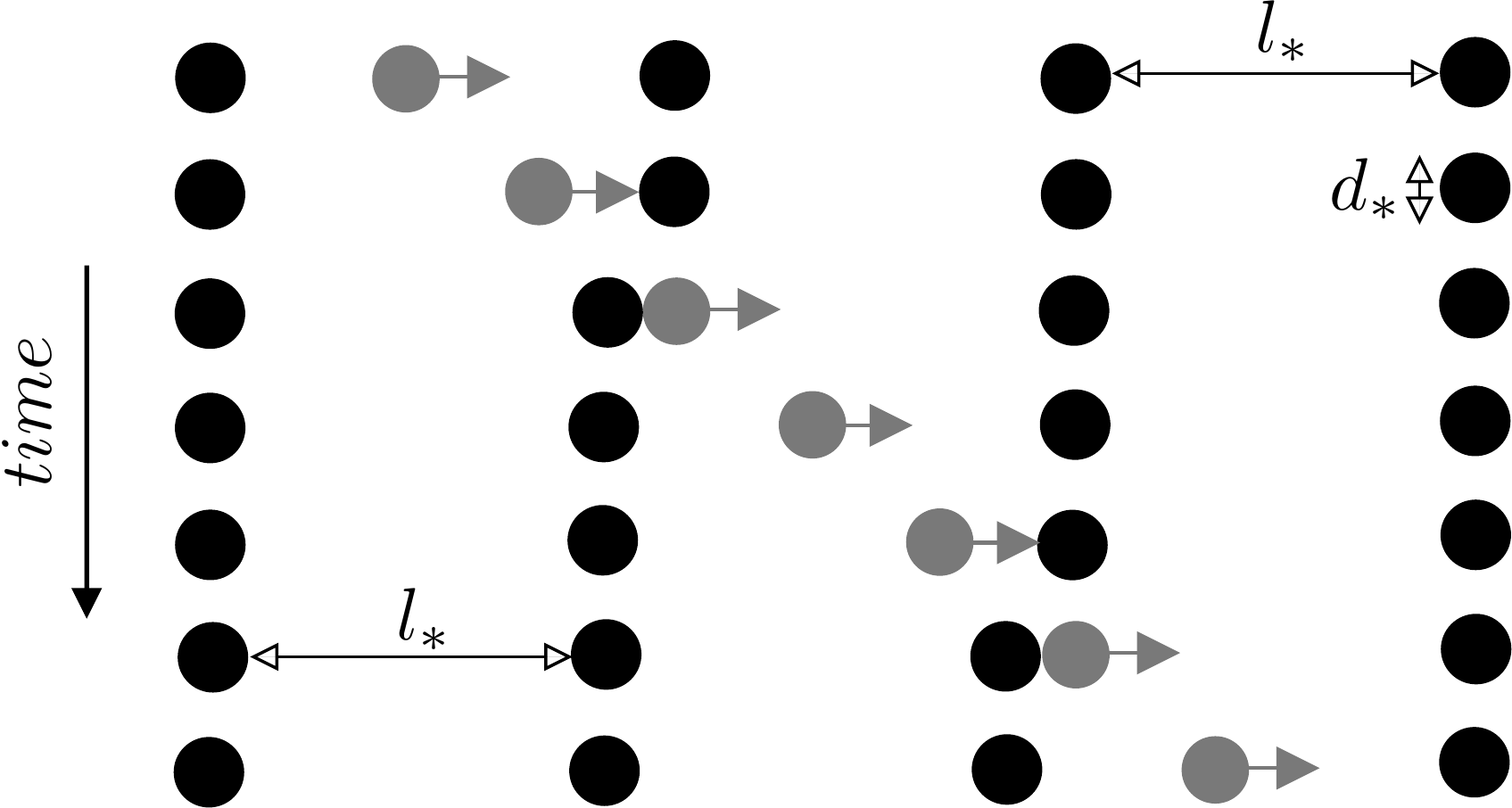} %
} %
\caption{ %
\emph{Left panel}. Potential $\Phi$ as in Assumption \ref{Ass.Pot}. \emph{Right panel}. Cartoon of a solitary wave in the hard-sphere model, which propagates by elastic collisions of a single active particle (gray) with a resting background (black). The corresponding distance and velocity profiles are affine transformations of the functions $R_\infty$ and $\hat{V}_\infty$ from \eqref{Eqn:LimitProfiles} and hence naturally related to the high-energy waves from Figures \ref{Fig:WavesDist} and \ref{Fig:WavesVel}.}
\label{Fig:Shock} %
\end{figure}
\end{result}
  The asymptotic analysis of the stability of high-energy waves is presented in \S\ref{sect:Stability} and
follows the general framework proposed by Friesecke and Pego   and uses the extension by Mizumachi to the energy space.   In particular, we study
the stability of waves in exponentially weighted spaces   as well as in the energy space   and in an orbital sense. The latter means that certain initial perturbations might shift or accelerate the wave, and this requires long-time adjustments of  the phase variable and the wave speed.   All other perturbations, however, remain small, fall behind the wave,  and decay exponentially in weighted norms centered around the moving bulk of the wave,    whereas solutions remain close in the energy norm.   We formulate our main results as follows and refer to \S\ref{sect:discussion} for more details concerning the proof strategy and the key arguments in the asymptotic analysis.
\begin{result}[nonlinear orbital stability in the sense of Friesecke and Pego]\label{res:Stab}
For any sufficiently large speed $\om_0$, the corresponding high-energy wave from Main Result \ref{res:Ex} is stable in the following sense,  where
\begin{align*}
U_{\om_0}\at{x}=\Bpair{R_{\om_0}\at{x+\tfrac12}}{V_{\om_0}\at{x}}
\end{align*}
denotes the wave profiles and where the positive constants $\eta_0$, $b_0$, $C_0$ depend on both $\Phi$ and $\om_0$:
Given FPUT initial data with
\begin{align*}
\bnorm{u_\cdot \at{0} - U_{\om_0}(. -\tau_0)}_{\ell^2\times\ell^2} \leq \sqrt\eta\,,\qquad \qquad  \bnorm{\mhexp{a\at{\cdot-\tau_0}} \bat{u_\cdot \at{0} - U_{\om_0}\at{\cdot-\tau_0} }}_{\ell^2\times\ell^2} \leq \eta
\end{align*}
for some $\tau_0\in\Rset$ and $\eta\in\oointerval{0}{\eta_0}$, there exist a unique wave speed $\om_{\infty}$ and a unique shift $\tau_\infty$ with
\begin{align*}
\abs{\om_0-\om_\infty}
+
\abs{\tau_0-\tau_\infty}+\bnorm{U_{\om_0}-U_{\om_\infty}}_{\fspaceL^2\times\fspaceL^2}\leq C_0\eta
\end{align*}
such that the solution $u=\pair{r}{v}$ to the FPUT chain \eqref{Eqn:FPU} satisfies the $\ell^2$-stability estimate
\begin{align*}
\bnorm{u_{\cdot}\at{t}-U_{\om_\infty}\at{\cdot-\om_\infty t-\tau_\infty}}_{\ell^2\times\ell^2}\leq C_0\sqrt\eta
\end{align*}
as well as the weighted decay  estimate
\begin{align*}
\bnorm{\mhexp{a\at{\cdot-\om_\infty t-\tau_\infty}}\bat{u_{\cdot}\at{t}-U_{\om_\infty}\at{\cdot-\om_\infty t-\tau_\infty}}}_{\ell^2\times\ell^2}\leq C_0\eta\mhexp{-b_0 t}
\end{align*}
for all $t\geq0$. Here, $\cdot$ stands for the integer variable $j\in\Zset$.
\end{result}

The next theorem gives stability for initial data in a larger space (the energy space $\ell^2\times\ell^2$) in the sense that solutions stay close to the set of traveling waves, but without   a decay estimate in the weighted norm and with less information about the shift $\tau(t)$, which in Main Result~\ref{res:Stab}  has the more explicit form $\tau(t)- \om_\infty t  \to \tau_\infty$ as $t \to \infty$.   In particular, the improved stability result also covers perturbation in form of small and near-sonic solitary waves, whose
exponential decay rate can be arbitrarily small.
\begin{result}[nonlinear orbital stability in the sense of Mizumachi]\label{res:StabEner}
For any sufficiently large speed $\om_0$, the corresponding high-energy wave from Main Result \ref{res:Ex} is stable within the energy space. More precisely, for every $\eta >0$ there exists a constant $\mu>0$ such that for given initial data with
\begin{align*}
\bnorm{u_\cdot \at{0} - U_{\om_0}(. -\tau_0)}_{\ell^2\times\ell^2} \leq \mu
\end{align*}
and  $\tau_0\in\Rset$, there exist constants $\om_\infty>0$, $\sigma \in (1, \om_\infty)$ and  $\fspaceC^1$-functions $\tau(t)$ and $\om(t)$
such that
\begin{align*}
\sup_{t \in \Rset}&\bnorm{u_{\cdot}\at{t}-U_{\om(t))}\at{\cdot-\tau(t)}}_{\ell^2\times\ell^2}\leq \eta\,,\qquad
&\lim_{t \to \infty} \bnorm{u_{\cdot}\at{t}-U_{\om_\infty}\at{\cdot-\tau(t)}}_{\ell^2(j \geq \sigma t) \times\ell^2(j \geq \sigma t)}=0
\end{align*}
and
\begin{align*}
\sup_{t \in \Rset} \left( |\om(t)-\omega_0| +|\dot{\tau}(t)- \omega_0|\right)  \leq C_0\mu\,,\qquad
\lim_{t \to \infty} \om(t)= \om_\infty\,, \qquad \lim_{t \to \infty} \dot{x}(t)= \om_\infty
\end{align*}
holds for the solution $u$ to \eqref{Eqn:FPU} and some constant $C_0$.
\end{result}

%
%
%
\subsection{Discussion and asymptotic proof strategy}
\label{sect:discussion}
%
%
The key observation for our asymptotic analysis of solitary waves with high energy is that the advance-delay-differential equations \eqref{Eqn:TWNonlIdentities.V1} can be replaced by an asymptotic ODE problem which dictates the fine structure of the distance profile near $x=0$ (\lq tip of the tent') and provides in turn all relevant properties of lattice waves in the limit $\om\to\infty$.   This idea will be explained in \S\ref{sect:Heuristics} on a heuristic level and has already been exploited \cite{HM15}. However, in \S\ref{sect:NonlWaves} we present a modified approximation strategy which is likewise based on the asymptotic shape ODE but has the following advantages.
\begin{enumerate}
\item
The high-energy waves in \cite{HM15} were constructed as solutions to a constrained optimization problem depending on a small parameter $\hat{\delta}$, and the wave speed $\om$ could not be described a priori but came out implicitly as a Lagrange multiplier. In particular, in the variational setting we were not able to prove that the relevant wave quantities depend smoothly on the wave speed. In the modified approach, however, $\om$ is a free parameter and we do not rely on optimization techniques. More precisely, we first identify in \S\ref{sect:AsympODE}, \S\ref{sect:Scaling}, and \S\ref{sect:AppForm} for any sufficiently large $\om$ an approximate solution and show afterwards in \S\ref{sect:AuxProb} and \S\ref{sect:ExAndUni} the existence of a unique exact   high-energy wave in a small neighborhood within a certain function space. In this way we do not only establish the smooth parameter dependence, see \S\ref{sect:Smoothness} and \S\ref{sect:Tails}, but can also provide accurate asymptotic formulas for the derivatives of the wave profiles with respect to $\om$. The latter are naturally related to the neutral Jordan modes in the linearized FPUT equation and feature prominently in the spectral analysis within \S\ref{sect:Stability}.
\par
We also mention that the numerical illustrations in Figure~\ref{Fig:WavesDist} and Figure~\ref{Fig:WavesVel} are computed by a straight forward implementation of the  constrained optimization problem. We refer to \cite{Her10}  and the caption of Figure \ref{Fig:WavesVel} for the details and recall that the uniqueness results in \S\ref{sect:ExAndUni} and \cite{HM17} guarantee that both the variational and the non-variational approach provide for large $\om$ the same family of high energy waves.
\item
In \cite{HM15}, the derivation of the shape ODE for the fine structure of the distance profile near $x{=}0$ (`tip scaling') was followed by asymptotic arguments to describe the jump like behavior of the velocity profiles near $x{=}\pm1/2$ (`transition scaling') and local properties of the distance profile near $x{=}{\pm}1$ (`base scaling'). This strategy   involved   explicit matching conditions and required several lengthy computations. The revised approach is more robust, and replaces the study of the advance-delay-differential equations \eqref{Eqn:TWNonlIdentities.V1} by the analysis of  equivalent integral equations.   More precisely,  integrating the traveling wave equations \eqref{Eqn:TWNonlIdentities.V1} we get
\begin{align}
\label{Eqn:TWNonlIdentities.V2}
\omega V_\omega+\chi\ast\Phi^\prime\at{R_\omega}=0\,,\qquad
\omega R_\omega+\chi\ast{V_\omega}=0,
\end{align}
where the convolution is with the  indicator function $\chi$ from \eqref{Eqn:LimitProfiles}. This system can also be written as
\begin{align}
\label{Eqn:NonlFPDist}
R_\omega=\chi\ast\chi \ast \bat{\omega^{-2}\Phi^\prime\at{R_\omega}},\qquad V_\om=-\chi\ast\bat{\om^{-1}\Phi^\prime\at{R_\omega}}
\end{align}
and turns out to be very useful in the asymptotic analysis. In particular, thanks to \eqref{Eqn:NonlFPDist} we do not need to introduce the transition and the base scaling anymore and satisfy the hidden matching conditions implicitly, see the discussion in \S\ref{sect:Heuristics}.
\item
In the present paper we cover a broader class of potentials   and distinguishing between the parameters $m$ and $k$ we gain   a better understanding of the different contributions to the error terms. In Assumption~\ref{Ass.Pot} we require that both $m>1$ and $k>1$ because otherwise some of the error estimates  derived below would no longer be small but could attain values of order $\DO{1}$ or larger. It is not clear, at least to the authors, whether high energy waves are unstable for $0<m\leq 1$ or $0<k\leq1$, or whether our approximation and stability results are still valid but necessitate a more detailed analysis in the proofs. To ease the notation we also require that $k\neq m$. This seems to be surprising at a first glance but gets more plausible if we study the class
\begin{align*}
\Phi\at{r}=\frac{1}{m\at{m+1}\at{1+r}^{m}}-\frac{1}{\at{m-k}\at{m+1}\at{1+r}^{m-k}}+\frac{k}{m\at{m-k}\at{m+1}}\,,\qquad m\neq k\,,
\end{align*}
which includes both \eqref{Eqn:ExPot1} and \eqref{Eqn:ExPot2} as special cases. In fact, for those potentials the next-to-leading term in the expansion for $r\to-1$ is singular and regular for $k<m$ and $k>m$, respectively, but not defined for $k=m$. The criticality of $k=m$ shows up several times in our asymptotic analysis and manifests in algebraic error estimates involving the exponent $l=\min\{k,\,m\}$. We emphasize that all asymptotic formulas can also be established for $k=m$ but entail logarithmic corrections in the error bounds, see the discussion after Lemma~\ref{Lem:DefBreveR} for more details.
\end{enumerate}
The non-asymptotic part of the    Friesecke-Pego theory   is reviewed in \S\ref{sect:FPCrit} and guarantees that the nonlinear orbital stability in the sense of Main Result \ref{res:Stab} depends only on the spectral properties of the linearized FPU chain. More precisely, inserting the formal ansatz
\begin{align*}
r_j\at{t}=R_\omega\at{j+\tfrac12-\om t}+S\pair{t}{j+\tfrac12-\om t }\,,\qquad
v_j\at{t}=V_\omega\at{j-\om t}+W\pair{t}{j-\om t }.
\end{align*}
into \eqref{Eqn:FPU}, we obtain after linearization with respect to the perturbations $S$ and $W$ the  dynamical equation
\begin{align*}
\partial _t \pair{S}{W}=\om \calL_\om\pair{S}{W}\,,
\end{align*}
where we scaled the right hand side by $\omega$ for convenience. The operator $\calL_\om$ acts on pairs $\pair{S}{W}$ of spatial functions via
\begin{align}
\label{Eqn:LinFPUOp}
\calL_\omega\pair{S}{W}
=
\Bpair{
\partial_x S+\om^{-1}\nabla W}{\partial_x W+\om^{-1}\nabla \Phi^{\prime\prime}\at{R_\omega}S}
\end{align}
and depends explicitly   on   the distance profile $R_\om$ of the underlying traveling wave. It has been observed in \cite{FP02,FP04a} that the spectrum of the operator $\calL_\om$ is $2\pi\iu$-periodic due to the spatial discreteness of \eqref{Eqn:FPU0} and that the existence of a one-parameter family of solitary waves dictates that the kernel is a two-dimensional Jordan block, where   the   proper and cyclic eigenfunctions represent the neutral modes related to shifts and accelerations of the wave, respectively. Moreover, the essential part of this spectrum can be computed explicitly and has strictly negative real part in exponentially weighted spaces that penalizes perturbations in front of the wave. These properties hold under fairly mild assumption on the interaction potential $\Phi$ and are not restricted to any asymptotic regime, see the discussion in \S\ref{sect:expspaces}.
\par
The main analytical task is to show that the operator $\calL_\omega$ does not admit
any unstable point spectrum in the symplectically orthogonal complement of the neutral modes, where the underlying symplectic product is provided by the Hamiltonian structure of \eqref{Eqn:FPU0}, see again \S\ref{sect:FPCrit} for the details. This crucial spectral property has been   established   in \cite{FP04b} for near sonic waves by rigorously relating the operator $\calL_\omega$  to the well-studied properties of solitary waves in the KdV equation and the corresponding linear differential operator.   In this paper we study the case of large $\omega$ for a wide class of interaction potentials but emphasize that the stability problem remains open for moderate wave speeds since there is no spectral theory available for linear advance-delay-differential operators with non-constant coefficients.
\par
Our proof strategy differs essentially from the asymptotic approach in \cite{FP04b} because in our case $\calL_\om$ is not related to the long-wave length limit and can hence no longer be regarded as a perturbation of the linear differential operator from the KdV theory. Instead we show $\at{i}$ that any appropriately rescaled eigenfunction solves locally the linearized shape ODE from \S\ref{sect:AsympODE} up to small error terms, and $\at{ii}$ that the global behavior of any eigenfunction is completely determined by its local properties, the desired spatial decay, and certain asymptotic matching conditions. The underlying ideas are explained in \S\ref{sect:StabilityEstimates} on   a heuristic ODE level while the corresponding rigorous results in \S\ref{sect:StabilityODEs} and \S\ref{sect:PointSpectrum}
concern an equivalent convolution equation with implicitly encoded matching conditions.
Our main results on    orbital stability are formulated in Theorem  \ref{Thm:NoOtherEigenvalues}    and  Corollary \ref{Cor:NonlStability} and
disprove   for all sufficiently large $\om$
the existence of unstable eigenfunctions  in the aforementioned symplectic complement   within suitable exponentially weighted spaces.  This implies the asymptotic stability in exponentially weighted spaces    but also in the energy space due to the additional arguments by Mizumachi.
%
%
\section{Family of nonlinear high-energy waves}
\label{sect:NonlWaves} %
%
%
In this section we prove the existence of a smooth family of nonlinear high-energy waves and derive explicit approximation formulas for the wave profiles as well as their derivatives with respect to the wave speed. For this analysis it is convenient to replace the large wave speed $\om$ by a corresponding small quantity. In this paper we choose
\begin{align}
\label{Eqn:DefDelta}
\delta_\omega:=\frac{1}{\omega^{2/m}}
\end{align}
because this quantity determines, as explained in \S\ref{sect:Heuristics}, the spatial scale of the asymptotic shape ODE and behaves asymptotically like $\hat{\delta}$, the small parameter from the variational approach in \cite{HM15}.

%
\subsection{Heuristic arguments and overview}
\label{sect:Heuristics}
%
As preparation for the rigorous asymptotic analysis derived below, we start with an informal overview on the high energy limit and explain on a heuristic level, why the solution $\pair{R_\omega}{V_\omega}$ to the nonlinear advance-delay-differential equations \eqref{Eqn:TWNonlIdentities.V1} is asymptotically determined by an ODE initial value problem. Our discussion is, as already mention in \S\ref{sect:Intro}, based on the arguments from \cite{HM15} but simplifies some computations since the crucial scaling laws and matching conditions are inferred from the nonlinear integral formulation in \eqref{Eqn:TWNonlIdentities.V2}.
\par
{\bf Distance profile near the origin.} The first key observation --- at least for unimodal and even wave profiles as in Figure~\ref{Fig:WavesDist} that attain their minimum at $x=0$ --- is the following.   If   $R\at{0}$ is close to $-1$,  the algebraic singularity of $\Phi$ implies
\begin{align}
\label{Eqn:Overview.1}
\abs{\Phi^\prime\bat{R_\omega\at{x\pm 1}}}\ll\abs{\Phi^\prime\bat{R_\omega\at{x}}}\qquad \text{for all}\quad \abs{x}<1/2\,,
\end{align}
so both the advance and the delay term in the second-order equation \eqref{Eqn:NonlADDE2Order} can be neglected. This gives the
approximate ODE
\begin{align*}
\om^2 R_\omega^{\prime\prime}\at{x}\approx -2 \Phi^\prime\bat{R_\omega\at{x}}\qquad \text{for}\quad \abs{x}<1/2
\end{align*}
for the `tip of the tent', but this equation is still singular as it involves very large quantities.
\par
{\bf {Asymptotic shape ODE.}}
In the next step we rescale both the spatial variable $x$ and the amplitude of the distance profile by the ansatz
\begin{align}
\label{Eqn:DefTildeY}
R_\omega\at{x}= -1+\delta_\omega\al_\omega\tilde{Y}_\omega\bat{\delta_\omega^{-1}\beta_\omega^{-1} x}\,,
\end{align}
with
\begin{align}
\label{Eqn:DefBeta}
\be_\omega:=\sqrt{\al_\omega^{m+2}}\,,
\end{align}
where the scaling parameter $\al_\omega$ will be determined below. Combining this with the properties of $\Phi$ from Assumption~\ref{Ass.Pot} and using
\begin{align}
\label{Eqn:SpaceScalign}
\tilde{x}:=\frac{x}{\delta_\omega\beta_\omega}
\end{align}
we get
\begin{align}
\label{Eqn:Overview.7}
\om^2 R^{\prime\prime}\at{x}=\frac{1}{\delta_\omega^{m+1}\al_\omega^{m+1}}\tilde{Y}_\omega^{\prime\prime}\at{\tilde{x}}\,,\qquad \qquad
\Phi^\prime\bat{R\at{x}}\approx -\frac{1}{m+1}\cdot\frac{1}{\delta_\omega^{m+1}\al_\omega^{m+1}\tilde{Y}_\omega\at{\tilde{x}}^{m+1}}
\end{align}
and arrive at the nonsingular ODE
\begin{align}
\label{Eqn:Overview.2}
\tilde{Y}_\omega^{\prime\prime}\at{\tilde{x}}\approx \frac{2}{m+1}\cdot\frac{1}{\tilde{Y}_\omega\at{\tilde{x}}^{m+1}}\,,
\end{align}
where we can specify initial conditions. Restricting to even profiles we prescribe
\begin{align}
\label{Eqn:Overview.3}
\tilde{Y}_\omega\at{0}\approx 1\,,\qquad \tilde{Y}_\omega^\prime\at{0}=0
\end{align}
and infer
\begin{align}
\notag
\al_\omega\approx \frac{{1+R_\omega\at{0}}}{\delta_\omega}\,.
\end{align}
as  heuristic rule from \eqref{Eqn:DefTildeY}.
\par
{\bf Further scaling law.}
Our next goal is to compute a more explicit formula for $\al_\omega$. The integral formulation \eqref{Eqn:NonlFPDist} implies
\begin{align*}
\om^2 R_\omega\at{0}=2\int\limits_{0}^{1/2}\at{1-x}\Phi^{\prime}\at{R_\omega\at{x}}\dint{x}
\end{align*}
thanks to $\at{\chi\ast\chi}\at{x}=\max\{0, 1-\abs{x}\}$ and the evenness of $R_\omega$. In view of \eqref{Eqn:DefTildeY} and \eqref{Eqn:Overview.7} we obtain
\begin{align}
\label{Eqn:Overview.5}
-1+\delta_\omega\al_\omega\tilde{Y}_\omega\at{0}\approx 2\int\limits_{0}^{\xi_\omega}\at{2\xi_\omega-\tilde{x}}\at{ -\frac{\delta_\omega \al_\omega}{\at{m+1}\tilde{Y}_\omega\at{\tilde{x}}^{m+1}}}\dint{\tilde{x}}
\end{align}
with
\begin{align}
\label{Eqn:DefXi}
\xi_\omega:=\frac{1}{2\delta_\omega\beta_\omega}.
\end{align}
Inserting the asymptotic ODE \eqref{Eqn:Overview.2} into \eqref{Eqn:Overview.5} provides
\begin{align*}
-1+\delta_\omega\al_\omega\tilde{Y}_\omega\bat{0}\approx -\delta_\omega\al_\omega\int\limits_{0}^{\xi_\omega}\at{2\xi_\omega-\tilde{x}} \tilde{Y}_\omega^{\prime\prime}\at{\tilde{x}}\dint{\tilde{x}}=
-2\delta_\omega\al_\omega\xi_\omega \tilde{Y}_\omega^{\prime}\bat{\xi_\omega}+
\delta_\omega\al_\omega\bat{\tilde{Y}_\omega^{\flat}\bat{\xi_\omega}-\tilde{Y}_\omega^{\flat}\bat{0}}\
\end{align*}
and hence
\begin{align*}
-1\approx-\frac{\al_\omega}{\be_\omega}\tilde{Y}_\omega^{\prime}\bat{\xi_\omega}+\delta_\omega\al_\omega\tilde{Y}_\omega^{\flat}\bat{\xi_\omega}
\end{align*}
due to $\tilde{Y}_\omega\bat{0}=-\tilde{Y}_\omega^\flat\bat{0}$, where the differential operator $^\flat$ is defined as
\begin{align}
\label{Eqn:FlatOperator}
\tilde{F}^\flat\at{\tilde{x}}:=
\tilde{x}\tilde{F}^\prime\at{\tilde{x}}-\tilde{F}\at{\tilde{x}}
\end{align}
and will be used frequently. By elementary ODE arguments --- see Lemma~\ref{Lem:AsympODE.Props} for precise statements --- we compute both $\tilde{Y}_\omega^{\prime}\bat{\xi_\omega}$ and  $\tilde{Y}_\omega^\flat\bat{\xi_\omega}$  and obtain
\begin{align}
\label{Eqn:Overview.6}
\al_\omega \approx \tilde{Y}_\omega^{\prime}\at{\infty}^{2/m}\approx
\at{\frac{4}{m\at{m+1}}}^{1/m}\,,
\end{align}
where we used that $\delta_\omega$ and $\xi_\omega$ are asymptotically small and large, respectively, as $\om\to\infty$. A precise definition of $\al_\omega$ --- and hence also of $\beta_\omega$ and $\xi_\omega$ --- will be given below in Lemma~\ref{Lem:DefAlpha.Eqn1}.
\par
{\bf Discussion.}
  The   formal asymptotic analysis from above can be summarized as follows. Rescaling the distance profile $R_\omega$ by \eqref{Eqn:DefTildeY}, where $\al_\omega$, $\be_\omega$,  and  $\delta_\omega$ are given as in \eqref{Eqn:DefDelta}+\eqref{Eqn:DefBeta}+\eqref{Eqn:Overview.6}, the rescaled distance profile $\tilde{Y}_\omega$ satisfies on the interval $\oointerval{-\xi_\omega}{+\xi_\omega}$ the ODE \eqref{Eqn:Overview.2} along with the initial values \eqref{Eqn:Overview.3} up to small error terms. Moreover, all these quantities are on the approximate level completely determined by $\om$ and $m$, and give rise to explicit asymptotic formulas for $R_\omega$ on the interval $x\in\oointerval{-1/2}{+1/2}$.
\par
The final argument is that the local behavior of the distance profile determines also the global one because \eqref{Eqn:Overview.1} can also be used to justify the approximate differential equations
\begin{align}
\label{Eqn:Overview.9}
\om^2 R_\omega^{\prime\prime}\at{x}  \approx \Phi^\prime\at{R_\omega\at{x-1}}   \approx -\tfrac12\omega^2
 R_\omega^{\prime\prime}\at{x-1}\qquad \text{for}\quad x\in\oointerval{+1/2}{+3/2}
\end{align}
for the `base of the tent' and
\begin{align*}
\om^2 R_\omega^{\prime\prime}\at{x}\approx 0 \qquad \text{for}\quad  x\in\oointerval{+3/2}{\infty}
\end{align*}
for the  \lq tail of the tent', which by evenness hold also for $x<-1/2$. In principle, these identities allow us --- as in \cite{HM15} --- to recover the entire distance profile provided that the constants of integration will be determined such that both $R_\omega$ and $R_\omega^\prime$ are continuous at $x=1/2$ and $x=3/2$, and vanish at $x=\infty$. For our asymptotic proofs, however, it is    more convenient   to rely on the approximation
\begin{align}
\label{Eqn:Overview.8}
 R_\omega\approx \chi\ast\chi\ast \at{\chi \frac{\Phi^\prime\at{R_\omega}}{\om^2}}\approx -\tfrac12  \chi\ast\chi\ast \bat{\chi R_\omega^{\prime\prime}}\,,
\end{align}
where the indicator function $\chi$ of the interval $\ccinterval{-1/2}{+1/2}$ serves as both convolution kernel and cut-off function.   This   formula bridges neatly between local and global approximations since the right hand side represents a $\fspaceC^1$-function on $\Rset$ which depends only on the behavior of $R_\omega$ on $\oointerval{-1/2}{+1/2}$. In particular, \eqref{Eqn:Overview.8} encodes the above mentioned matching conditions in a concise and implicit way.
\par
We finally mention that the even velocity profile $V_\omega$ can be
approximated either piecewise via the approximate differential equation
\begin{align*}
\om V_\omega^\prime\at{x}\approx \Phi^\prime\bat{R_\omega\at{x-1/2}}\approx \om^2R^{\prime\prime}_\omega\at{x-1/2}\quad \text{for}\quad x\in\ccinterval{0}{1}\,,\qquad \om V_\omega^\prime\at{x}\approx 0\quad \text{for}\quad x\in\cointerval{1}{\infty}
\end{align*}
or better globally by $\om V_\omega \approx \tfrac12\chi\ast\at{\chi \Phi^\prime\at{R_\omega}}$.
\par
{\bf Outlook.} In what follows we employ the outlined ideas and prove that the asymptotic ODE problem corresponding to \eqref{Eqn:Overview.2} and \eqref{Eqn:Overview.3} determines in fact all asymptotic properties of lattice waves in the limit $\om\to\infty$. Of course, the analytical details will be more involved   than   the above non-rigorous computations  and require explicit bounds for the error terms. In principle we could exploit ODE perturbation techniques as in \cite{HM15}, but here we proceed differently and in a more concise way. We first identify in \S\ref{sect:AppForm} almost explicit formulas for the approximate solutions and conclude afterwards in \S\ref{sect:ExAndUni} the existence of a nearby exact solution by a fixed point argument in a suitable function space. From a technical point of view, the key ingredients   to   our approach are $\at{i}$ the properties of the solution to the shape ODE studied in \S\ref{sect:AsympODE}, $\at{ii}$ the precise definition of the scaling parameters in  \S\ref{sect:Scaling}, and $\at{iii}$   the   careful investigation of a nonlocal but linear auxiliary problem, which will be introduced in \S\ref{sect:AuxProb} and may be regarded as a simplified and scaled version of the linearized traveling wave equation \eqref{Eqn:NonlFPDist}. In this way we can also establish a local uniqueness result for even high-energy waves; this question remained open in \cite{HM15} and could only be answered in \cite{HM17}. Finally, in \S\ref{sect:Smoothness} and \S\ref{sect:Tails} we quantify the smooth parameter dependence and estimate the tail decay, respectively.
\par
%
%
\subsection{Asymptotic shape ODE}
\label{sect:AsympODE}
%
In this section   we   provide detail information on the solution to the ODE initial value problem
\begin{align}
\label{Lem:AsympODE.Props.Eqn1}
\bar{Y}^{\prime\prime}\at{\tilde{x}}=\frac{2}{m+1}\frac{1}{\bar{Y}\at{\tilde{x}}^{m+1}}\,,\qquad
\bar{Y}\at{0}=1\,,\qquad \bar{Y}^\prime\at{0}=0\,,
\end{align}
which determines --- via $\tilde{Y}_\omega\approx \bar{Y}$ and as explained in \S\ref{sect:Heuristics} ---  the asymptotic shape of the lattice waves in the high-energy limit $\om\to\infty$. We also study the linearized ODE
\begin{align}
\label{Lem:LinODE.Props.Eqn1}
\bar{T}^{\prime\prime}\at{\tilde{x}}=-\frac{2}{\bar{Y}\at{\tilde{x}}^{m+2}}
\bar{T}\at{\tilde{x}},
\end{align}
because the two linearly independent solutions are important building blocks for the subsequent asymptotic and spectral analysis. Notice that the function $\bar{Y}$ is uniquely determined by the parameter $m$ from Assumption~\ref{Ass.Pot}, and see Figure~\ref{Fig:ODESolutions} for numerical examples.
\begin{figure}[ht!] %
\centering{ %
\includegraphics[width=0.9\textwidth]{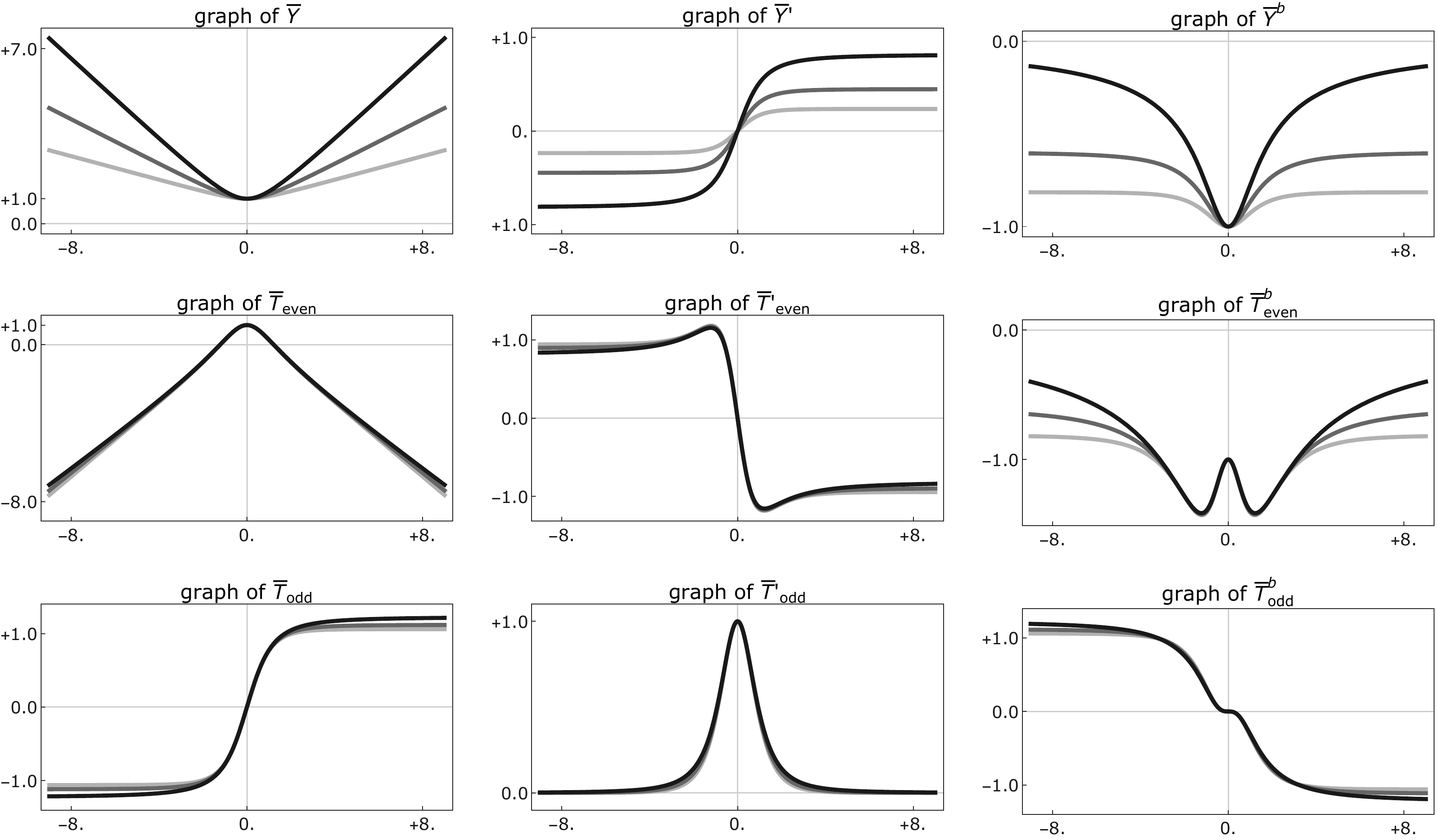} %
} %
\caption{
\emph{Top row}: Numerical solution of the nonlinear shape ODE \eqref{Lem:AsympODE.Props.Eqn1}: graphs of $\bar{Y}$ (left), $\bar{Y}^\prime$ (center), and  $\bar{Y}^\flat$ (right) for $m=2$ (black), $m=4$ (dark gray), and $m=8$ (light gray).
\emph{Centre row, bottom row}: %
The respective plots for the even solution $\ol{T}_\even$ and the odd solution
$\ol{T}_\odd$ of the linearized shape ODE \eqref{Lem:LinODE.Props.Eqn1}.%
}\label{Fig:ODESolutions} %
\end{figure} %
\begin{lemma}[asymptotic shape ODE]
\label{Lem:AsympODE.Props}
The unique solution to the ODE initial value problem \eqref{Lem:AsympODE.Props.Eqn1} is even, convex, strictly positive and grows asymptotically linear. Moreover, we have
\begin{align*}
\nat{\bar{Y}^{\prime}}^{\prime\prime}\at{\tilde{x}}=-\frac{2\bar{Y}^\prime\at{\tilde{x}}}{\bar{Y}\at{\tilde{x}}^{m+2}}\,,\qquad
\nat{\bar{Y}^\flat}^{\prime\prime}\at{\tilde{x}}=-\frac{2\bar{Y}^\flat\at{\tilde{x}}}{\bar{Y}\at{\tilde{x}}^{m+2}}-\frac{2m}{\at{m+1}\bar{Y}\at{\tilde{x}}^{m+1}}\,,
\end{align*}
where both $\bar{Y}^\prime$ and $\bar{Y}^\flat$ are asymptotically constant with
\begin{align}
\notag
\lim_{\tilde{x}\to\infty}\bar{Y}^\prime\at{\tilde{x}}=\frac{2}{\sqrt{m\at{m+1}}}\,,\qquad
\lim_{\tilde{x}\to\infty}\bar{Y}^\flat\at{\tilde{x}}= \gamma
\end{align}
and
\begin{align*}
\sup_{\tilde{x}\in\Rset}\at{
 \abs{\tilde{x}}^{m} \babs{\bar{Y}^\prime\at{\tilde{x}}-\bar{Y}^\prime\at{\infty}}
+ \abs{\tilde{x}}^{m-1} \babs{\bar{Y}^\flat\at{\tilde{x}}-\bar{Y}^\flat\at{\infty}}}\leq C\,.
\end{align*}
Here, the constants $  \gamma $, $C$ depend only on $m$ and the differential operator $^\flat$ has been defined in \eqref{Eqn:FlatOperator}.
\end{lemma}
\begin{proof}
A simple phase plane analysis shows that $\bar{Y}$ is even and that both $\bar{Y}$ and $\bar{Y}^\prime$ are increasing for $\tilde{x}>0$. The equations for $\bar{Y}^{\prime}$ and $\bar{Y}^{\flat}$ follow by differentiating \eqref{Lem:AsympODE.Props.Eqn1}, and   using   direct computations we verify the conservation law
\begin{align*}
\frac12\bar{Y}^\prime\at{\tilde{x}}^2+\frac{2}{m\at{m+1}}\frac{1}{\bar{Y}\at{\tilde{x}}^{m}}=\frac{2}{m\at{m+1}}.
\end{align*}
This energy law finally implies the asymptotic conditions for $\bar{Y}^{\prime}$ and   $\bar{Y}^{\flat}$ as well as the claimed error estimates.
\end{proof}
\begin{lemma}[linearized shape ODE]
\label{Lem:LinODE.Props}
The solution space to the linear ODE \eqref{Lem:LinODE.Props.Eqn1} is spanned by  an even function $\bar{T}_\even$ and an odd function $\bar{T}_\odd$, which are defined by
\begin{align*}
\bar{T}_\even\at{\tilde{x}}:=-\frac{m}{2} \bar{Y}\at{\tilde{x}}-\at{\frac{m}{2}+1}\bar{Y}^\flat\at{\tilde{x}}\,,\qquad
\bar{T}_\odd\at{\tilde{x}}:=\frac{\at{m+1}}{2}\bar{Y}^\prime\at{\tilde{x}}
\end{align*}
and comply with  the normalization condition
\begin{align*}
\bar{T}_\even\at{0}=\bar{T}_\odd^\prime\at{0}=1\,.
\end{align*}
Moreover,  $\bar{T}_\even$ and $\bar{T}_\odd$  are asymptotically affine and constant, respectively, with
\begin{align*}
\lim_{\tilde{x}\to\infty}\bar{T}_\even^\prime\at{\tilde{x}}
=-\sqrt{\frac{m}{m+1}}\,,\qquad
\lim_{\tilde{x}\to\infty}\bar{T}_\even^\flat\at{\tilde{x}}
=  \gamma \,,\qquad \lim_{\tilde{x}\to\infty}\bar{T}_\odd\at{\tilde{x}}
=\sqrt{\frac{m+1}{m}}
\end{align*}
and satisfy the decay estimates
\begin{align*}
\sup_{\tilde{x}\in\Rset}
\at{\abs{\tilde{x}}^{m} \babs{\bar{T}_\even^\prime\at{\tilde{x}}-\bar{T}_\even^\prime\at{\infty}}
+ \abs{\tilde{x}}^{m-1} \babs{\bar{T}_\even^\flat\at{\tilde{x}}-\bar{T}_\even^\flat\at{\infty}}}\leq C
\end{align*}
as well as
\begin{align*}
\sup_{\tilde{x}\in\Rset}\at{
 \abs{\tilde{x}}^{m+1}
\babs{\bar{T}_\odd^\prime\at{\tilde{x}}}
+ \abs{\tilde{x}}^{m} \babs{\bar{T}_\odd^\flat\at{\tilde{x}}-\bar{T}_\odd^\flat\at{\infty}}}\leq C
\end{align*}
for some constant $C$ depending on $m$ and   $\gamma$ as in Lemma \ref{Lem:AsympODE.Props}.
\end{lemma}
\begin{proof}
  We   compute that $\bar{T}_\even$ and $\bar{T}_\odd$ solve \eqref{Lem:LinODE.Props.Eqn1}.
All other assertions follow from Lemma~\ref{Lem:AsympODE.Props}.
\end{proof}
We finally mention explicit expressions for $m=2$, namely
\begin{align*}
\bar{Y}\at{\tilde{x}}=\frac{\sqrt{9 + 6 x^2}}{3}\,,\qquad
\bar{T}_\even\at{\tilde{x}}=  \frac{3-2\tilde{x}^2}{\sqrt{9 + 6 x^2}}  \,,\qquad
\bar{T}_\odd\at{\tilde{x}}=\frac{3\tilde{x}}{\sqrt{9 + 6 x^2}}\,,
\end{align*}
which imply the non-generic property $\bar{Y}^\flat\at{\infty}=0$.
%
%
%
\subsection{Scaling parameters}\label{sect:Scaling}
%
%
Motivated by the heuristic discussion in \S\ref{sect:Heuristics} we aim to study the wave profiles $R_\omega$ and $V_\omega$ as functions of the rescaled space variable $\tilde{x}$ from \eqref{Eqn:SpaceScalign}  in order to resolve the fine structure of the distance profile near $x=0$.   In a first step we specify   the $\om$-dependence of $\al_\omega$, $\beta_\omega$, and $\xi_\omega$ by the following result, which does not alter the algebraic relations in \eqref{Eqn:DefBeta} and \eqref{Eqn:DefXi} but provides a precise definition for $\al_\omega$ in terms of $\bar{Y}$. This value is asymptotically consistent with \eqref{Eqn:Overview.6} and ensures that the approximate distance profile constructed in Lemma~\ref{Lem:DefBreveR} below has sufficiently nice properties.
\begin{figure}[ht!] %
\centering{ %
\includegraphics[width=0.8\textwidth]{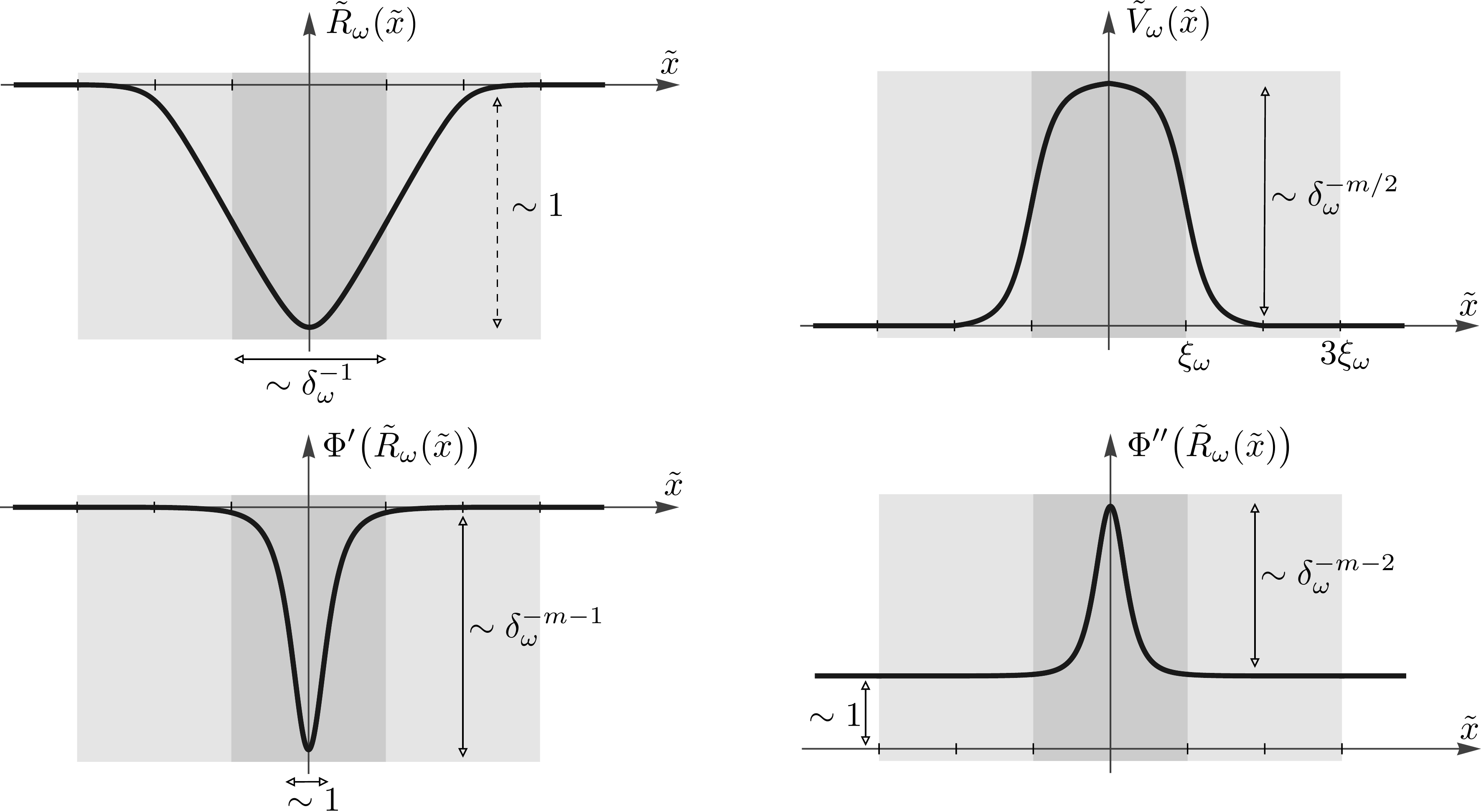} %
} %
\caption{\emph{Top row}. The scaled distance profile $\tilde{R}_\omega$ and the corresponding velocity profile $\tilde{V}_\omega$ from \eqref{Eqn:ScaledProfiles}. The shaded region indicate the intervals $I_\omega$ and   $3I_\omega\setminus I_\omega$   related to the `tip of the tent' and the `base of the tent', respectively. The explicit approximations $\breve{R}_\delta$ and $\breve{V}_\delta$ are constructed in Lemma~\ref{Lem:DefBreveR} and Lemma~\ref{Lem:ApproxVelProfile}, respectively, and vanish identically outside $3I_\omega$ and $2I_\omega$, respectively. \emph{Bottom row}. The resulting profiles after applying $\Phi^\prime$ and $\Phi^{\prime\prime}$, see Assumption~\ref{Ass.Pot}.
} %
\label{Fig:Profiles} %
\end{figure} %
\begin{lemma}[parameters $\al_\omega$, $\be_\omega$, and $\xi_\omega$]
\label{Lem:DefAlpha}
For all sufficiently large $\omega$, there exists a unique $\xi_\omega\geq1$ such that
\begin{align}
\label{Lem:DefAlpha.Eqn1}
\delta_\omega^{\frac{m}{m+2}}\at{\xi_\omega\bar{Y}^{\prime}\nat{\xi_\omega}+ \bar{Y}\nat{\xi_\omega}}= \bat{2\xi_\omega}^{\frac{2}{m+2}}\,.
\end{align}
Moreover, we have
\begin{align}
\label{Lem:DefAlpha.Eqn2}
\be_\omega^{\frac{2}{m+2}}=\al_\omega :=\bat{2\delta_\omega\xi_\omega}^{-\frac{2}{m+2}} \quad \xrightarrow{\;\;\omega\to\infty\;\;} \quad \at{\frac{4}{m\at{m+1}}}^{1/m}
\end{align}
as well as
\begin{align}
\label{Lem:DefAlpha.Eqn3}
\delta_\omega\al_\omega \xi_\omega\bar{Y}^{\prime}\nat{\xi_\omega} \quad \xrightarrow{\;\;\omega\to\infty\;\;} \quad \tfrac12\,,\qquad
\delta_\omega\al_\omega \bar{Y}\nat{\xi_\omega} \quad\xrightarrow{\;\;\omega\to\infty\;\;} \quad \tfrac12\,,
\end{align}
and $\xi_\omega$ depends smoothly on $\omega$.
\end{lemma}
\begin{proof}
The left hand side in \eqref{Lem:DefAlpha.Eqn1} is a strictly increasing and asymptotically linear function in its argument $\xi_\omega\geq0$, attains a positive minimum at the origin, and involves the small prefactor $\delta_\omega^{2/\at{m+1}}$. On the other hand, the right hand side is concave, strictly increasing, and independent of $\delta_\omega$. It also vanishes at the origin and grows sublinearly. We therefore conclude for all sufficiently large $\omega$ that there exists precisely two solutions $\xi_{1,\omega}$ and $\xi_{2,\omega}$ to \eqref{Lem:DefAlpha.Eqn1}, which depend smoothly on $\omega$ and satisfy
\begin{align*}
0<\xi_{1,\omega}<\xi_{2,\omega}<\infty\,,\qquad \lim_{\omega\to\infty}\xi_{1,\omega}=0\,,\qquad
\lim_{\omega\to\infty}\xi_{2,\omega}=\infty\,.
\end{align*}
We set $\xi_\omega:=\xi_{2,\omega}$ and in view of the asymptotic properties of $\bar{Y}$ from Lemma~\ref{Lem:AsympODE.Props} we reformulate \eqref{Lem:DefAlpha.Eqn1} as
\begin{align*}
2\xi_\omega \Bat{\bar{Y}^\prime\at{\infty}+\DO{\xi_\omega^{-m}}}-\Bat{\bar{Y}^\flat\at{\infty}+\DO{\xi_\omega^{-m+1}}}=\bat{2\xi_\omega}^{\frac{2}{m+2}}\delta_\omega^{-\frac{m}{m+2}}\,.
\end{align*}
This implies
\begin{align*}
\at{2\xi_\omega\delta_\omega}^{\frac{m}{m+2}}=\at{\bar{Y}^\prime\at{\infty}-   \at{2\xi_\omega}^{-1}
\bar{Y}^\flat\at{\infty}+\DO{\xi_\omega^{-m}}}^{-1}
\end{align*}
and yields the desired results by direct computations and thanks to the asymptotic formulas from Lemma~\ref{Lem:AsympODE.Props}.
\end{proof}
Our subsequent analysis relies on the ansatz
\begin{align}
\label{Eqn:ScaledProfiles}
\tilde{R}_\omega\at{\tilde{x}}:= R_\omega\bat{\delta_\omega\beta_\omega\tilde{x}}\,,\qquad\qquad
\tilde{V}_\omega\at{\tilde{x}}:= V_\omega\bat{\delta_\omega\beta_\omega\tilde{x}}
\end{align}
which scales the space variable but not the amplitudes, and the integral equations in \eqref{Eqn:NonlFPDist} transform into
\begin{align}
\label{Eqn:ScaledTWEqnForDistances}
\tilde{R}_\omega = \tilde{\chi}_\omega\ast\tilde{\chi}_\omega\ast \Bat{\delta_\omega^{m+2}\al_\omega^{m+2}\Phi^\prime\bat{\tilde{R}_\omega}}
\end{align}
and
\begin{align}
\label{Eqn:ScaledTWVelocities}
\tilde{V}_\omega = -\tilde{\chi}_\omega\ast \Bat{\delta_\omega^{m/2+1}\al_\omega^{m/2+1}\Phi^\prime\bat{\tilde{R}_\omega}}\,.
\end{align}
Here, $\tilde{\chi}_\omega$ is the indicator function of the interval
\begin{align}
\label{Eqn:DefChiAndI}
 I_\omega:=\ccinterval{-\xi_\omega}{+\xi_\omega}\,,
\end{align}
which corresponds to the interval $\ccinterval{-1/2}{+1/2}$ with respect to the unscaled space variable $x$, and we readily verify
\begin{align}
\label{Eqn:ScaledKernel}
\bat{\tilde{\chi}_\omega\ast\tilde{\chi}_\omega}\at{\tilde{x}}= \max\big\{ 2\xi_\omega-\abs{\tilde{x} },0\big\}\,.
\end{align}
  In other words,    the convolution kernel in \eqref{Eqn:ScaledTWEqnForDistances} is a tent map with height $2\xi_\omega$ and base length $4\xi_\omega$.   For later use we also introduce a modified discrete Laplacian by
\begin{align}
\label{Eqn:DefModLapl}
\bat{\tilde{\Delta}_{w}\tilde{U}}\at{\tilde{x}}=\mhexp{+w}\tilde{U}\at{\tilde{x}+2\xi_\omega}+
\mhexp{-w}\tilde{U}\at{\tilde{x}-2\xi_\omega}-2\tilde{U}\at{\tilde{x}}\,,
\end{align}
where $w\in\Cset$ denotes an arbitrary weight parameter.
\par
As illustrated in Figure~\ref{Fig:Profiles}, the profiles $\Phi^\prime\bat{\tilde{R}_\omega}$ and $\Phi^{\prime\prime}\bat{\tilde{R}_\omega}$ possess  very large amplitudes
in the limit $\om\to\infty$. We therefore define
\begin{align}
\label{Eqn:DefTildeP}
\tilde{P}_\omega\at{\tilde{x}}:=\delta_\omega^{m+2}\al_\omega^{m+2}\Phi^{\prime\prime}
\bat{\tilde{R}_\omega\at{\tilde{x}}}\,,\qquad
\tilde{K}_\omega := \delta_\omega^{m+1}\al_\omega^{m+1}\Phi^{\prime}
\bat{\tilde{R}_\omega\at{\tilde{x}}}
\end{align}
and expect --- according to the informal discussion in \S\ref{sect:Heuristics} --- that
\begin{align}
\label{Eqn:FormalResultA}
\tilde{P}_\omega\bat{\tilde{x}}\approx\bar{P}\at{\tilde{x}}
\qquad \text{and}\qquad
\tilde{K}_\omega\bat{\tilde{x}}\approx\tilde{K}\at{\tilde{x}}
\qquad \text{for}\quad \tilde{x}\in I_\omega\,,
\end{align}
where the limit profiles
\begin{align}
\label{Eqn:DefBarPZ}
\bar{P}\at{\tilde{x}}:=\frac{1}{\bar{Y}\at{\tilde{x}}^{m+2}}\,,\qquad
\bar{K}\at{\tilde{x}}:=-\frac{1}{\at{m+1}\bar{Y}\at{\tilde{x}}^{m+1}}
\end{align}
can be computed from the unique solution of the asymptotic shape ODE \eqref{Lem:AsympODE.Props.Eqn1}.
\par
We finally mention that the amplitude of $\tilde{R}_\omega$ remains bounded as $\om\to\infty$ while $\tilde{V}_\omega$ attains values of order $\DO{\omega}$. We decided not to normalize $\tilde{V}_\omega$ because this would complicate the notations in \S\ref{sect:Stability}   but we   always display normalized velocities in the plots of the numerical data, see for instance Figure~\ref{Fig:WavesVel}.
%
%
%
\subsection{Explicit approximation formulas}\label{sect:AppForm}
%
%
We now introduce an approximate distance profile. The existence of a nearby exact solution is proven in Theorem~\ref{Thm:ExistenceNonlWaves} below, which also provides explicit bounds the approximation error.
\begin{lemma}[approximate distance profile]
\label{Lem:DefBreveR} %
The function
\begin{align}
\label{Lem:DefBreveR.Eqn0}
\breve{R}_\omega:=\tilde{\chi}_\omega\ast \tilde{\chi}_\omega\ast
\at{\tilde{\chi}_\omega\delta_\omega\al_\omega\bar{K}}
\end{align}
is even, continuously differentiable, compactly supported in $3I_\omega$, increasing for $0\leq\tilde{x}\leq 3 \xi_\omega$, convex for $0\leq\tilde{x}\leq \xi_\omega$, and concave for $\xi_\omega\leq\tilde{x}\leq 3 \xi_\omega$. It can be regarded as an approximate distance profile as it satisfies
\begin{align}
\label{Lem:DefBreveR.EqnB}
\breve{R}_\omega=\tilde{\chi}_\omega\ast \tilde{\chi}_\omega\ast
\at{\delta_\omega^{m+2}\al_\omega^{m+2}\Phi^\prime\bat{\breve{R}_\omega}+ \breve{E}_\omega}\,,\qquad \breve{E}_\omega:=
  - \delta_\omega \al_\omega \bat{\tilde{\chi}_\omega\bat{\breve{K}_\omega-\bar{K}}  - \at{1-\tilde{\chi}_\omega}\breve{K}_\omega}
\end{align}
with small error terms in the sense of
\begin{align}
\label{Lem:DefBreveR.EqnA}
\bnorm{\tilde{\chi}_\omega\bat{\breve{K}_\omega-\bar{K}}}_1\leq C\delta_\omega^{\min\{k,m\}}\,,\qquad
\bnorm{\nat{1-\tilde{\chi}_\omega}\breve{K}_\omega}_1\leq C\delta_\omega^{m}
\end{align}
and
\begin{align}
\label{Lem:DefBreveR.EqnC}
\bnorm{\tilde{\chi}_\omega\bat{\breve{P}_\omega-\bar{P}}\bar{Y}}_1\leq C\delta_\omega^{\min\{k,m\}}\,,\qquad
\bnorm{\nat{1-\tilde{\chi}_\omega}\breve{P}_\omega}_\infty\leq C\delta_\omega^{m+2}
\end{align}
for some constant $C$ independent of $\omega$, where $\breve{P}_\omega$ and $\breve{K}_\omega$ are defined in terms of $\breve{R}_\omega$ analogously to \eqref{Eqn:DefTildeP}. Moreover, we have
\begin{align}
\label{Lem:DefBreveR.Eqn2}
\breve{R}_\omega\at{\tilde{x}}=-1+\delta_\omega\al_\omega \bar{Y}\at{\tilde{x}} \qquad \text{for}\quad 0\leq \tilde{x}\leq\xi_\omega
\end{align}
and
\begin{align}
\label{Lem:DefBreveR.Eqn3}
\breve{R}_\omega\bat{\tilde{x}}=\tfrac12\delta_\omega\al_\omega\at{
\bar{Y}\nat{\xi_\omega}-\bar{Y}\at{\tilde{x}-2\xi_\omega}}-
\tfrac12\delta_\omega\al_\omega\bar{Y}^\prime\nat{\xi_\omega}\bat{
3\xi_\omega-\tilde{x}} \qquad \text{for}\quad \xi_\omega\leq \tilde{x}\leq3\xi_\omega
\end{align}
as well as
\begin{align}
\label{Lem:DefBreveR.Eqn1}
\xi_\omega\breve{R}_\omega^\prime\nat{\xi_\omega}=-\breve{R}_\omega\nat{\xi_\omega}\quad\to \quad\tfrac12\,,\qquad
\breve{R}_\omega\nat{2\xi_\omega}\quad\to \quad0\,,\qquad
\breve{R}_\omega^\prime\nat{2\xi_\omega}\quad \to \quad 0
\end{align}
as $\omega\to\infty$.
\end{lemma}
\begin{proof}
\emph{\ul{Formulas for $\breve{R}_\omega$}}: %
By construction, see \eqref{Lem:AsympODE.Props.Eqn1}+\eqref{Eqn:ScaledKernel}+\eqref{Eqn:DefBarPZ}+\eqref{Lem:DefBreveR.Eqn0}, the function $\breve{R}_\omega$ satisfies
\begin{align}
\label{Lem:DefBreveR.PEqn1}
\breve{R}_\omega^{\prime\prime}\at{\tilde{x}}=-2\delta_\omega\al_\omega
\bar{K}\at{\tilde{x}}=\delta_\omega\al_\omega\bar{Y}^{\prime\prime}\at{\tilde{x}}\qquad \text{for}\quad \tilde{x}\in I_\omega\,,
\end{align}
and
\begin{align}
\label{Lem:DefBreveR.PEqn2}
\begin{split}
\breve{R}_\omega\at{0}&=2  \delta_\omega\al_\omega  \int\limits_0^{\xi_\omega}\bat{2\tilde{\xi_\omega}-\tilde{x}}
\bar{K}\at{\tilde{x}}\dint{\tilde{x}}=
-\delta_\omega\al_\omega\int\limits_0^{\xi_\omega}\bat{2\tilde{\xi_\omega}-\tilde{x}}
\bar{Y}^{\prime\prime}\at{\tilde{x}}\dint{\tilde{x}}
\\&=
-2\delta_\omega\al_\omega \xi_\omega \bar{Y}^{\prime}\nat{\xi_\omega}+
\delta_\omega\al_\omega \bat{\bar{Y}^\flat\nat{\xi_\omega}-\bar{Y}^\flat\nat{0}}
=-\delta_\omega\al_\omega\bat{\xi_\omega\bar{Y}^{\prime}\nat{\xi_\omega}+ \bar{Y}\nat{\xi_\omega}}+\delta_\omega\al_\omega\\&
=-1+\delta_\omega\al_\omega\,,
\end{split}
\end{align}
where the last equality holds due to \eqref{Lem:DefAlpha.Eqn1}, i.e., by our choice of $\al_\omega$ and $\xi_\omega$. The combination of \eqref{Lem:DefBreveR.PEqn1} and \eqref{Lem:DefBreveR.PEqn2} gives the representation formula \eqref{Lem:DefBreveR.Eqn2},
the smoothness of $\breve{R}_\omega$ is a consequence of $\breve{R}_\omega^{\prime\prime}= \delta_\om\alpha_\om \tilde{\Delta}_{0}\bat{\tilde{\chi}_\omega \bar{K}}$ with $\tilde{\Delta}_0$ as in \eqref{Eqn:DefModLapl}, and $\supp\,\breve{R}_\omega\subseteq 3I_\omega$ is provided by \eqref{Lem:DefBreveR.Eqn0} and $\supp\, \tilde{\chi}_\omega=I_\omega$. Moreover, \eqref{Lem:DefBreveR.Eqn3} follows in view of \eqref{Lem:DefBreveR.Eqn2}+\eqref{Lem:DefBreveR.PEqn1} from
\begin{align*}
\breve{R}_\omega^{\prime\prime}\at{\tilde{x}}=-\tfrac12 \breve{R}_\omega^{\prime\prime}\at{\tilde{x}-2\xi_\om}=-\tfrac12\delta_\omega\alpha_\omega \bar{Y}_\omega^{\prime\prime}\at{\tilde{x}-2\xi_\om}\qquad
\text{for}\quad \xi_\omega\leq \tilde{x}\leq3\xi_\omega
\end{align*}
thanks to the smoothness of $\breve{R}_\om$ at $\tilde{x}=\xi_\om$, and \eqref{Lem:DefBreveR.Eqn1} can be concluded from Lemma~\ref{Lem:AsympODE.Props} and Lemma~\ref{Lem:DefAlpha}. Finally, the claimed monotonicity properties of $\breve{R}_\omega$ and its derivative are implied by \eqref{Lem:DefBreveR.Eqn2} and \eqref{Lem:DefBreveR.Eqn3}.
\par
\emph{\ul{Error estimates}}: %
The properties of $\xi_\omega$ and $\bar{Y}$ ensure
\begin{align*}
-1\leq\breve{R}_\omega\at{\tilde{x}}=-1+\delta_\omega\al_\omega\bar{Y}\at{\tilde{x}}\leq
  -1+\delta_\omega\al_\omega\bar{Y}\at{\xi_\omega}
\leq   -\tfrac14\quad\text{for}\quad \abs{\tilde{x}}\leq\xi_\omega\,,
\end{align*}
  where we also used the convergence results in \eqref{Lem:DefAlpha.Eqn3}.
By   Assumption~\ref{Ass.Pot} we get
\begin{align}
\label{Lem:DefBreveR.PEqn3}
\babs{\breve{K}_\omega\at{\tilde{x}}-\bar{K}\at{\tilde{x}}}\leq C\delta_\omega^{k}\bar{Y}\at{\tilde{x}}^{-m-1+k}
\leq C\delta_\omega^{k}\at{1+\abs{\tilde{x}}}^{-m-1+k}
\qquad \text{for}\quad \abs{\tilde{x}}\leq\xi_\omega
\end{align}
and hence --- for both $k<m$ and $k>m$ --- the desired estimate \eqref{Lem:DefBreveR.EqnA}$_1$ after integration (see also the comment below). Moreover, thanks to
\begin{align*}
-\tfrac{3}{4}\leq\breve{R}_\omega\at{\tilde{x}}\leq0 \quad\text{for}\quad \xi_\omega\leq \abs{\tilde{x}}\leq3\xi_\omega\,,\qquad \breve{R}_\omega\at{\tilde{x}}=0
\quad\text{for}\quad \abs{\tilde{x}}\geq3\xi_\omega
\end{align*}
and the regularity properties of $\Phi$ we find
\begin{align*}
\babs{\breve{K}_\omega\at{\tilde{x}}}\leq
C\delta_\omega^{m+1}\babs{\breve{R}_\omega\at{\tilde{x}}}\leq
C\delta_\omega^{m+1} \quad\text{for}\quad \xi_\omega\leq \abs{\tilde{x}}\leq 3\xi_\omega\,,\qquad \breve{K}_\omega\at{\tilde{x}}=0
\quad\text{for}\quad \abs{\tilde{x}}\geq3\xi_\omega\,,
\end{align*}
  so the claim \eqref{Lem:DefBreveR.EqnA}$_2$ follows   after integration over $\Rset\setminus I_\omega$. In the same way we demonstrate
\begin{align}
\label{Lem:DefBreveR.PEqn4}
\babs{\breve{P}_\omega\at{\tilde{x}}-\bar{P}\at{\tilde{x}}}\bar{Y}\at{\tilde{x}}\leq C\delta_\omega^{k}\bar{Y}\at{\tilde{x}}^{-m-1+k}
\quad \text{for}\quad \abs{\tilde{x}}\leq\xi_\omega
\end{align}
as well as
\begin{align*}
\babs{\breve{P}_\omega\at{\tilde{x}}}\leq C\delta_{\omega}^{m+2}\quad \text{for}\quad \abs{\tilde{x}}\geq\xi_\omega
\end{align*}
to obtain \eqref{Lem:DefBreveR.EqnC}. Finally, the identity \eqref{Lem:DefBreveR.EqnB} holds by construction and the proof is complete.
\end{proof}
The proof of Lemma~\ref{Lem:DefBreveR} exploits --- see the arguments after \eqref{Lem:DefBreveR.PEqn3} and \eqref{Lem:DefBreveR.PEqn4} --- the elementary estimate
\begin{align*}
\xi_\omega^{-k}\int\limits_0^{\xi_\omega}{\at{1+\tilde{x}}}^{k-m-1}\dint{\tilde{x}}
\leq C \left\{\begin{array}{lcl}
\xi_\omega^{-k}&&\text{for $k<m$}\\
\xi_\omega^{-m}\ln\xi_\omega&&\text{for $k=m$}\\
\xi_\omega^{-m}&&\text{for $k>m$}\\
\end{array}\right.
\end{align*}
for either $k< m$ or $k>m$, and this gives rise to the error terms in both \eqref{Lem:DefBreveR.EqnA}$_1$ and
\eqref{Lem:DefBreveR.EqnC}$_1$. However, the assertions of Lemma~\ref{Lem:DefBreveR} --- and hence all asymptotic results derived below --- remain valid for $k=m$ provided that we replace $\delta_\omega^{\min\{{k,m}\}}$  by $\delta_\omega^m\babs{\ln \delta_\omega}$.
\par
We further emphasize that \eqref{Lem:DefBreveR.Eqn2} can be viewed as the analogue to the heuristic formulas \eqref{Eqn:DefTildeY}+\eqref{Eqn:Overview.2}+\eqref{Eqn:Overview.3} as it links  $\breve{R}_\omega$ to $\bar{Y}$, the solution of the asymptotic ODE initial value problem \eqref{Lem:AsympODE.Props.Eqn1}. Moreover, differentiating \eqref{Lem:DefBreveR.Eqn3} gives the analogue to \eqref{Eqn:Overview.9}, and we conclude that the ad hoc definition of $\breve{R}_\omega$ in \eqref{Lem:DefBreveR.Eqn0} encodes all informal ODE arguments from the heuristic discussion in \S\ref{sect:Heuristics}.
\par
Finally, $\breve{R}_\omega$ is completely determined by $\omega$ and $m$, and the corresponding approximation to the unscaled distance profile $R_\omega$ can be obtained as follows:
\begin{enumerate}
\item
compute $\bar{Y}$ as the solution to   the   initial value problem \eqref{Lem:AsympODE.Props.Eqn1} with parameter $m$,
\item
determine $\xi_\omega$ by solving the nonlinear but scalar equation \eqref{Lem:DefAlpha.Eqn1}, which also fixes the values of $\al_\omega$,  $\beta_\omega$ via \eqref{Lem:DefAlpha.Eqn2},
\item
construct both the scaled `tip of the tent' and   the   scaled 'base of the tent' from the explicit formulas \eqref{Lem:DefBreveR.Eqn2}
and \eqref{Lem:DefBreveR.Eqn3} for $\tilde{x}\in I_\omega$ and $\tilde{x}\in 3I_\omega \setminus I_\omega$, respectively, with $I_\omega$ as in \eqref{Eqn:DefChiAndI},
\item
glue the local profiles together, extend trivially, and pass to the original space variable $x$ by scaling according to \eqref{Eqn:SpaceScalign} and \eqref{Eqn:ScaledProfiles}.
\end{enumerate}
We used this strategy to compute the ODE approximations in Figures~\ref{Fig:WavesDist} and~\ref{Fig:WavesErr}.
%
%
%
\subsection{Simplified linear traveling wave equation}\label{sect:AuxProb}
%
In this section we state and prove the key technical result in the chapter on nonlinear waves, which below allows us to demonstrate the existence, uniqueness, and smooth $\omega$-dependence of high-energy waves, and to validate the explicit approximation formulas from the previous section.
\begin{lemma}[solvability of a nonlocal auxiliary problem]
\label{Lem:JordanAuxRes} %
For fixed $\omega$ and any given $\tilde{F}\in\fspaceL^1_\even$ there exists a unique solution $\tilde{U}\in\fspaceL^1_\even$ to
\begin{align}
\label{Lem:JordanAuxRes.Eqn1} %
\tilde{U}=\tilde{\chi}_\omega\ast\tilde{\chi}_\omega\ast \bat{\tilde{\chi}_\omega\bar{P}\,\tilde{U}+\tilde{F}}\,,
\end{align}
which is continuously differentiable and satisfies
\begin{align}
\label{Lem:JordanAuxRes.Eqn2} %
 \bnorm{\tilde{U}}_{*,\,\omega}:=\babs{\tilde{U}\at{0}}+ \bnorm{\tilde{U}^{\prime\prime}}_1+
\delta_\omega^2\bnorm{\tilde{U}}_1+\delta_\omega\bnorm{\tilde{U}}_\infty\leq C\bnorm{\tilde{F}}_1
\end{align}
for some constant $C$ depending only on the parameter $m$. Moreover,
we have
\begin{align}
\label{Lem:JordanAuxRes.Eqn3a}
\xi_\omega\tilde{U}^\prime\nat{\xi_\omega}+\tilde{U}\nat{\xi_\omega}=0\,,\qquad
\supp\,\tilde{U}\subseteq 3 I_\omega
\end{align}
and
\begin{align}
\label{Lem:JordanAuxRes.Eqn3b}
\tilde{U}\at{x}=\tfrac12 \tilde{U}\at{\xi_\omega}-\tfrac12\tilde{U}\at{\tilde{x}-2\xi_\omega}+\tfrac12\tilde{U}^\prime\at{\xi_\omega}\at{\tilde{x}-3\xi_\omega}\quad \text{for}\quad \xi_\omega\leq\tilde{x}\leq 3\xi_\omega
\end{align}
provided that $\tilde{F}$ is supported in $I_\omega$.
\end{lemma}
\begin{proof}
\emph{\ul{Reformulation of the problem}}: %
Since $\tilde{\chi}_\omega$  vanishes for $\abs{\tilde{x}}\geq \xi_\omega$, the even function $\tilde{U}$ is uniquely determined by its restriction to $I_\omega$. Therefore, and since $\tilde{U}$ satisfies the linear advance-delay differential equation
\begin{align}
\label{Lem:JordanAuxRes.PEqn0}
\tilde{U}^{\prime\prime}=\tilde{\Delta}_{0}\bat{\tilde{\chi}_\omega\bar{P}\,\tilde{U}+\tilde{F}}\,,
\end{align}
according to the definition of $\tilde{\chi}_\omega$ in \eqref{Eqn:ScaledKernel}, we conclude that \eqref{Lem:JordanAuxRes.Eqn1} is equivalent to the combination of the ODE
\begin{align}
\label{Lem:JordanAuxRes.PEqn1} %
\tilde{U}^{\prime\prime}\at{\tilde{x}}=-2\bar{P}\at{\tilde{x}}\tilde{U}\at{\tilde{x}}+  \tilde{E}_\omega\at{\tilde{x}} \qquad \text{for}\quad \tilde{x}\in I_\omega
\end{align}
and the nonlocal consistency relation
\begin{align}
\label{Lem:JordanAuxRes.PEqn2} %
\tilde{U}\at{0}=2\int\limits_0^{2\xi_\omega}\at{2\xi_\omega-\tilde{x}}\Bat{\tilde{\chi}_\omega\bat{\tilde{x}}\bar{P}\at{\tilde{x}}\tilde{U}\at{\tilde{x}}+\tilde{F}\at{\tilde{x}}}\dint{\tilde{x}}\,,
\end{align}
where we used the abbreviation
\begin{align*}
\tilde{E}_\omega:=\tilde{\Delta}_{0}\tilde{F}
\end{align*}
in \eqref{Lem:JordanAuxRes.PEqn1} and derived \eqref{Lem:JordanAuxRes.PEqn2} by evaluating the convolution integral in the claim \eqref{Lem:JordanAuxRes.Eqn1} at $x=0$ and by means of \eqref{Eqn:ScaledKernel}. In the next step we show that \eqref{Lem:JordanAuxRes.PEqn2} determines a unique value for $\tilde{U}\at{0}$ and hence a unique even solution to the initial value problem corresponding to \eqref{Lem:JordanAuxRes.PEqn1}.
\par
\emph{\ul{Existence and uniqueness of $\tilde{U} $}}: %
In view of \eqref{Lem:JordanAuxRes.PEqn1}, the consistency relation can be written as
\begin{align}
\label{Lem:JordanAuxRes.PEqn5}
\tilde{U}\at{0}=-\int\limits_0^{\xi_\omega}\at{2\xi_\omega-\tilde{x}}\tilde{U}^{\prime\prime}\at{\tilde{x}}\dint{\tilde{x}}+\tilde{d}_\omega
\end{align}
with
\begin{align*}
\tilde{d}_\omega:=
\int\limits_0^{\xi_\omega}\at{2\xi_\omega-\tilde{x}}\tilde{E}_\omega\at{\tilde{x}}\dint{\tilde{x}}+
2\int\limits_0^{2\xi_\omega}\at{2\xi_\omega-\tilde{x}}\tilde{F}\at{\tilde{x}}\dint{\tilde{x}}\,,
\end{align*}
and after direct computations in \eqref{Lem:JordanAuxRes.PEqn5} we derive the simplified formula
\begin{align}
\label{Lem:JordanAuxRes.PEqn3}
\xi_\omega\tilde{U}^{\prime}\at{\xi_\omega}+\tilde{U}\at{\xi_\omega}=\tilde{d}_\omega
\end{align}
as an equivalent reformulation of \eqref{Lem:JordanAuxRes.PEqn2}. Moreover, the Duhamel Principle applied to \eqref{Lem:JordanAuxRes.PEqn1} provides the representation formulas
\begin{align}
\label{Lem:JordanAuxRes.PEqn4a}
\tilde{U}\at{\tilde{x}}=\tilde{U}\at{0}\bar{T}_\even\at{\tilde{x}}+\int\limits_0^{\tilde{x}}\at{\bar{T}_\odd\at{\tilde{x}}\bar{T}_\even\at{\tilde{y}}-\bar{T}_\even\at{\tilde{x}}\bar{T}_\odd\at{\tilde{y}}}\tilde{E}_\omega\at{\tilde{y}}\dint{\tilde{y}}\qquad \text{for}\quad \tilde{x}\in I_\omega
\end{align}
and
\begin{align*}
\tilde{U}^\prime\at{\tilde{x}}=
\tilde{U}\at{0}\bar{T}_\even^\prime\at{\tilde{x}}+\int\limits_0^{\tilde{x}}\at{\bar{T}_\odd^\prime\at{\tilde{x}}\bar{T}_\even\at{\tilde{y}}-\bar{T}_\even^\prime\at{\tilde{x}}\bar{T}_\odd\at{\tilde{y}}}\tilde{E}_\omega\at{\tilde{y}}\dint{\tilde{y}}\qquad \text{for}\quad \tilde{x}\in I_\omega\,,
\end{align*}
which hold due to the Wronski identities
\begin{align*}
\bar{T}_\odd^\prime\at{\tilde{x}}\bar{T}_\even\at{\tilde{x}}-\bar{T}_\even^\prime\at{\tilde{x}}\bar{T}_\odd\at{\tilde{x}}=
\bar{T}_\odd^\prime\at{0}\bar{T}_\even\at{0}-\bar{T}_\even^\prime\at{0}\bar{T}_\odd\at{0}=1
\end{align*}
and imply

\begin{align*}
\tilde{U}\at{\xi_\omega}=\at{\tilde{U}\at{0}-\int\limits_0^{\xi_\omega}\bar{T}_\odd\at{\tilde{y}}\tilde{E}_\omega\at{\tilde{y}}\dint{\tilde{y}}}\bar{T}_\even\at{\xi_\omega}+
\at{\int\limits_0^{\xi_\omega}\bar{T}_\even\at{\tilde{y}}\tilde{E}_\omega\at{\tilde{y}}\dint{\tilde{y}}}\bar{T}_\odd\at{\xi_\omega}
\end{align*}
and
\begin{align*}
\tilde{U}^\prime\at{\xi_\omega}=\at{\tilde{U}\at{0}-\int\limits_0^{\xi_\omega}\bar{T}_\odd\at{\tilde{y}}\tilde{E}_\omega\at{\tilde{y}}\dint{\tilde{y}}}\bar{T}_\even^\prime\at{\xi_\omega}+
\at{\int\limits_0^{\xi_\omega}\bar{T}_\even\at{\tilde{y}}\tilde{E}_\omega\at{\tilde{y}}\dint{\tilde{y}}}\bar{T}_\odd^\prime\at{\xi_\omega}\,.
\end{align*}
Inserting the latter two identities into \eqref{Lem:JordanAuxRes.PEqn3} --- and employing the asymptotic properties of $\bar{T}_\even$ and $\bar{T}_\odd$ from Lemma~\ref{Lem:LinODE.Props} ---  reveals that $\tilde{U}\at{0}$ is uniquely determined and satisfies
\begin{align*}
\babs{\tilde{U}\at{0}}\leq
\abs{\int\limits_0^{\xi_\omega}\bar{T}_\odd\at{\tilde{y}}\tilde{E}_\omega\at{\tilde{y}}\dint{\tilde{y}}}+C\xi_\delta^{-1}\abs{\int\limits_0^{\xi_\omega}\bar{T}_\even\at{\tilde{y}}\tilde{E}_\omega\at{\tilde{y}}\dint{\tilde{y}}}+
C\xi_\delta^{-1}\tilde{d}_\omega\leq C\bat{\nnorm{\tilde{E}_\omega}_1+\delta_\omega\nabs{\tilde{d}_\omega}}\,,
\end{align*}
which yields the desired bound for $\nabs{\tilde{U}\at{0}}$ since we have $\nnorm{\tilde{E}_\omega}_1\leq 4\nnorm{\tilde{F}}_1$ and $\nabs{\tilde{d}_\omega}\leq C\delta_\omega^{-1}\nnorm{\tilde{F}}_1$.
\par
\emph{\ul{Estimates for $\tilde{U}$}}: %
From \eqref{Lem:JordanAuxRes.PEqn4a} and the properties of   $\tilde{T}_\even$, $\tilde{~T}_\odd$   we further infer
\begin{align*}
\babs{\bar{P}\at{\tilde{x}}\tilde{U}\at{\tilde{x}}}\leq\babs{\tilde{U}\at{0}}\babs{\bar{P}\at{\tilde{x}}\bar{T}_\even\at{\tilde{x}}}+
C\babs{\tilde{x}\bar{P}\at{\tilde{x}}} \bnorm{\tilde{F}}_1\qquad \text{for}\quad \tilde{x}\in I_\omega\,,
\end{align*}
and obtain
\begin{align*}
\nnorm{\tilde{\chi}_\omega\bar{P}\,\tilde{U}}_1\leq \bnorm{\tilde{F}}_1
\end{align*}
after integration with respect to $\tilde{x}\in I_\omega$. Moreover, Young's inequality for convolutions   in \eqref{Lem:JordanAuxRes.Eqn1}   implies
\begin{align*}
\delta_\omega^{2}\bnorm{\tilde{U}}_1+\delta_\omega\bnorm{\tilde{U}}_\infty\leq  C\bnorm{\tilde{\chi}_\omega\bar{P}\,\tilde{U}+\tilde{F}}_1\leq C\bnorm{\tilde{F}}_1
\end{align*}
thanks to $\nnorm{\tilde{\chi}_\omega\ast\tilde{\chi}_\omega }_1\leq C\delta_\omega^{-2}$ and $\nnorm{\tilde{\chi}_\omega\ast\tilde{\chi}_\omega }_\infty\leq C\delta_\omega^{-1}$, while
\begin{align*}
\bnorm{\tilde{U}^{\prime\prime}}_1\leq 4\bnorm{\tilde{\chi}_\omega\bar{P}\,\tilde{U}+\tilde{F}}_1\leq C \bnorm{\tilde{F}}_1
\end{align*}
is a consequence of \eqref{Lem:JordanAuxRes.PEqn0}.
\par
\emph{\ul{Concluding arguments}}: %
Now suppose $\tilde{F}\at{\tilde{x}}=0$ for $x\notin I_\omega$ and notice that this implies $\tilde{d}_\omega=0$ and hence \eqref{Lem:JordanAuxRes.Eqn3a} via \eqref{Lem:JordanAuxRes.PEqn3}. Moreover, differentiating  \eqref{Lem:JordanAuxRes.Eqn1} twice with respect to $\tilde{x}$ gives
\begin{align*}
\tilde{U}^{\prime\prime}\at{\tilde{x}}=-2\bar{P}\bat{\tilde{x}}
\tilde{U}\bat{\tilde{x}}-2\tilde{F}\bat{\tilde{x}}=-2 \tilde{U}^{\prime\prime}\at{\tilde{x}+2\xi_\omega}\qquad\text{for}\quad \abs{\tilde{x}}\leq \xi_\omega\,,
\end{align*}
and we deduce that the function $\tilde{C}$ with
\begin{align}
\label{Lem:JordanAuxRes.PEqn8}
\tilde{C}\at{\tilde{x}}:=\tilde{U}\at{\tilde{x}}+2\tilde{U}\at{\tilde{x}+2\xi_\omega}
\end{align}
is affine on $I_\omega$ as its second derivatives vanishes on that interval. After Taylor expansion and due to
\begin{align*}
\tilde{U}\at{3\xi_\omega}=\tilde{U}^\prime\at{3\xi_\omega}=0
\end{align*}
we therefore find
\begin{align}
\label{Lem:JordanAuxRes.PEqn9}
\tilde{C}\at{\tilde{x}}=\tilde{C}\at{\xi_\omega}+\tilde{C}^\prime\at{\xi_\omega}\at{x-\xi_\omega}
=\tilde{U}\at{\xi_\omega}+\tilde{U}^\prime\at{\xi_\omega}\at{\tilde{x}-\xi_\omega}
\qquad\text{for}\quad \abs{\tilde{x}}\leq \xi_\omega\,.
\end{align}
The claim \eqref{Lem:JordanAuxRes.Eqn3b} finally follows from evaluating \eqref{Lem:JordanAuxRes.PEqn8} and \eqref{Lem:JordanAuxRes.PEqn9} at $\tilde{x}-2\xi_\omega$ and by rearranging terms.
\end{proof}
Notice that  the problem \eqref{Lem:JordanAuxRes.Eqn1} can be regarded as a simplification of the linearized traveling wave equation provided that we are able to control the approximation $\tilde{P}_\omega\approx \tilde{\chi}_\omega\bar{P}$, which is consequently among our main goals in this section. We further emphasize that
we have
\begin{align}
\label{Lem:JordanAuxRes.Eqn4a} %
\bnorm{\tilde{U}^\prime}_\infty\leq   \bnorm{U^{\prime\prime}}_1\leq   \bnorm{\tilde{U}}_{*,\,\omega}
\end{align}
for any even function $\tilde{U}$ and that the properties of   $\bar{Y}$   from Lemma~\ref{Lem:AsympODE.Props} imply the pointwise estimate
\begin{align}
\label{Lem:JordanAuxRes.Eqn4b} %
\babs{\tilde{U}\at{\tilde{x}}}\leq
\babs{\tilde{U}\at{0}}+\bnorm{\tilde{U}^{\prime}}_\infty\abs{\tilde{x}}
\leq
\babs{\tilde{U}\at{0}}+\bnorm{\tilde{U}^{\prime\prime}}_1\abs{\tilde{x}}
\leq C  \bnorm{\tilde{U}}_{*,\,\omega}\bar{Y}\at{\tilde{x}}
\end{align}
for all $\tilde{x}\in\Rset$, which plays a prominent role in the derivation of our asymptotic arguments.
%
%
%
%
%
\subsection{Existence and uniqueness}\label{sect:ExAndUni}
%
We finally combine the approximate solutions from \S\ref{sect:AppForm} with the auxiliary problem from \S\ref{sect:AuxProb} to prove the existence of high-energy waves by means of Banach's Fixed Point Theorem.
\begin{theorem}[$\omega$-family of nonlinear lattice waves]
\label{Thm:ExistenceNonlWaves} %
There exists a constant $D$ which depends only on $\Phi$ such that the scaled traveling wave equation \eqref{Eqn:ScaledTWEqnForDistances} admits for all sufficiently large $\omega$ a unique even solution $\tilde{R}_\omega$ in the set
\begin{align}
\label{Thm:ExistenceNonlWaves.Eqn2}
\bnorm{\tilde{R}_\omega-\breve{R}_\omega}_{*,\,\omega}\leq \delta_\omega^{\min\{k,m\}+1}D\,,
\end{align}
where the approximate wave $\breve{R}_\omega$ and the norm $\nnorm{\cdot }_{*,\,\omega}$ are defined in Lemma~\ref{Lem:DefBreveR} and Lemma~\ref{Lem:JordanAuxRes}. Moreover, $\tilde{R}_\omega$ attains values in $\ocinterval{-1}{0}$ and we have
\begin{align}
\label{Thm:ExistenceNonlWaves.Eqn3}
\bnorm{\tilde{\chi}_\omega\bat{\tilde{K}_\omega-\bar{K}}}_1\leq C\delta_\omega^{\min\{k,m\}}\,,\qquad   \bnorm{\at{1-\tilde{\chi}_\omega}\tilde{K}_\omega}_1\leq C\delta_\omega^{m}
\end{align}
as well as
\begin{align}
\label{Thm:ExistenceNonlWaves.Eqn1}
\bnorm{\tilde{\chi}_\omega\bat{\tilde{P}_\omega-\bar{P}}\bar{Y}}_1\leq C\delta_\omega^{\min\{k,m\}}\,,\qquad   \bnorm{\at{1-\tilde{\chi}_\omega}\tilde{P}_\omega}_\infty\leq C\delta_\omega^{m+2}
\end{align}
for some constant $C$ independent of $\omega$, where $\tilde{P}_\omega$ and $\tilde{K}_\omega$ are defined in \eqref{Eqn:DefTildeP}.
\end{theorem}
\begin{proof}
Within this proof we use the abbreviation
\begin{align*}
l:=\min\{k,m\}
\end{align*}
and   setting $\Phi\at{r}:=0$ for $r>0$ we can consider $\Phi$ as a   continuously differentiable function on $\oointerval{-1}{\infty}$ with piecewise continuous second derivative.   Moreover, we restrict all considerations to the space of even functions.
\par
\emph{\ul{Reformulation as fixed-point problem}}: %
We make the ansatz
\begin{align*}
\tilde{R}_\omega = \breve{R}_\omega + \delta_\omega^{l+1}\tilde{U}_\omega\,,\qquad
\end{align*}
and write the equation for $\tilde{R}_\omega$ first as
\begin{align*}
\tilde{U}_\omega = \tilde{\chi}_\omega\ast \tilde{\chi}_\omega\ast \Bat{\tilde{\chi}_\omega\bar{P}\, \tilde{U}_\omega+\tilde{A}_\omega + \tilde{B}_\omega\tilde{U}_\omega + \tilde{\calC}_\omega\nato{\tilde{U}_\omega}+
\tilde{\calD}_\omega\bato{\tilde{U}_\omega}}
\end{align*}
and afterwards as $\tilde{U}_\omega= \tilde{\calF}_\omega\bato{\tilde{U}_\omega}$ with
\begin{align*}
\calF_\omega\bato{\tilde{U}} :=  \tilde{\calK}_\omega  \Bato{\tilde{A}_\omega + \tilde{B}_\omega\tilde{U} + \tilde{\calC}_\omega\nato{\tilde{U}}+\tilde{\calD}_\omega\bato{\tilde{U}}}\,.
\end{align*}
Here, $ \tilde{\calK}_\omega$   denotes the linear solution operator the auxiliary problem \eqref{Lem:JordanAuxRes.Eqn1} and we have
\begin{align*}
\tilde{A}_\omega := -\delta_\omega^{-l-1}\breve{E}_\omega
\,,\qquad
\tilde{B}_\omega := \tilde{\chi}_\omega\bat{\breve{P}_\omega-\bar{P}}
\end{align*}
with $\breve{E}_\omega$ as in \eqref{Lem:DefBreveR.EqnB}, as well as
\begin{align*}
\tilde{\calC}_\omega\nato{\tilde{U}}:=\delta_\omega^{m+1-l}\al_\omega^{m+2}\tilde{\chi}_\omega\at{\Phi^{\prime}\bat{\breve{R}_\omega+\delta_\omega^{l+1}\tilde{U}}-\Phi^{\prime}\bat{\breve{R}_\omega}-\delta_\omega^{l+1}\Phi^{\prime\prime}\bat{\breve{R}_\omega}\tilde{U}}
\end{align*}
and
\begin{align*}
\tilde{\calD}_\omega\nato{\tilde{U}}:=\at{1-\tilde{\chi}_\omega}\delta_\omega^{m+1-l}\al_\omega^{m+2}\Bat{\Phi^\prime\bat{\breve{R}_\omega+\delta_\omega^{l+1}\tilde{U}}-\Phi^\prime\bat{\breve{R}_\omega}}\,.
\end{align*}
Our strategy is to show that $\tilde{\calF}_\omega$ is contractive on some ball in a certain function space.
\par
\emph{\ul{Properties of the operator $\tilde{\calF}_\omega$}}: %
Suppose that
\begin{align*}
\bnorm{\tilde{U}}_{*,\,\omega}\leq D
\end{align*}
where the constant $D$ will be chosen below, and   observe that \eqref{Lem:DefBreveR.EqnA}, \eqref{Lem:DefBreveR.EqnC} and  \eqref{Lem:JordanAuxRes.Eqn4b}   provide
\begin{align*}
\bnorm{\tilde{A}_\omega}_1\leq C
\end{align*}
  as well as
\begin{align*}
\bnorm{\tilde{B}_\omega\tilde{U}}_1\leq
\bnorm{\tilde{\chi}_\omega\nat{\tilde{P}_\omega-\bar{P}}\bar{Y}}_1\bnorm{\tilde{U}\bar{Y}^{-1}}_{\infty} \leq  C    \bnorm{\tilde{\chi}_\omega\nat{\tilde{P}_\omega-\bar{P}}\bar{Y}}_1\bnorm{\tilde{U}}_{*,\,\omega} \leq C\delta_\omega^l  D\,.
\end{align*}
Now let $\tilde{U}_*$ be another function with $\nnorm{\tilde{U}_*}_{*,\,\omega}\leq D$   and notice that the function
\begin{align*}
u\,\;\mapsto\;\, \psi_{\tilde{x}}\at{u}:=
 \Phi^{\prime}\bat{\breve{R}_\omega\at{\tilde{x}}+\delta_\omega^{l+1}u}-\Phi^{\prime}\bat{\breve{R}_\omega\at{\tilde{x}}}-\delta_\omega^{l+1}\Phi^{\prime\prime}\bat{\breve{R}_\omega\at{\tilde{x}}}u
\end{align*}
satisfies $\psi_{\tilde{x}}\at{0}=0=\psi^\prime_{\tilde{x}}\at{0}$ for any
given $\tilde{x}$. The Mean Value Theorem therefore yields via
\begin{align*}
\psi_{\tilde{x}}\at{u}-\psi_{\tilde{x}}\at{u_*}=\psi_{\tilde{x}}^{\prime}\at{u_1}\bat{u-u_*}=\psi_{\tilde{x}}^{\prime\prime}\at{u_2}u_1\at{u-u_*}\qquad \text{for some $u_1$, $u_2$ depending on $u$ and $u_*$}
\end{align*}
the estimate
\begin{align}
\label{Thm:ExistenceNonlWaves.PEqn2a}
\Babs{\tilde{\calC}_\omega\nato{\tilde{U}}\at{\tilde{x}}-\tilde{\calC}_\omega\nato{\tilde{U}_*}\at{\tilde{x}}}&\leq
C\delta_\omega^{m+l+3}\Babs{\Phi^{\prime\prime\prime}\bat{\breve{R}_\omega\at{\tilde{x}}+\delta_\omega^{l+1}\tilde{U}_2\at{\tilde{x}}}}\abs{\tilde{U}_1\at{\tilde{x}}}\abs{\tilde{U}\at{\tilde{x}}-\tilde{U}_*\at{\tilde{x}}}\,.
\end{align}
Here, both $\tilde{U}_1\at{\tilde{x}}$ and $\tilde{U}_2\at{\tilde{x}}$ belong to the smallest interval containing $\{0,\tilde{U}\at{\tilde{x}},\tilde{U}_*\at{\tilde{x}}\}$ and can be bounded by
\begin{align}
\label{Thm:ExistenceNonlWaves.PEqn2b}
\babs{\tilde{U}_i\at{\tilde{x}}}\leq  \max\Big\{\babs{\tilde{U}\at{\tilde{x}}},\,\; \babs{\tilde{U}_*\at{\tilde{x}}}\Big\}\leq  D\bar{Y}\at{\tilde{x}}\,,\qquad i=1,2
\end{align}
thanks to \eqref{Lem:JordanAuxRes.Eqn4b}, which also implies that
\begin{align}
\label{Thm:ExistenceNonlWaves.PEqn2c}
\babs{\tilde{U}\at{\tilde{x}}-\tilde{U}_*\at{\tilde{x}}}\leq   C   \nnorm{\tilde{U}-\tilde{U}_*}_{*,\,\omega}\bar{Y}\at{\tilde{x}}\,.
\end{align}
From \eqref{Lem:DefBreveR.Eqn2}+\eqref{Lem:JordanAuxRes.Eqn4b} and for
\begin{align*}
\tilde{x}\in I_\omega\,,\qquad\delta_\omega\leq D/C
\end{align*}
we further infer the pointwise estimate
\begin{align}
\label{Thm:ExistenceNonlWaves.PEqn3}
-1+\tfrac12 \delta_\omega \al_\omega \bar{Y}\at{\tilde{x}}\leq
\breve{R}_\omega\at{\tilde{x}}+\delta_\omega^{l+1}\tilde{U}_2\at{\tilde{x}} \leq -1+2\delta_\omega \al_\omega \bar{Y}\at{\tilde{x}} \,,
\end{align}
and conclude in view of Assumption~\ref{Ass.Pot} and by \eqref{Thm:ExistenceNonlWaves.PEqn2a}+\eqref{Thm:ExistenceNonlWaves.PEqn2b}+\eqref{Thm:ExistenceNonlWaves.PEqn2c} that
\begin{align*}
\Babs{\tilde{\calC}_\omega\nato{\tilde{U}}\at{\tilde{x}}-\tilde{\calC}_\omega\nato{\tilde{U}_*}\at{\tilde{x}}}&\leq C\delta_\omega^{l}\frac{\babs{\tilde{U}_1\at{\tilde{x}}}\babs{\tilde{U}\at{\tilde{x}}-\tilde{U}_*\at{\tilde{x}}}}{\bar{Y}\at{\tilde{x}}^{m+3}}
\leq
\delta_\omega^lC D\frac{\bnorm{\tilde{U}-\tilde{U}_*}_{*,\,\omega}}{\bar{Y}\at{\tilde{x}}^{m+1}}\,.
\end{align*}
After integration with respect to $\tilde{x}\in I_\omega$ we finally get
\begin{align*}
\bnorm{\tilde{\calC}_\omega\nato{\tilde{U}}-\tilde{\calC}_\omega\nato{\tilde{U}_*}}_1\leq \delta_\omega^{l}CD\bnorm{\tilde{U}-\tilde{U}_*}_{*,\,\omega}\,,\qquad
\bnorm{\tilde{\calC}_\omega\nato{\tilde{U}}}_1\leq \delta_\omega^{l}CD^2
\end{align*}
for all sufficiently large $\omega$, where the second identity follows from the first one via $\tilde{U}_*=0$.  Using
\begin{align*}
 \breve{R}_\omega\at{\tilde{x}}+\delta_\omega^{l+1} \tilde{U}_\omega\at{\tilde{x}}\geq -\tfrac{3}{4}\qquad \text{for}\quad \tilde{x}\notin I_\omega
\end{align*}
and the analogous result for $\tilde{U}_*$ we further verify
\begin{align*}
\abs{\tilde{\calD}_\omega\bato{\tilde{U}}\at{\tilde{x}}-\tilde{\calD}_\omega\bato{\tilde{U}_*}\at{\tilde{x}}}\leq C\delta_\omega^{m+2}
\babs{\tilde{U}\at{\tilde{x}}-\tilde{U}_*\at{\tilde{x}}}\qquad\text{for}\quad\tilde{x}\notin I_\omega
\end{align*}
  since $\Phi^\prime$ is globally Lipschitz continuous on $\cointerval{-3/4}{\infty}$. We thus obtain
\begin{align*}
\bnorm{\tilde{\calD}_\omega\bato{\tilde{U}}-\tilde{\calD}_\omega\bato{\tilde{U}_*}}_1\leq
\delta_\omega^{m+2}C\bnorm{\tilde{U}-\tilde{U}_*}_1
\leq
\delta_\omega^{m}C\bnorm{\tilde{U}-\tilde{U}_*}_{*,\,\omega}\,,\qquad
\bnorm{\tilde{\calD}_\omega\bato{\tilde{U}}}_{1}\leq \delta_\omega^{m}CD
\end{align*}
by integration over $\Rset\setminus I_\omega$ and setting $\tilde{U}_*=0$.
\par
\emph{\ul{Existence and local uniqueness}}: %
Combining all partial results derived so far we demonstrate the implication
\begin{align*}
\bnorm{\tilde{U}}_{*,\,\omega}\leq D\quad\implies \quad
\bnorm{\tilde{\calF}_\omega\bato{\tilde{U}}}_{*,\,\omega}\leq C\at{1+\delta_\omega^l D+\delta_\omega^l D^2+\delta_\omega^{m}D}\leq D
\end{align*}
for all  sufficiently large $\omega$ provided that $D$ is first chosen sufficiently large and independent of $\omega$. Moreover, the operator $\tilde{\cal F}_\omega$ is --- again for large $\omega$ --- also contractive in the $D$-ball with respect to the
$\nnorm{\cdot}_{*,\,\omega}$-norm, so both the existence and local uniqueness of a fixed point $\tilde{U}_\omega$ is granted by Banach's Contraction Mapping Principle. Notice also that \eqref{Eqn:ScaledTWEqnForDistances} combined with the non-negativity of $-\Phi^\prime$ and $\tilde{\chi}_\omega$ implies $\tilde{R}_\omega\leq0$, i.e. the constructed solution attains in fact  values in the interval $\ocinterval{-1}{0}$ and is hence not affected by the continuation of $\Phi$.
\par
\emph{\ul{Further estimates}}: %
For $\tilde{x}\in I_\omega$, the formulas \eqref{Lem:JordanAuxRes.Eqn4b},  \eqref{Thm:ExistenceNonlWaves.Eqn2} guarantee
\begin{align*}
\babs{\tilde{R}_\omega\at{\tilde{x}}-\breve{R}_\omega\at{\tilde{x}}}\leq   C    \bnorm{\tilde{R}_\omega-\breve{R}_\omega}_{*,\,\omega}  \bar{Y}\at{\tilde{x}}\leq  C D\,   \delta_\omega^{l+1}  \bar{Y}\at{\tilde{x}} \,,
\end{align*}
so in view of the pointwise bounds from \eqref{Thm:ExistenceNonlWaves.PEqn3} and Assumption~\ref{Ass.Pot} we establish the estimates
\begin{align*}
\abs{\breve{K}_\omega\at{\tilde{x}}-\tilde{K}_\omega\at{\tilde{x}}}\leq C\delta_\omega^{-1} \frac{\babs{\tilde{R}_\omega\at{\tilde{x}}-\breve{R}_\omega\at{\tilde{x}}}}{\tilde{Y}\at{\tilde{x}}^{m+2}}\leq C\delta_\omega^l \bar{Y}\at{\tilde{x}}^{-m-1}
\end{align*}
as well as
\begin{align*}
\abs{\breve{P}_\omega\at{\tilde{x}}-\tilde{P}_\omega\at{\tilde{x}}}\bar{Y}\at{\tilde{x}}\leq C\delta_\omega^{-1} \frac{\babs{\tilde{R}_\omega\at{\tilde{x}}-\breve{R}_\omega\at{\tilde{x}}}}{\tilde{Y}\at{\tilde{x}}^{m+3}}\bar{Y}\at{\tilde{x}}\leq C\delta_\omega^l \bar{Y}\at{\tilde{x}}^{-m-1}
\end{align*}
by Taylor arguments. The claims \eqref{Thm:ExistenceNonlWaves.Eqn3}$_1$ and \eqref{Thm:ExistenceNonlWaves.Eqn1}$_1$ now follow after integration with respect to $\tilde{x}\in I_\omega$  from \eqref{Lem:DefBreveR.EqnA}$_1$ and \eqref{Lem:DefBreveR.EqnC}$_1$ by the triangle inequality. For $\tilde{x}\not\in I_\omega$, the regularity of $\Phi$ on compact subsets of $\ocinterval{-1}{0}$  implies improved Taylor estimates. More precisely, we get
\begin{align*}
\bnorm{\at{1-\tilde{\chi}_\omega}\bat{\breve{K}_\omega-\tilde{K}_\omega}}_1\leq C\delta_\omega^{m+1}\bnorm{\tilde{R}_\omega-\breve{R}_\omega}_1\leq
C\delta_\omega^{m-1}\bnorm{\tilde{R}_\omega-\breve{R}_\omega}_{*,\,\omega}\leq C\delta_\omega^{m+l}
\end{align*}
as well as
\begin{align*}
\bnorm{\at{1-\tilde{\chi}_\omega}\bat{\breve{P}_\omega-\tilde{P}_\omega}}_\infty\leq C\delta_\omega^{m+2}\bnorm{\tilde{R}_\omega-\breve{R}_\omega}_\infty\leq
C\delta_\omega^{m+1}\bnorm{\tilde{R}_\omega-\breve{R}_\omega}_{*,\,\omega}\leq C\delta_\omega^{m+l+2}\,,
\end{align*}
so \eqref{Thm:ExistenceNonlWaves.Eqn3}$_2$ and \eqref{Thm:ExistenceNonlWaves.Eqn2}$_2$ can be deduced from
\eqref{Lem:DefBreveR.EqnA}$_2$ and \eqref{Lem:DefBreveR.EqnC}$_2$.
\end{proof}
\begin{figure}[ht!] %
\centering{ %
\includegraphics[width=0.95\textwidth]{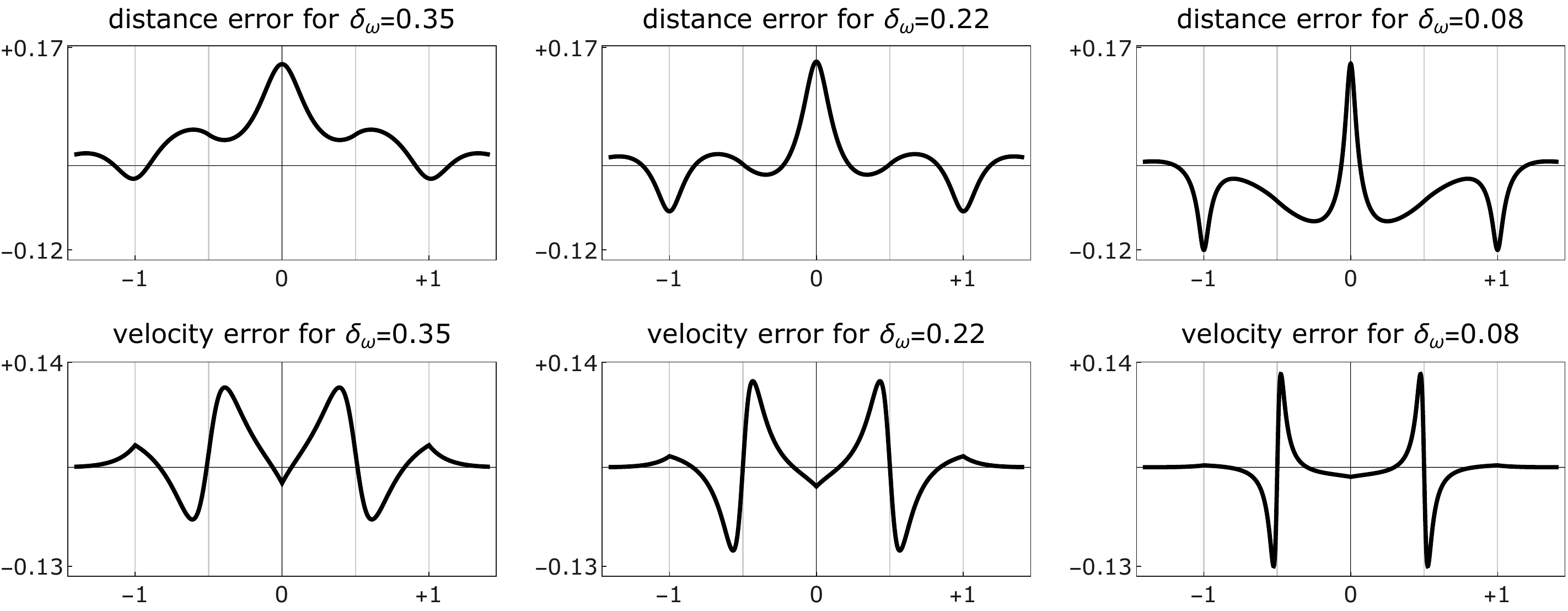}
} %
\caption{ %
Distance-approximation error $\delta_\omega^{-\min\{k,m\}-1}\nat{\tilde{R}_\omega-\breve{R}_\omega}$ (\emph{top row}) and
velocity-approximation error $\delta_\omega^{-\min\{k,m\}}\nat{\om^{-1}\tilde{V}_\omega-\om^{-1}\breve{V}_\omega}$
(\emph{bottom row}) plotted versus $x=\delta_\omega\be_\omega\tilde{x}$ for the waves from Figure~\ref{Fig:WavesDist}.   Recall that in our numerical simulations we first compute $\btriple{\omega}{\tilde{R}_\omega}{\tilde{V}_\omega}$ and
afterwards $(\breve{R}_\omega, \breve{V}_\omega)$ by Lemma~\ref{Lem:DefAlpha} and Lemma~\ref{Lem:DefBreveR}, see the discussion in \S\ref{sect:Intro}.
} %
\label{Fig:WavesErr} %
\end{figure} %
We conclude this section with some comments. First, Theorem \ref{Thm:ExistenceNonlWaves} does not guarantee that the exact wave profile $\tilde{R}_\om$ is unimodal although this properties might be deduced from a more refined analysis, see also the discussion in \cite{HM15}. Secondly, the error estimate \eqref{Thm:ExistenceNonlWaves.Eqn2} implies
\begin{align}
\label{Eqn:PredictedError}
\bnorm{\tilde{R}_\omega-\breve{R}_\omega}_\infty\leq C\delta_\omega^{\min\{k,m\}}\leq C\delta_\omega^{\min\{k,m\}}\bnorm{\breve{R}_\omega}_\infty
\end{align}
as well as
\begin{align*}
\bnorm{\tilde{R}_\omega-\breve{R}_\omega}_1\leq C\delta_\omega^{\min\{k,m\}-1}\leq C\delta_\omega^{\min\{k,m\}}\bnorm{\breve{R}_\omega}_1\,,
\end{align*}
i.e. both the $\fspaceL^\infty$-error and the $\fspaceL^1$-error with respect to the unscaled space variable $x$ are at least of order $\bDO{\delta_\omega^{\min\{k,\,m\}}}$. Moreover, since the formulas in Lemma \ref{Lem:DefBreveR} provide
\begin{align*}
\bnorm{\breve{R}_\omega-\tilde{R}_\infty}_\infty\leq C\delta_\omega\,,\qquad \bnorm{\breve{R}_\omega-\tilde{R}_\infty}_1\leq C\,,
\end{align*}
where $\tilde{R}_\infty$ denotes the rescaled tent map $R_\infty$ from \eqref{Eqn:LimitProfiles}, we conclude that $\tilde{R}_\om-\breve{R}_\om$ is asymptotically smaller than $\tilde{R}_\om-\tilde{R}_\infty$. This observation is also supported by the numerical data from Figure~\ref{Fig:WavesDist}. The simulation in
Figure \ref{Fig:WavesErr}, however, indicate that the real approximation error is even considerably smaller than the global prediction in \eqref{Eqn:PredictedError} and does not exceed the local one, which is given by
\begin{align*}
\bnorm{\tilde{R}_\omega-\breve{R}_\omega}_{\infty, \ccinterval{-\tilde{y}}{+\tilde{y}}}\leq \babs{\tilde{R}_\omega\at{0}-\breve{R}_\omega\at{0}}+C_{\tilde{y}}\bnorm{\tilde{R}_\omega^{\prime\prime}-\breve{R}_\omega^{\prime\prime}}_{1,\Rset}\leq C_{\tilde{y}}\delta_\omega^{\min\{k,m\}+1}
\end{align*}
for any compact set $\ccinterval{-\tilde{y}}{+\tilde{y}}$ independent of $\om$.
\par
Finally, for completeness we state the corresponding approximation result for the velocity profile and recall that the integral identity \eqref{Eqn:TWNonlIdentities.V2}$_1$ predicts the amplitude of both $V_\omega$ and $\tilde{V}_\omega$ to be of order $\DO{\omega}$.
\begin{lemma}[approximate velocity profile]
\label{Lem:ApproxVelProfile}
The even function
\begin{align*}
\breve{V}_\omega:= -\om\,\tilde{\chi}_\omega\ast
\at{\tilde{\chi}_\omega\,\al_\omega^{-m/2}\bar{K}}
\end{align*}
satisfies $\breve{R}_\omega=-\om^{-1}\,\tilde{\chi}_\omega\ast \bat{\delta_\omega\beta_\omega\breve{V}_\omega}$ as well as
\begin{align*}
\breve{V}_\omega\at{\tilde{x}}=-\tfrac12\,\om\,\al_\omega^{-m/2}\at{ \bar{Y}^\prime\bat{\tilde{x}-\xi_\omega}
-\bar{Y}^\prime\at{\xi_\omega}}
 \qquad \text{for}\quad 0\leq \tilde{x}\leq2\xi_\omega
\end{align*}
and $\breve{V}_\omega\at{\tilde{x}}=0$ for $\tilde{x}\geq 2\xi_\omega$. Moreover, we have
\begin{align*}
\delta_\omega \bnorm{\om^{-1}\tilde{V}_\omega-\om^{-1}\breve{V}_\omega}_1+\bnorm{
\om^{-1}\tilde{V}_\omega-\om^{-1}\breve{V}_\omega}_\infty+\bnorm{
\om^{-1}\tilde{V}_\omega^\prime-\om^{-1}\breve{V}_\omega^\prime}_1\leq C \delta_\omega^{\min\{k,m\}}
\end{align*}
for some constant independent of $\omega$, where $\tilde{V}_\omega$ is completely determined by $\tilde{R}_\omega$ via \eqref{Eqn:ScaledTWVelocities}.
\end{lemma}
\begin{proof}
Our definition implies $\supp \,\breve{V}_\omega=2 I_\omega$ as well as
\begin{align*}
\breve{V}_\omega\at{\tilde{x}}=-\om\,\al_\omega^{-m/2}\int\limits_{\tilde{x}-\xi_\omega}^{\xi_\omega}\bar{K}\at{\tilde{y}}\dint\tilde{y}=\tfrac12\,\om\, \al_\omega^{-m/2}\int\limits_{\tilde{x}-\xi_\omega}^{\xi_\omega}  \bar{Y}^{\prime\prime}\at{\tilde{y}}  \dint\tilde{y}\qquad \text{for}\quad  0\leq\tilde{x}\leq 2\xi_\omega\,,
\end{align*}
so the desired representation formula for $\breve{V}_\omega$ follows immediately. Moreover, by construction --- see \eqref{Eqn:DefDelta}+\eqref{Eqn:DefBeta}+\eqref{Eqn:ScaledTWVelocities}+\eqref{Eqn:FormalResultA} --- we have
\begin{align*}
\tilde{V}_\omega-\breve{V}_\omega=\om\,\tilde{\chi}_\omega\ast\bat{\al_{\omega}^{-\frac{m}{2}}\bat{\tilde{\chi}_\omega\bar{K}-\tilde{K}_\omega}}
\end{align*}
and Young's inequality for convolutions provides
\begin{align*}
\delta_\omega\bnorm{\tilde{V}_\omega-\breve{V}_\omega}_1+\bnorm{\tilde{V}_\omega-\breve{V}_\omega}_\infty+\bnorm{\tilde{V}_\omega^\prime-\breve{V}_\omega^\prime}_1\leq C\omega
\Bat{\delta_\omega\nnorm{\tilde{\chi}_\omega}_1+\nnorm{\tilde{\chi}_\omega}_\infty+\nnorm{\tilde{\chi}_\omega^\prime}_1}\bnorm{\tilde{K}_\omega-\tilde{\chi}_\omega\bar{K}}_1\,,
\end{align*}
where we used that $\tilde{\chi}_\omega^\prime$ is a sum of two signed Dirac distributions with bounded mass. The claimed error estimates are thus a byproduct of \eqref{Lem:DefBreveR.EqnA}.
\end{proof}
%
%
%
\subsection{Smooth parameter dependence}\label{sect:Smoothness}
%
In preparation   for   the spectral analysis in \S\ref{sect:Stability} we prove that the distance profile $R_\omega$ provided by Theorem~\ref{Thm:ExistenceNonlWaves} depends smoothly on $\om$, which is not immediately granted by abstract principles since the fixed point argument from \S\ref{sect:ExAndUni} works with $\om$-dependent norms and operators. More importantly, we   establish almost explicit   approximation formulas for $\partial_\omega R_\omega$ which allow us to compute the base functions for the neutral Jordan modes   up to   high accuracy.
\par
The starting point for our considerations is the identity
\begin{align}
\label{Eqn:ODEDistances}
\partial_\omega R_\omega=\chi\ast\chi\ast\Bat{\omega^{-2}\Phi^{\prime\prime}\bat{R_\omega}
\partial_\omega R_\omega-2\omega^{-3}\Phi^{\prime}\bat{R_\omega}}\,,
\end{align}
which follows by symbolically differentiating the integrated traveling wave equation \eqref{Eqn:NonlFPDist} with respect to $\om$. Due to the scaling \eqref{Eqn:ScaledTWEqnForDistances} and the definitions in \eqref{Eqn:DefTildeP}, this equation can be written as
\begin{align}
\label{Eqn:ODEDistances.Scaled}
\tilde{Q}_\omega=\tilde{\chi}_\omega\ast\tilde{\chi}_\omega\ast\at{\tilde{P}_\omega \tilde{Q}_\omega+\tilde{K}_\omega}\,,
\end{align}
where the unknown
\begin{align}
\label{Eqn:DefCheckQ}
  \tilde{Q}_\omega\at{\tilde{x}}=-\tfrac12\delta_\omega^{-m/2-1}\al_\omega^{-1}\partial_\omega R_\omega\bat{\delta_\omega \beta_\omega\tilde{x}}=
m^{-1}\al_\omega^{-1}\partial_{\delta_\omega} R_\omega\bat{\delta_\omega \beta_\omega\tilde{x}}\,.
\end{align}
can be expected to have amplitudes of order $\DO{1}$ and   represents   a scaled analogue to the derivative of $R_\omega$ with respect to the small quantity $\delta_\omega$.   Notice also that \eqref{Eqn:ODEDistances} is easier to analyze than the corresponding integral equation   for    $\partial_\omega \tilde{R}_\omega$ as the latter involves a convolution kernel which depends explicitly on $\om$.

\par
Our strategy is to prove the existence of $\partial_\omega R_\omega$ by showing that the nonlocal equation \eqref{Eqn:ODEDistances.Scaled} admits a unique solution, and to approximate the latter by a function $\breve{Q}_\omega$, which is well defined by
\begin{align}
\label{Eqn:DefCheckU}
\breve{Q}_\omega=\tilde{\chi}_\omega\ast\tilde{\chi}_\omega\ast \nat{\tilde{\chi}_\omega\bar{P}\,\breve{Q}_\omega+\tilde{\chi}_\omega\bar{K}}
\end{align}
and can be computed explicitly, see the proof of Lemma~\ref{Lem:AsympCheckU}.
\begin{figure}[ht!] %
\centering{ %
\includegraphics[width=0.95\textwidth]{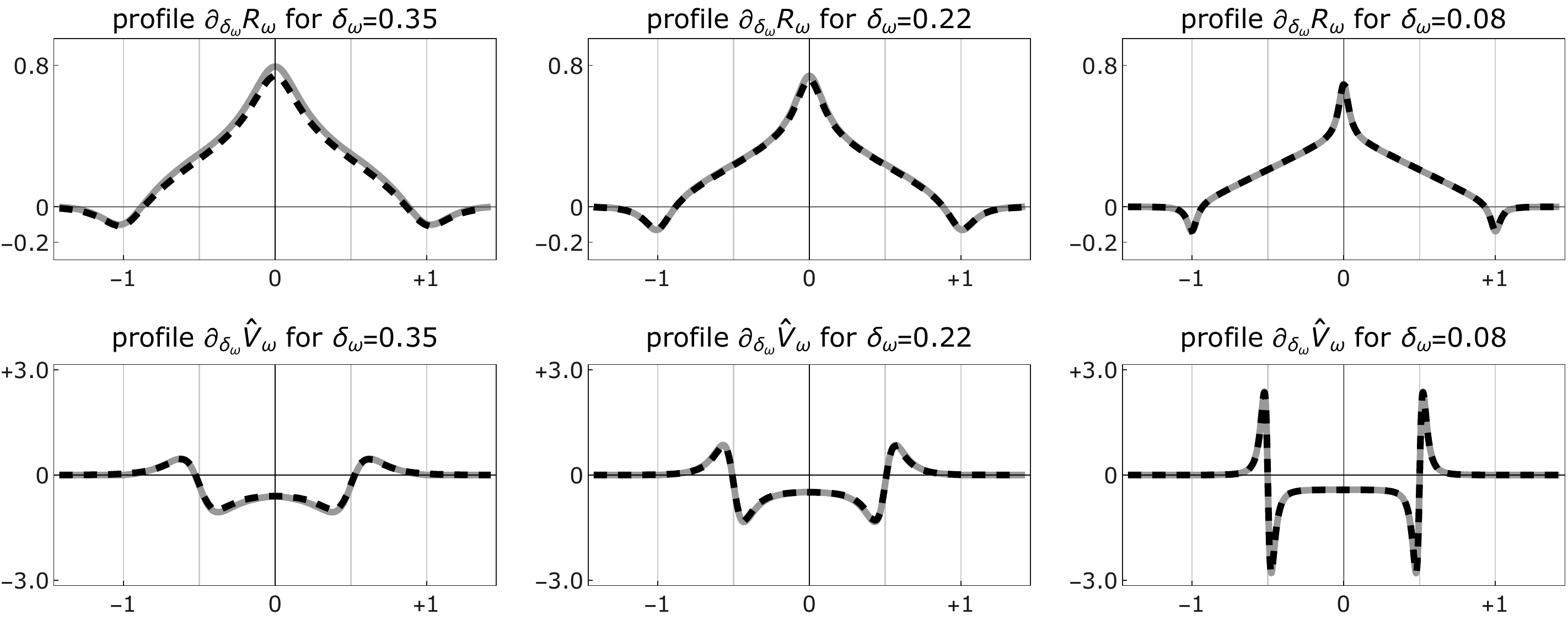}
} %
\caption{ %
  Derivative of the distance profile $R_\omega$ (\emph{top row}) and the normalized velocity profile $\hat{V}_\omega=\om^{-1}V_\omega$ (\emph{bottom row}) with respect to $\delta_\omega$ computed with the numerical data (black, dashed)  from Figures~\ref{Fig:WavesDist} and~\ref{Fig:WavesVel}, and using the approximate formulas (gray, solid) from Lemmas~\ref{Lem:DefAlpha} and~\ref{Lem:DefBreveR}. The corresponding rigorous results are stated in Theorem~\ref{Thm:SmoothnessNonlWaves} and Lemma~\ref{Lem:AsympCheckU}.
} %
\label{Fig:Derivatives} %
\end{figure} %
\begin{theorem}[smooth parameter dependence]
\label{Thm:SmoothnessNonlWaves}
  The family $\nat{\tilde{R}_\omega}_\omega$ from Theorem~\ref{Thm:ExistenceNonlWaves} is $\fspaceC^2$-smooth in the sense that the derivatives
\begin{align*}
\partial_{\tilde{x}}\tilde{R}_\omega,\qquad \partial_{\tilde{\omega}}\tilde{R}_\omega,\,\qquad
\partial_{\tilde{x}}^2\tilde{R}_\omega
,\,\qquad
\partial_{\tilde{x}}\partial_\omega\tilde{R}_\omega
,\,\qquad
\partial_\omega\partial_{\tilde{x}}\tilde{R}_\omega
,\,\qquad
\partial_{\omega}^2\tilde{R}_\omega
\end{align*}
exist for any sufficiently large $\om$ as continuous functions in the variables $\pair{\omega}{\tilde{x}}$, and a similar statement holds for the corresponding family  $\nat{R_\omega}_\omega$ that is provided by the scaling \eqref{Eqn:ScaledProfiles}. Moreover, we have
\begin{align}
\label{Thm:SmoothnessNonlWaves.Eqn1}
\bnorm{\tilde{Q}_\omega-\breve{Q}_\omega}_{*,\,\omega}\leq C\delta_\omega^{\min\{k,m\}}
\end{align}
for all sufficiently large $\omega$ and a constant $C$ independent of $\omega$.
\end{theorem}
\begin{proof}
\emph{\ul{Linearized traveling wave problem}}: We first show that for any given $\tilde{F}\in\fspaceL^1_\even$ there exists a unique solution $\tilde{U}\in\fspaceL^1_\even$ to the linear problem
\begin{align}
\label{Thm:SmoothnessNonlWaves.PEqn2}
\tilde{U}=\tilde{\chi}_\omega\ast\tilde{\chi}_\omega\ast\bat{\tilde{P}_\omega\tilde{U}+\tilde{F}}\,,
\end{align}
which can equivalently be written as
\begin{align}
\label{Thm:SmoothnessNonlWaves.PEqn0}
\tilde{U}=  \tilde{\calK}_\omega  \Bato{\tilde{\chi}_\omega\bat{\tilde{P}_\omega-\bar{P}}\tilde{U}+\at{1-\tilde{\chi}_\omega}\tilde{P}_\omega\tilde{U}+\tilde{F}}
\end{align}
with   $\tilde{\calK}_\omega$   being the solution operator to the auxiliary problem \eqref{Lem:JordanAuxRes.Eqn1}. In view of Lemma~\ref{Lem:JordanAuxRes} we conclude that any solution $\tilde{U}$ necessarily belongs to the Banach space
\begin{align*}
\tilde{\fspace{X}}_\omega := \big\{\tilde{U}\;:\;\bnorm{\tilde{U}}_{*,\,\omega}<\infty\big\}\,,
\end{align*}
and combining \eqref{Lem:JordanAuxRes.Eqn2} with \eqref{Thm:ExistenceNonlWaves.Eqn1} yields
\begin{align*}
\bnorm{\tilde{\chi}_\omega\bat{\tilde{P}_\omega-\bar{P}}\tilde{U}}_1 \leq
\bnorm{\tilde{\chi}_\omega\bat{\tilde{P}_\omega-\bar{P}}\bar{Y}}_1\bnorm{\tilde{U}\bar{Y}^{-1}}_{\infty}  \leq C
\bnorm{\tilde{\chi}_\omega\bat{\tilde{P}_\omega-\bar{P}}\bar{Y}}_1\bnorm{\tilde{U}}_{*,\,\omega} \leq \delta_\omega^{\min\{k,m\}}C\bnorm{\tilde{U}}_{*,\,\omega}
\end{align*}
thanks to \eqref{Lem:JordanAuxRes.Eqn4b} as well as
\begin{align*}
\bnorm{\at{1-\tilde{\chi}_\omega}\tilde{P}_\omega\tilde{U}}_1\leq
 \bnorm{\at{1-\tilde{\chi}_\omega}\tilde{P}_\omega}_\infty  \bnorm{\tilde{U}}_1  \leq C\delta_\omega^{m}\bnorm{\tilde{U}}_{*,\,\omega}\,.
\end{align*}
We thus conclude --- at least for sufficiently large $\omega$ --- that the right hand side in \eqref{Thm:SmoothnessNonlWaves.PEqn0} defines an affine operator that maps  $\tilde{\fspace{X}}_\omega$ contractively into itself.  The existence and uniqueness of solutions is therefore granted by Banach's Contraction Mapping Principle, and we readily verify the continuity estimate
\begin{align}
\label{Thm:SmoothnessNonlWaves.PEqn1}
\bnorm{\tilde{U}}_{*,\,\omega}\leq C \bnorm{\tilde{F}}_1
\end{align}
for some constant $C$ independent of $\omega$.   We also infer from the definition of the norm $\norm{\cdot}_{*,\,\omega}$ --- see \eqref{Lem:JordanAuxRes.Eqn2} and \eqref{Lem:JordanAuxRes.Eqn4a} --- that $\nnorm{\tilde{U}}_{*,\,\omega}<\infty$ implies that $\tilde{U}$ is continuous with respect to the space variable $\tilde{x}$.
\par
\emph{\ul{ODE for $R_\omega$}}: The identities \eqref{Eqn:ODEDistances.Scaled} and \eqref{Eqn:DefCheckQ} can be written as
\begin{align}
\label{Thm:SmoothnessNonlWaves.PEqn3}
\partial_\omega R_\omega =-2\delta_\omega^{m/2+1}\al_\omega\tilde{\calS}_\omega^{-1}\ato{\tilde{\calM}_\omega\bato{\delta_\omega^{m+1}\al_\omega^{m+1}\Phi^{\prime}\bat{\tilde{\calS}_\omega\ato{R_\omega}}}}\,,
\end{align}
where $\tilde{\calM}_\omega$ denotes the solution operator to \eqref{Thm:SmoothnessNonlWaves.PEqn2} and $\tilde{\calS}_\omega$ abbreviates the scaling $x\mapsto\tilde{x}$, see \eqref{Eqn:SpaceScalign}. Moreover, the superposition operator corresponding to   $\Phi^{\prime}$ is Lipschitz continuous   on the open set
\begin{align*}
\big\{\tilde{R}_\omega\in \tilde{\fspace{X}}_\omega\;:\; \min_{\tilde{x}\in\Rset}{\tilde{R}_\omega\at{\tilde{x}}}>-1\big\}\,,
\end{align*}
where we tacitly assumed --- as in the proof of Theorem~\ref{Thm:ExistenceNonlWaves} --- that $\Phi^\prime$ is trivially extended to the interval $\oointerval{-1}{\infty}$ in order to ensure well-definedness. The initial value problem corresponding to the Banach-valued ODE \eqref{Thm:SmoothnessNonlWaves.PEqn3} is hence locally well-posed and gives rise to a   family   of smooth traveling waves. Due to the local uniqueness we finally conclude that the $\omega$-family from Theorem~\ref{Thm:ExistenceNonlWaves} is in fact   continuously differentiable with respect to $\om$ and that $\partial_\omega R_\omega$ depends continuously on both $x$ and $\om$. Moreover,  differentiating \eqref{Eqn:ODEDistances} with respect to $\om$ and inverting the linear main part by $\tilde{\calM}_\omega$ gives an ODE for $\partial_\omega R_\omega$, which involves --- among other terms --- the  functions $\Phi^{\prime\prime}$ and $\Phi^{\prime\prime\prime}$, but there is no regularity issue as these functions are only evaluated at points $R_\omega\at{x}\in\ocinterval{-1}{0}$. From this we conclude that $\partial_\omega^2 R_\omega$ is well-defined and depends continuously on both $\om$ and $x$. Moreover, the spatial regularization of the convolution with the tent map $\chi\ast\chi$ combined with the integral equations for $R_\omega$ and $\partial_\omega R_\omega$  guarantees that the partial derivatives $\partial_x R_\omega$, $\partial_x\partial_\omega R_\omega$, $\partial_\omega \partial_x R_\omega$, and $\partial_{x}^2 R_\omega$ are all well-defined and continuous with respect to the variables $\pair{\omega}{x}$. Finally, using \eqref{Eqn:ScaledProfiles} we conclude that  $\tilde{R}_\omega=\tilde{\calS}_\omega\ato{R_\omega}$  has also the desired regularity properties.

\par
\emph{\ul{Approximation error}}: By construction,
we have
\begin{align*}
\breve{Q}_\omega =   \tilde{\calK}_\omega   \bato{\tilde{\chi}_\omega\bar{K}}\,,
\qquad
\tilde{Q}_\omega=\tilde{\calM}_\omega\bato{\tilde{K}_\omega}=  \tilde{\calK}_\omega  \bato{\tilde{\chi}_\omega\bar{K}+\tilde{E}_\omega}
\end{align*}
with error terms
\begin{align*}
\tilde{E}_\omega=\tilde{\chi}_\omega\bat{\tilde{P}_\omega-\bar{P}}\tilde{Q}_\omega+\at{1-\tilde{\chi}_\omega}\tilde{P}_\omega\tilde{Q}_\omega+\tilde{\chi}_\omega\bat{\tilde{K}_\omega-\bar{K}}+\at{1-\tilde{\chi}_\omega}\tilde{K}_\omega\,,
\end{align*}
and the continuity of   $\tilde{\calK}_\omega$   and $\tilde{\calM}_\omega$ --- see \eqref{Lem:JordanAuxRes.Eqn1} and \eqref{Thm:SmoothnessNonlWaves.PEqn1} --- imply
\begin{align}
\label{Thm:SmoothnessNonlWaves.PEqn4}
\bnorm{\breve{Q}_\omega}_{*,\,\omega}\leq C \bnorm{\tilde{\chi}_\omega\bar{K}}_1\leq C\,,\qquad
\bnorm{\tilde{Q}_\omega}_{*,\,\omega} \leq
C\bnorm{\tilde{\chi}_\omega\bar{K}}_1+C\bnorm{\tilde{E}_\omega}_1\leq C\bat{1+\bnorm{\tilde{E}_\omega}_1}
\end{align}
as well as
\begin{align}
\label{Thm:SmoothnessNonlWaves.PEqn5}
\bnorm{\tilde{Q}_\omega-\breve{Q}_\omega}_{*,\,\omega}\leq C\bnorm{\tilde{E}_\omega}_1\,.
\end{align}
On the other hand, by \eqref{Lem:JordanAuxRes.Eqn2}, \eqref{Lem:JordanAuxRes.Eqn4b} and H\"older arguments we estimate
\begin{align*}
\bnorm{\tilde{E}_\omega}_{1}\leq C\at{\bnorm{\tilde{\chi}_\omega\bat{\tilde{P}_\omega-\bar{P}}\tilde{Y}_\omega}_1+\bnorm{\at{1-\tilde{\chi}_\omega}\tilde{P}_\omega}_\infty \delta_\omega^{-2}}\bnorm{\tilde{Q}_\omega}_{*,\,\omega}+
\bnorm{\tilde{\chi}_\omega\bat{\tilde{K}_\omega-\bar{K}_\omega}+\at{1-\tilde{\chi}_\omega}\tilde{K}_\omega}_1
\end{align*}
and inserting \eqref{Thm:SmoothnessNonlWaves.PEqn4} as well as the error estimates from Theorem~\ref{Thm:ExistenceNonlWaves} we first obtain
\begin{align*}
\bnorm{\tilde{E}_\omega}_{1}\leq C\bat{\delta_\omega^{\min\{k,m\}}+\delta_\omega^{m}}\bat{1+\bnorm{\tilde{E}_\omega}_{1}}+C\delta_\omega^{m}
\end{align*}
and afterwards the claim \eqref{Thm:SmoothnessNonlWaves.Eqn1} thanks to
\eqref{Thm:SmoothnessNonlWaves.PEqn5}.
\end{proof}
  The regularity   statements in Theorem~\ref{Thm:SmoothnessNonlWaves} are not optimal but sufficient for our purposes as they guarantee the $\fspaceC^2$-regularity of the map $\pair{\omega}{x}\to R_\omega\at{x}$ and hence also the Schwarz identity $\partial_x\partial_\omega R_\omega=\partial_\omega \partial_x R_\omega$. By boot strapping one easily shows that this maps actually admits infinitely many derivatives provided that $\Phi$ does so on the interval $\ocinterval{-1}{0}$.
\par
  We also remark that the error bounds for $\bnorm{\tilde{Q}_\omega-\breve{Q}_\omega}_{*,\,\omega}$ and
$\nnorm{\tilde{R}_\omega-\breve{R}_\omega}_{*,\,\omega}$ differ by one power of $\delta_\omega$ because \eqref{Thm:SmoothnessNonlWaves.Eqn1} and  \eqref{Lem:JordanAuxRes.Eqn2} give
\begin{align*}
\bnorm{\tilde{Q}_\omega-\breve{Q}_\omega}_\infty\leq \delta_\omega^{\min\{k,m\}-1}\leq  \delta_\omega^{\min\{k,m\}-1}\bnorm{\breve{Q}_\omega}_\infty
\end{align*}
as well as
\begin{align*}
\bnorm{\tilde{Q}_\omega-\breve{Q}_\omega}_1\leq \delta_\omega^{\min\{k,m\}-2}\leq  \delta_\omega^{\min\{k,m\}-1}\bnorm{\breve{Q}_\omega}_1\,.
\end{align*}
To conclude our study of the parameter dependence   we finally analyze the defining equation for the approximation $\breve{Q}_\omega$.  In particular,   we show that the latter is completely determined by $\bar{Y}$, see \eqref{Lem:AsympCheckU.PEqn2}.
\begin{lemma}[properties of the approximation $\breve{Q}_\omega$]
\label{Lem:AsympCheckU}
The unique solution $\breve{Q}_\omega$ to \eqref{Eqn:DefCheckU} satisfies
\begin{align*}
\breve{Q}_\omega\quad\xrightarrow{\quad \om\to\infty\quad}\quad-m^{-1}\bar{Y}^\flat
\end{align*}
in the sense of locally uniform convergence. Moreover, we have
\begin{align}
\label{Lem:AsympCheckU.Eqn1}
\babs{m\,\breve{Q}_\omega\at{0}-1}+\babs{\breve{Q}^\prime_\omega\nat{\xi_\omega}}\leq C\delta_\omega
\end{align}
and
\begin{align}
\label{Lem:AsympCheckU.Eqn2}
  \bnorm{\breve{Q}_\omega}_\infty + \bnorm{\breve{Q}_\omega^\prime}_\infty +\delta_\omega \bnorm{\breve{Q}_\omega}_1 \leq C
\end{align}
for some constant $C$ independent of $\omega$.
\end{lemma}
\begin{proof}
Lemma~\ref{Lem:JordanAuxRes} ensures that $\breve{Q}_\omega$ is well defined with
\begin{align}
\label{Lem:AsympCheckU.PEqn1}
\xi_\omega\breve{Q}_\omega^\prime\nat{\xi_\omega}+\breve{Q}_\omega\nat{\xi_\omega}=0\,,
\end{align}
and differentiating \eqref{Eqn:DefCheckU} twice with respect to $\tilde{x}$ reveals
\begin{align*}
\breve{Q}^{\prime\prime}_\omega\at{\tilde{x}}=-2 \bar{P}\at{\tilde{x}}\breve{Q}_\omega\at{\tilde{x}}-2\bar{K}\at{\tilde{x}}\qquad\text{for}\quad\tilde{x}\in I_\omega\,.
\end{align*}
On the other hand, the asymptotic shape ODE \eqref{Lem:AsympODE.Props.Eqn1} can also be written as
\begin{align*}
\bar{Y}^{\prime\prime}\at{\tilde{x}}=-2\bar{P}\at{\tilde{x}}\bar{Y}\at{\tilde{x}}-\at{m+2}\bar{K}\at{\tilde{x}}\,,
\end{align*}
see \eqref{Eqn:DefBarPZ}, and in view of Lemma~\ref{Lem:LinODE.Props} we conclude that the even function $\at{m+2}\breve{Q}_\omega-\bar{Y}$ is a multiple of $\bar{T}_\even$.  In other words, there exists a constant $c$ such that
\begin{align}
\label{Lem:AsympCheckU.PEqn2}
\bat{m+2}\breve{Q}_\omega\at{\tilde{x}}=\at{1+m c}
\bar{Y}\at{\tilde{x}}+\at{m+2} c\bar{Y}^\flat\at{\tilde{x}}=
  \at{1+m\, c}
\tilde{x}\bar{Y}^\prime\at{\tilde{x}}+\at{2\, c-1}\bar{Y}^\flat\at{\tilde{x}}
\end{align}
holds for $\tilde{x}\in I_\delta$. We therefore have
\begin{align*}
\bat{m+2}\breve{Q}_\omega\at{\xi_\omega}
=
\at{1+m\,c}
\xi_\omega\bar{Y}^\prime\at{\xi_\omega}+\at{2\, c-1}\bar{Y}^\flat\at{\xi_\omega}
\end{align*}
as well as
\begin{align*}
\bat{m+2}\xi_\omega\breve{Q}_\omega^\prime\bat{\xi_\omega}=\at{1+m\,c}\xi_\omega\bar{Y}^\prime\bat{\xi_\omega}+\at{m+2}\,c\,\xi_\omega^{\,2}\bar{Y}^{\prime\prime}\bat{\xi_\omega}\,,
\end{align*}
and inserting this into \eqref{Lem:AsympCheckU.PEqn1} we arrive at
\begin{align*}
0=2\at{1+m\,c}\xi_\omega\bar{Y}^\prime\bat{\xi_\omega}+\at{m+2}\,c\,\xi_\omega^{\,2}\bar{Y}^{\prime\prime}\bat{\xi_\omega}+\at{2\,c}\bar{Y}^\flat\nat{\xi_\omega}\,.
\end{align*}
The asymptotic properties of $\bar{Y}$ --- see Lemma~\ref{Lem:AsympODE.Props} --- imply via
\begin{align*}
0=2\at{1+m\,c}\xi_\omega\bar{Y}^\prime\bat{\infty}-\at{-1+2\,c}\bar{Y}^\flat\nat{\infty}+\DO{\delta_\omega^{m-1}}
\end{align*}
the expansions
\begin{align}
\label{Lem:AsympCheckU.PEqn3}
1+m\,c=\DO{\delta_\omega}\,,c=-\frac{1}{m}+\DO{\delta_\omega}\,,
\end{align}
so using \eqref{Lem:AsympCheckU.PEqn2} we readily verify \eqref{Lem:AsympCheckU.Eqn1} as well as
\begin{align*}
\nnorm{\tilde{\chi}_\omega \breve{Q}_\omega}_\infty+\nnorm{\tilde{\chi}_\omega \breve{Q}_\omega^\prime}_\infty\leq C\,.
\end{align*}
The latter estimate also gives   \eqref{Lem:AsympCheckU.Eqn2} after evaluating \eqref{Lem:JordanAuxRes.Eqn3a} and \eqref{Lem:JordanAuxRes.Eqn3b} with $\tilde{U}=\breve{Q}_\omega$ and noticing that the compact support of $\breve{Q}_\omega$ implies $\nnorm{\breve{Q}_\omega}_1\leq C\delta_\omega^{-1}\nnorm{\breve{Q}_\omega}_\infty$. Finally, the locally uniform convergence claim is a direct consequence of
\eqref{Lem:AsympCheckU.PEqn2} and \eqref{Lem:AsympCheckU.PEqn3}.
\end{proof}
The proof of Lemma~\ref{Lem:AsympCheckU}, see especially \eqref{Lem:AsympCheckU.PEqn2} and \eqref{Lem:AsympCheckU.PEqn3}, reveals that  $\partial_{\delta_\omega}R_\omega$ is a genuine two-scale function in the sense that a tent-shaped background profile is superimposed by three decorations having amplitude of order $\DO{1}$ but width of order $\DO{\delta_\omega}$. This observation is confirmed by the numerical data in Figure~\ref{Fig:Derivatives}, which also illustrate that the decorations near $x=\pm1$ are --- due to the discrete Laplacian in \eqref{Eqn:NonlADDE2Order} --- shifted and scaled variants of the decorations near $x=0$.
%
%
\subsection{Tail estimates}
\label{sect:Tails}
%

%
Our final goal in this chapter is to establish the exponential space localization of the distance profile and its parameter derivative. For later use in \S\ref{sect:Stability} it is convenient to base this analysis on $R_\omega$ instead of $\tilde{R}_\omega$, and to work with
\begin{align}
\label{Eqn:DefQ}
Q_\omega\at{x}=\tilde{Q}_\omega\bat{\delta_\omega^{-1}\beta_\omega^{-1}x }=
m \,\al_\omega \,\partial_{\delta_\omega} R_\omega\at{x}
\end{align}
instead of $\partial_\omega R_\omega$, see \eqref{Eqn:DefCheckQ}.
\par
We first derive some global bounds which are consistent with  the heuristic rules \eqref{Eqn:DefTildeY} and $\tilde{Y}_\omega\approx \bar{Y}$, and show afterwards for large $\om$ that both $R_\omega$ and $Q_\omega$ decay rapidly as $x\to\pm\infty$.
\begin{lemma}[global estimates for $R_\omega$ and $Q_\omega$]
\label{Lem:GlobalBounds}
We have
\begin{align*}
\bnorm{R_\omega}_\infty+ \bnorm{\partial_x R_\omega}_\infty+\bnorm{Q_\omega}_\infty+ \delta_\omega \bnorm{\partial_x Q_\omega}_\infty\leq C
\end{align*}
for some constant $C$ independent of $\om$.
\end{lemma}
\begin{proof}
The approximation results from Theorem~\ref{Thm:ExistenceNonlWaves} combined with the definition of $\nnorm{\cdot}_{*,\,\omega}$ --- see \eqref{Lem:JordanAuxRes.Eqn2} and \eqref{Lem:JordanAuxRes.Eqn4a} --- ensure
\begin{align*}
\bnorm{\tilde{R}_\omega - \breve{R}_\omega}_\infty\leq C\delta_\omega ^{l-1}\,,\qquad
\bnorm{\partial_{\tilde{x}}\tilde{R}_\omega - \partial_{\tilde{x}}\breve{R}_\omega}_\infty\leq C\delta_\omega ^l
\end{align*}
with $l=\min\{k,m\}$, while Lemma~\ref{Lem:DefBreveR} provides
\begin{align*}
\bnorm{\breve{R}_\omega}_\infty \leq C\,,\qquad
\bnorm{\partial_{\tilde{x}}\breve{R}_\omega}_\infty \leq C\delta_\omega\,.
\end{align*}
The   claim   for $R_\omega$ is thus granted by the scaling ansatz \eqref{Eqn:ScaledProfiles} and Lemma~\ref{Lem:DefAlpha} since the spatial scaling \eqref{Eqn:SpaceScalign} implies  $\partial_{\tilde{x}}=\delta_\omega\beta_\omega\partial_x$. Similarly, the claim about $Q_\omega$ follows from \eqref{Thm:SmoothnessNonlWaves.Eqn1} and \eqref{Lem:AsympCheckU.Eqn2}.
\end{proof}
\begin{lemma}[rapid exponential decay with respect to $x$]
\label{Lem:TailEstimates}
For any given $a>0$ and all sufficiently large $\omega$ we have
\begin{align}
\label{Lem:TailEstimates.Eqn1}
\babs{R_\omega\at{x}}+
\babs{\partial_{x}{R}_\omega\at{x}}\leq
C\delta_\omega^{m} \exp\at{-a\,\nabs{x}}\qquad
\text{for}\quad \abs{x}\geq \tfrac32
\end{align}
as well as
\begin{align}
\label{Lem:TailEstimates.Eqn2}
  \babs{Q_\omega\at{x}}+\babs{\partial_x Q_\omega\at{x}}\leq
C\delta_\omega^{m-1}\at{1+\abs{x}}\exp\at{-a\abs{x}}\qquad
\text{for}\quad \abs{x}\geq \tfrac32
\end{align}
for some constant $C$ which depends on $a$ and $\Phi$ but not on $\omega$.
\end{lemma}
\begin{proof}\emph{\ul{Preliminaries}}:
Within this proof we set
\begin{align}
\label{Lem:TailEstimates.PEqn0}
D := \sup_{-3/4\leq r\leq 0}\abs{ \Phi^{\prime\prime}\at{r}}
\end{align}
and notice that this implies $\abs{\Phi^\prime\at{r}}\leq D\abs{r}$ for all $r\in\ccinterval{-3/4}{0}$ thanks to the Mean Value Theorem and $\Phi^\prime\at{0}=0$. Moreover, for given $a>0$ we can always guarantee the estimate
\begin{align*}
 D\delta_\omega^m\leq \exp\at{-a}
\end{align*}
by choosing $\om$ sufficiently large.
\par
\ul{\emph{Tightness of ${R}_\omega$}}: Since ${R}_\omega$ is even, it suffices to consider ${x}>0$. The approximation results from Lemma~\ref{Lem:DefBreveR} and Theorem~\ref{Thm:ExistenceNonlWaves} ensure
\begin{align}
\label{Lem:TailEstimates.PEqn1} %
-\tfrac34\leq {R}_\omega\at{x}\leq 0
\qquad\text{for}\quad{x}\geq \tfrac12\,,
\end{align}
so using
\begin{align*}
\supp\,\nat{\chi\ast \chi}\subseteq \ccinterval{-1}{+1}\,,\qquad \bnorm{\chi\ast \chi}_1=1
\end{align*}
we infer from the integral equation \eqref{Eqn:NonlFPDist} and \eqref{Lem:TailEstimates.PEqn1} the estimate
\begin{align}
\label{Lem:TailEstimates.PEqn2}
 \babs{{R}_\omega\at{x}}\leq \delta_\omega^m\sup\limits_{{y}\in\ccinterval{{x}-1}{{x}+1}}\Phi^\prime\bat{R_\omega\at{y}} \leq D\delta_\omega^m\sup\limits_{{y}\in\ccinterval{{x}-1}{{x}+1}}
\babs{{R}_\omega\at{{y}}}\,.
\end{align}
Setting
\begin{align*}
r_{\omega,\,j}:=\sup_{x\geq 1/2+j}\babs{{R}_\omega\at{x}}
\end{align*}
the estimate \eqref{Lem:TailEstimates.PEqn2} implies $r_{w,\,j+1}\leq  D \delta_\omega^m r_{\om,\,j}$ and we easily verify
\begin{align}
\label{Lem:TailEstimates.PEqn3}
r_{\om\,,j}\leq \at{D\delta_\omega^m}^j r_{\om,\,0}\leq  C \at{D \delta_\omega^m}^j
\qquad\text{for}\quad j\geq0
\end{align}
by induction over $j$, where the constant $C$ is provided by Lemma~\ref{Lem:GlobalBounds}. For $x>\tfrac32$, this result can be written as
\begin{align*}
\abs{R_\omega\at{x}}\leq r_{\om,\, 1+\floor{x-3/2}}\leq C D\delta_\omega^m \exp\at{-a\floor{x-\tfrac32}}\leq C\exp\at{\tfrac52a} D\delta_\omega^m\exp\at{-ax}\,,
\end{align*}
where $\floor{\cdot}$ denotes the floor variant of the integer part, and yields the desired decay estimate for ${R}_\omega$. Moreover, the integral equation \eqref{Eqn:NonlFPDist} also gives
\begin{align*}
\partial_{x}{R}_\omega = \bat{\partial_{x}\nat{\chi\ast \chi}}\ast\Bat{\om^{-2}\Phi^\prime\bat{{R}_\omega}}\,,
\end{align*}
and due to
\begin{align*}
\bnorm{\partial_{x}\at{\chi\ast\chi}}_\infty=1\,,\qquad \supp\, \bat{\partial_{x}\at{\chi\ast\chi}}\subset\ccinterval{-1}{+1}
\end{align*}
we obtain
\begin{align*}
\babs{\partial_{x}{R}_\omega\at{x}}\leq D\delta_\omega^{m}\int\limits_{x-1}^{x+1}
\abs{{R}_\omega\at{{y}}}\dint{{y}}\,.
\end{align*}
For ${x}\geq \tfrac32 $ we can therefore estimate
\begin{align*}
\babs{\partial_{x}{R}_\omega\at{x}}\leq C D\delta_\omega
\qquad
\end{align*}
thanks to Lemma~\ref{Lem:GlobalBounds}, while for $x>5/2$ we can employ the exponential decay of $R_\omega$ to justify the improved bound
\begin{align*}
\babs{\partial_{x}{R}_\omega\at{x}}\leq CD\delta_\omega^m \int\limits_{x-1}^{+\infty}\exp\at{-ay}\dint{y}\leq C D a^{-1}\delta_\omega^m\exp\at{a}\exp\at{-ax}\,.
\qquad
\end{align*}
The combination of the latter two results yield the decay claim for $\partial_x{R}_\omega$ and hence \eqref{Lem:TailEstimates.Eqn1}.
\par
\emph{\ul{Tightness of ${Q}_\omega$}}:
The nonlocal equation for $\tilde{Q}_\omega$ --- see \eqref{Eqn:ODEDistances}, \eqref{Eqn:DefCheckQ}, and \eqref{Eqn:DefQ} --- transforms into
\begin{align*}
Q_\omega=\chi \ast \chi\ast\at{\om^{-2}\Phi^{\prime\prime}\at{R_\omega}Q_\omega + F_\omega}\,,
\end{align*}
where $F_\omega$ is given in terms of $R_\omega$ and satisfies
\begin{align}
\label{Lem:TailEstimates.PEqn4}
\babs{F_\omega\at{x}}\leq C\delta_\omega^{m-1}\abs{\Phi^\prime\bat{R_\omega\at{x}}}\leq CD\delta_\omega^{m-1} \abs{R_\omega\at{x}}\qquad \text{for}\quad x\geq 1/2
\end{align}
due to our choice of $D$ in \eqref{Lem:TailEstimates.PEqn0}. We define
\begin{align*}
q_{\omega,\,j}:=\sup_{ {x}\geq 1/2+j}\babs{{Q}_\omega\at{x}}\,,\qquad
f_{\omega,\,j}:=\sup_{ {x}\geq 1/2+j}\babs{{F}_\omega\at{x}}
\end{align*}
and derive --- in analogy to the discussion above --- first the recursion
\begin{align*}
q_{\om,\, j+1}\leq \at{D \delta_\omega^m}^j \bat{q_{\om,\,j}+ f_{\om,\,j}}
\end{align*}
and afterwards the discrete Duhamel estimate
\begin{align*}
q_{\om,\, j}\leq \at{D\delta_\omega^m}^jq_{\om,\,0}+\sum_{i=0}^{j-1}\at{D\delta_\omega^{m}}^{j-1-i} f_{\om,\, i}\,.
\end{align*}
Inspecting \eqref{Lem:TailEstimates.PEqn3} and \eqref{Lem:TailEstimates.PEqn4} we verify
\begin{align*}
f_{\om,\,i}\leq C D\delta_\omega^{m-1} r_{\om,i}\leq C\delta_\omega^{-1}\at{D\delta_\omega^m}^{i+1}\,,
\end{align*}
which implies
\begin{align*}
q_{\om,\, j}\leq \at{D\delta_\omega^m}^j\at{q_{\om,\,0}+C\delta_\omega^{-1}j}\qquad \text{for}\quad j\geq0
\end{align*}
and hence the desired decay rate for $Q_\omega$ since Lemma~\ref{Lem:GlobalBounds} provides $q_{\om,\,0}\leq C$. Finally, writing
\begin{align*}
\partial_x Q_\omega=\bat{\partial_x\at{\chi_\omega \ast \chi_\omega}}\ast\at{\om^{-2}\Phi^{\prime\prime}\at{R_\omega}Q_\omega + F_\omega}
\end{align*}
we can argue as above to control the decay of $\partial_x Q_\omega$.

\end{proof}

We emphasize that the decay estimates in Lemma~\ref{Lem:TailEstimates} are not optimal. In fact, inserting the exponential ansatz $R_\omega\at{x}=C_\omega\exp\at{-a_\omega\abs{x}}$ as $\abs{x}\to\infty$ into the second-order advance-delay-differential equation \eqref{Eqn:NonlADDE2Order}, we obtain after Taylor expanding $\Phi^\prime$ the transcendental tail identity
\begin{align}
\label{Eqn:DecayEquation}
\om^2 = \frac{\Phi^{\prime\prime}\at{0}\at{\mhexp{+a_\omega}+\mhexp{-a_\omega}-2}}{a_\omega^2}=\Phi^{\prime\prime}\at{0}\frac{4\sinh^2\at{\tfrac12a_\omega}}{a_\omega^2}\,,
\end{align}
which determines $a_\omega>0$ uniquely provided that $\om$ exceeds the sound speed $\sqrt{\Phi^{\prime\prime}\at{0}}$ and implies the asymptotic law $a_\omega=2\ln\omega+\Do{\ln\omega}\to\infty$ as $\omega\to\infty$. For   our   asymptotic analysis, however, the rough estimates \eqref{Lem:TailEstimates.Eqn1} and \eqref{Lem:TailEstimates.Eqn2} with arbitrary but fixed decay rate are sufficient and guarantee in \S\ref{sect:Stability} that all tail contributions are negligible.
%

%
\section{Linear and nonlinear orbital stability}\label{sect:Stability}
%
%

This chapter is devoted to the dynamical stability of the high-energy waves constructed in \S\ref{sect:NonlWaves}. In particular, we study the linearized traveling wave operator $\calL_\omega$ from \eqref{Eqn:LinFPUOp} and the corresponding eigenproblem
\begin{align}
\label{Eqn:LinEVP}
\at{\partial_x - \la}S + \om^{-1}\nabla W=0\,,\qquad
\at{\partial_x - \la}V + \om^{-1}\nabla \Phi^{\prime\prime}\at{R_\omega}S=0
\end{align}
with eigenvalue $\la\in\Cset$ in certain function spaces and prove that the only non-stable eigenfunctions $\pair{S}{W}$ are the neutral ones that are generated by the shift symmetry and the spatial discreteness.  We next explain why this already implies the nonlinear orbital stability according to the Friesecke-Pego theory and postpone a more detailed overview of our asymptotic analysis of the spectral problem \eqref{Eqn:LinEVP} to \S\ref{sect:LinODEs}.

%

\subsection{The Friesecke-Pego criterion for nonlinear stability}
\label{sect:FPCrit}
%
%
%

%
Recall that the traveling waves from \S\ref{sect:NonlWaves} are supersonic while small perturbations in the tails propagate no faster than the sound speed $\sqrt{\Phi^{\prime\prime}\at{0}}\ll\omega$. On the heuristic level it is hence clear that small perturbations in front of the wave can affect the stability far more easily than those behind the pulse because the latter are only influenced by the exponentially small backward tail while the former will eventually interact with the nonlinear bulk of the wave.   We therefore work   in exponentially weighted norms that penalize perturbations in front of the propagating pulse. This idea   is natural and   has been used in \cite{PeWe94} and \cite{FP02} for PDEs and lattices, respectively, and provides also the framework for our analysis.    The stability with respect to the energy norm has later been studied by Mizumachi
in \cite{Miz09} by adapting ideas of \cite{MaMe05} to split perturbations in the energy space into an exponentially decaying part and another small solution to the nonlinear equation which is outrun in finite time by the main nonlinear wave. The Mizumachi approach, however, also requires to study the linearized stability problem in exponentially weighted spaces, see \cite{Miz09} and the proof of Corollary \ref{Cor:NonlStability} for more details.

\par
  In what follows we fix a decay parameter $a>0$, define   two linear multiplication operators $\calE_{-a}$ and $\calE_{+a}$ by
\begin{align*}
\bat{\calE_{\pm a} U}\at{x}=\mhexp{\pm ax}U\at{x}
\end{align*}
and introduce the spaces
\begin{align*}
\fspaceL_{\pm a}^2:=\calE_{\mp a} \ato{\fspaceL^2}\,,\qquad
\fspaceH_{\pm a}^1:=\calE_{\mp a} \ato{\fspaceH^1}\,,
\end{align*}
which are Banach spaces with respect to their natural, exponentially weighted norms. Notice that these definitions encode that a function $U_+\in \fspaceL^2_{+a}$ decays rapidly for $x\to+\infty$ so that $\calE_{+a}U_{+}$ still belongs to $\fspaceL^2$. Moreover, the operators $\calE_{\pm a}$ can also be applied to functions on $\Zset$, and this gives rise to the exponentially weighted spaces
\begin{align*}
\ell_{\pm a}^2=\calE_{\mp a}\ato{\ell^2}\,.
\end{align*}
as discrete analogues to $\fspaceL^2_{\pm a }$. In what follows we ignore the Hilbert space structure of $\fspaceL^2_{\pm a}$ and identify its dual spaces with $\fspaceL^2_{\mp a}$ according to the standard dual pairing
\begin{align}
\label{Eqn:DualPairing}
\bskp{U_-}{U_+}=\int\limits_{\Rset}U_-\at{x} \ol{U_+\at{x}}\dint{x}\,.
\end{align}
Concerning $a$, we rely from now on the following standing assumption. In particular, we are only interested in eigenfunctions $\pair{R}{S}$ that decay exponentially as $x\to+\infty$ but might grow for $x\to-\infty$.
\begin{assumption}[weighted ansatz space for perturbations]
\label{Def:PrmA}
The real-valued weight parameter $a>0$ is independent of $\om$, and we regard the operator $\calL_\omega$ from \eqref{Eqn:LinFPUOp} as defined on $\fspaceH^1_a{\times}\fspaceH^1_a$ and taking values in $\fspaceL^2_a{\times}\fspaceL^2_a$.
Moreover, any generic constant $C$ is allowed to depend on $a$.%
\end{assumption}
A further key ingredient to the Friesecke-Pego criterion   for stability   is the symplectic product
\begin{align}
\label{Eqn:SympProd}
\sigma\bpair{\pair{S_+}{W_+}}{\pair{S_-}{W_-}}:=\bskp{ S_+}{\nabla^{-1}W_-}+\bskp{ W_+}{\nabla^{-1}S_-}\qquad \text{with}\quad S_{\pm},W_{\pm}\in\fspaceL^2_{\pm a}\,,
\end{align}
where we anticipated a result from Lemma~\ref{Lem:DiscrDiffOperators}, namely that the centered difference operator $\nabla$ is invertible on exponentially weighted spaces. This symplectic product is naturally related to the FPUT dynamics since its first order formulation \eqref{Eqn:FPU} can formally be written in non-canonical form as
\begin{align*}
\begin{pmatrix}
0&\nabla^{-1}\\
\nabla^{-1}&0
\end{pmatrix}\frac{\dint}{\dint t}\begin{pmatrix}
{r}_{j-1/2}\\v_j
\end{pmatrix}
=
\begin{pmatrix}
\partial_r E\bpair{{r}_{j-1/2}}{v_j}\\
\partial_v E\bpair{{r}_{j-1/2}}{v_j}
\end{pmatrix}\,,\qquad E\pair{r}{v}=\tfrac12v^2+\Phi\at{r}
\end{align*}
with skew-symmetric operator on the left hand side and gradient of the energy density on the right hand side. We further introduce
\begin{align}
\label{Eqn:JordanFcts}
\bpair{S_{*,\,\omega}}{W_{*,\,\omega}}
:=
\bpair{\partial_x R_{\omega}}{\partial_x V_{\omega}}\,,
\qquad
\bpair{S_{\#,\,\omega}}{W_{\#,\,\omega}}
:=
\bpair{\partial_{\delta_\omega} R_{\omega}}{\partial_{\delta_\omega} V_{\omega}}
\end{align}
as these functions define the neutral stability modes, see Lemma~\ref{Lem:SpectralProperties}, and are hence important for our asymptotic analysis. Finally, extending the potential $\Phi$ by
\begin{align*}
\Phi\at{r}=\tfrac12\Phi^{\prime\prime}\at{0}r^2\qquad \text{for}\quad  r>0
\end{align*}
to a $\fspaceC^2$-smooth and strictly convex function on $\oointerval{-1}{\infty}$, we can summarize the non-asymptotic part of the Friesecke-Pego theory on orbital stability as follows.
\begin{theorem}[Friesecke-Pego, linear implies nonlinear stability]
\label{Thm:FPcrit}
Suppose that the family
\begin{align*}
  \om \quad\mapsto\quad \pair{R_\omega}{V_\omega}\in\fspaceL^2\times\fspaceL^2
\end{align*}
of solitary waves   with $\om \in\oointerval{\om_-}{\om_+}$   has the following properties:
\begin{enumerate}[leftmargin=.1\textwidth, labelsep=.02\textwidth]
\item[(P1)]\emph{\ul{Local convexity of interaction potential}}.
The distance component $R_\omega$ takes values in the interval $\oointerval{-1}{\infty}$, i.e., in the strict-convexity domain of $\Phi$.
\item[(P2)] \emph{\ul{Smoothness and decay}}.
The shift-sampling-map
\begin{align}
\label{Thm:FPcrit.Eqn1}
\pair{\tau}{\omega}\in\Rset\times\oointerval{\om_-}{\om_+}\quad\mapsto\quad
\Bpair{R_\omega\at{j+\tfrac12-\tau}}{V_\omega\at{j-\tau}}_{j\in\Zset}\in \at{\ell^2_{-a}\cap\ell^2_{+a}}^2
\end{align}
is $\fspaceC^1$-regular.
\item[(P3)]
\emph{\ul{Energy/speed transversality}}.
The wave energy $\om \mapsto h_\omega$ from \eqref{Eqn:TotalEnergy} has no critical point for $\om \in\oointerval{\om_-}{\om_+}$.
\end{enumerate}
Suppose further that a given wave speed $\om_0\in\oointerval{\om_-}{\om_+}$ complies with the following conditions:
\begin{enumerate}[leftmargin=.1\textwidth, labelsep=.02\textwidth]
\item[(S1)] \emph{\ul{Supersonic speed  and rapid decay}}.
We have $\om_0^2 > \Phi^{\prime\prime}\at{0}$ and $a_{\om_0}>a$, where $a_{\om_0}>0$ is the spatial decay rate predicted by the linear theory as in \eqref{Eqn:DecayEquation}.
\item[(S2)]\emph{\ul{No unstable eigenfunctions in symplectic complement of neutral ones}}.
The operator $\calL_{\om_0}$ admits for $\mhRe \lambda \geq0 $ and  $\abs{\mhIm \la} \leq \pi$  no eigenfunction $\pair{S}{W}\in \fspaceL_{a}^2\times\fspaceL_a^2$ that would satisfy
\begin{align*}
\sigma\Bpair{\pair{\calE_{2\pi\iu n }S_{*,\,\om_0}}{\calE_{2\pi\iu n }W_{*,\,\om_0}}}{\pair{S}{W}} &=0
\end{align*}
and
\begin{align*}
\sigma\Bpair{\pair{\calE_{2\pi\iu n }S_{\#,\,\om_0}}{\calE_{2\pi\iu n }W_{\#,\,\om_0}}}{\pair{S}{W}} &=0
\end{align*}
for all $n\in\Zset$.
\end{enumerate}
Then, the solitary wave with speed $\om_0$ is   orbitally stable in the sense of Main Result~\ref{res:Stab}.
\end{theorem}
\begin{proof}
This theorem is a simple combination of the result that linear implies nonlinear stability \cite[Theorem 1.1]{FP02} with the reformulation of the linear stability condition in terms of the spectral properties of \eqref{Eqn:LinEVP} in \cite[Theorem 2.2]{FP04a} and its restatement using \cite[Equation (1.15)]{FP04a}. We are using the same symplectic product and the same eigenvalue problem, but we are taking into account the extra shift in the distance profile and adapt the notations. For instance, the wave speed is  $\om$ instead of $c$, the potential is called $\Phi$ and not $V$, and the profile components for the wave under investigation are denoted by $\pair{R_{\om_0}}{V_{\om_0}}$ instead of $\pair{r_*}{p_*}$. We also write  $\npair{S_{*,\,\omega}}{W_{*,\,\omega}}$ and $\npair{S_{\#,\,\omega}}{W_{\#,\,\omega}}$ instead of $v_1,v_2$ for the neutral modes and use the derivative with respect to the parameter $\delta_\omega$ instead of $\omega$ in the definition of $\npair{S_{\#,\,\omega}}{W_{\#,\,\omega}}$.
\end{proof}
The high-energy waves form \S\ref{sect:NonlWaves} satisfy (P1) and (S1) by construction and due to \eqref{Eqn:DecayEquation}, while in Lemma~\ref{Lem:Regularity} below we derive (P2) from the regularity part of Theorem~\ref{Thm:SmoothnessNonlWaves} and the decay estimates in Lemma~\ref{Lem:TailEstimates}. Moreover, it has already been observed in \cite[Proposition 1.3, Equation (1.16)]{FP04a}
that (P3)   does not need   to be checked explicitly as it is  a consequence of (S2); see also the related comment after the proof of Lemma~\ref{Lem:SpectralProperties} below. The verification of the linear stability condition (S2), however, is much more involved and requires a careful asymptotic analysis of the linear advance-delay-differential equation \eqref{Eqn:LinEVP} in the limit $\om\to\infty$. The corresponding result for near sonic waves with $\om\gtrapprox\sqrt{\Phi^{\prime\prime}}\at{0}$ has been derived in \cite{FP04b} by regarding the  FPUT chain as a perturbation of the KdV equation, which is completely integrable and well-understood.
%

%
\subsection{Waves and spectrum in the weighted spaces}\label{sect:expspaces}
%

Before we study the spatial dynamics of eigenfunctions in the limit $\om\to\infty$, we
elucidate the role of the exponential weights in greater detail. The first result concerns the smooth parameter dependence.
\begin{lemma}[regularity in continuous and discrete weighted spaces]
\label{Lem:Regularity}
The map
\begin{align*}
\om\quad \mapsto\quad \bpair{R_\omega}{V_\omega}\in \at{\fspaceH^1_{-a}\cap \fspaceH^1_{+a}}^2
\end{align*}
as well as the shift-sampling map from \eqref{Thm:FPcrit.Eqn1} are, for sufficiently large $\om$, $\fspaceC^1$-regular in the sense of Fr\'echet.
\end{lemma}
\begin{proof}
\emph{\ul{Continuous setting}}:
The smoothness of $\om\mapsto R_\omega\in H^1_{\pm a}$ is a consequence of the exponential decay stated in Lemma~\ref{Lem:TailEstimates} and the regularity from   Theorem~\ref{Thm:SmoothnessNonlWaves},   which especially ensures that $\partial_\omega\partial_x R_\omega= \partial_x \partial_\omega R_\omega$. The corresponding properties of the velocities are then granted by \eqref{Eqn:TWNonlIdentities.V2}$_1$ since $\fspaceH^1_{\pm a}$ is invariant under the convolution with the indicator function $\chi$ and because $\Phi^\prime\at{R_\omega\at{x}}$ behaves like $R_\omega\at{x}$ as  $x\to\pm\infty$
thanks to Assumption~\ref{Ass.Pot}.
\par
\emph{\ul{Discrete setting}}:
Using the integral equation \eqref{Eqn:NonlFPDist} we write
\begin{align}
\label{Lem:Regularity.PEqn1}
r_j\pair{\tau}{\omega}:=R_\omega\at{j+\tfrac12+\tau} = \int\limits_{j+\tau}^{j+\tau+1} F_{R}\pair{\omega}{x}\dint{x}\,,
\end{align}
where the function $F_R$ defined by
\begin{align*}
F_R\pair{\omega}{x}:=-  \om^{-1}   V_\omega\at{x}=\om^{-2}\int\limits_{x-1/2}^{x+1/2} \Phi^\prime\at{R_\omega\at{y}}\dint{y}
\end{align*}
is $\fspaceC^2$-smooth and exponentially localized according to Theorem~\ref{Thm:ExistenceNonlWaves} and Lemma~\ref{Lem:TailEstimates}. In order to study the regularity of the map $r=\at{r_j}_{j\in\Zset}$, we verify by direct computation the pointwise derivatives
\begin{align*}
\partial_\tau
r_j\pair{\tau}{\omega}=\int\limits_{j+\tau}^{j+\tau+1} \partial_x F_R\pair{\omega}{x}\dint{x}\,,\qquad
\partial_\omega
r_j\pair{\tau}{\omega}=\int\limits_{j+\tau}^{j+\tau+1} \partial_\omega F_R\pair{\omega}{x}\dint{x}\,.
\end{align*}
For $\partial_\tau r=\at{\partial_\tau r_j}_{j\in\Zset}$  we thus find
\begin{align*}
\bnorm{\partial_\tau r\pair{\tau}{\omega}}^2_{\ell^2_{\pm a}}&=
\sum_{j\in\Zset} \mhexp{\pm2a j}\at{\;\;\int\limits_{j+\tau}^{j+\tau+1} \partial_x F_R\pair{\omega}{x}\dint{x}}^2
\\&\leq
\sum_{j\in\Zset} \mhexp{\pm2a j}
\int\limits_{j+\tau}^{j+\tau+1}\babs{ \partial_x F_R\pair{\omega}{x}}^2\dint{x}
\leq C\mhexp{2a\abs{\tau}} \int\limits_{-\infty}^{+\infty}\mhexp{\pm2ax}\mhexp{-4a\abs{x}}\dint{x}
\leq C\mhexp{2a\abs{\tau}}
\end{align*}
by H\"older's inequality and because Lemma~\ref{Lem:TailEstimates} cannot only be evaluated for the decay rate $a$ but also for the larger one $2a$. We further get
\begin{align*}
\bnorm{\partial_\tau r\pair{\tau+\tilde{\tau}}{\om+\tilde{\omega}}-
\partial_\tau r\pair{\tau}{\omega}}^2_{\ell^2_{\pm a}}&\leq C\mhexp{2a\abs{\tau}}
\int\limits_{-\infty}^{\infty} \mhexp{\pm2 ax}\babs{\partial_x F_R\pair{\om+\tilde{\omega}}{x+\tilde{\tau}}-\partial_x F_R\pair{\omega}{x}}^2\dint{x}\,,
\end{align*}
so the Dominated Convergence Theorem implies
\begin{align*}
\bnorm{\partial_\tau r\pair{\tau+\tilde{\tau}}{\om+\tilde{\omega}}-
\partial_\tau r\pair{\tau}{\omega}}^2_{\ell^2_{\pm a}}\quad\xrightarrow{\;\;\pair{\tilde{\tau}}{\tilde{\omega}}\to\pair{0}{0}\;\;}\quad 0
\end{align*}
due to the exponential localization and the continuity of $\partial_x F_R$. In summary, we have shown that $\partial_\tau r$ belongs to the weighted spaces $\ell^2_{\pm a}$ and depends there continuously on the variables $\pair{\tau}{\omega}$. Repeating all arguments for $\partial_\omega r$, we conclude that the map $r$ from \eqref{Lem:Regularity.PEqn1} is in fact $\fspaceC^1$-regular, and the statement concerning $v=\at{v_j}_{j\in\Zset}$ with
\begin{align*}
v_j\pair{\tau}{\omega}:=V_\omega\at{j+\tau} = \int\limits_{j+\tau-1/2}^{j+\tau+1/2} F_V\pair{\omega}{x}\dint{x}\qquad\text{and}\qquad F_V\pair{\omega}{x}:=-\om^{-1} \Phi^\prime\bat{R_\omega\at{x}}
\end{align*}
follows analogously.
\end{proof}
Our second observation in this section is that the exponential weights allow us to invert discrete differential operators.
\begin{lemma}[differential operators in weighted spaces]
\label{Lem:DiscrDiffOperators}
The operators $\nabla$ and $\Delta$ are continuously invertible on $\fspaceL^2_{\pm a}$. They also commute with each other and with $\partial_x$.
\end{lemma}
\begin{proof}
A direct computation reveals
\begin{align*}
\at{\calE_{+a}\nabla\calE_{-a}U}\at{x}&=\mhexp{-a/2}U\at{x+\tfrac12}-\mhexp{+a/2}U\at{x-\tfrac12}\,,\\
\at{\calE_{+a}\Delta\calE_{-a}U}\at{x}&=\mhexp{-a}U\at{x+1}+\mhexp{+a}U\at{x-1}-2U\at{x}\,,
\end{align*}
and we conclude that both the operators $\calE_{+a}\nabla\calE_{-a}$ and $\calE_{+a}\Delta\calE_{-a}$ diagonalize in Fourier space and correspond to the symbol functions $d_a$ and $d_a^2$, respectively, with $d_a\at\kappa:= 2 \sinh\at{\iu\kappa/2+a/2}$. Moreover, denoting by $\partial D_\varrho$ the circle of radius $\varrho$ in $\Cset$, we find
\begin{align*}
\abs{d_a\at\kappa}\geq  \mathrm{dist}\,\bpair{\partial D_{\mhexp{+a/2}}}{\partial D_{\mhexp{-a/2}}}=2\sinh\at{a/2}>0
\end{align*}
and the invertibility assertions in $\fspaceL^2_{+a}$ follow from Parseval's Theorem and the definition of $\fspaceL^2_{+a}$. The arguments for $\fspaceL^2_{-a}$ are analogous and the claimed commutator relations are obvious.
\end{proof}
\begin{figure}[ht!] %
\centering{ %
\includegraphics[width=0.4\textwidth]{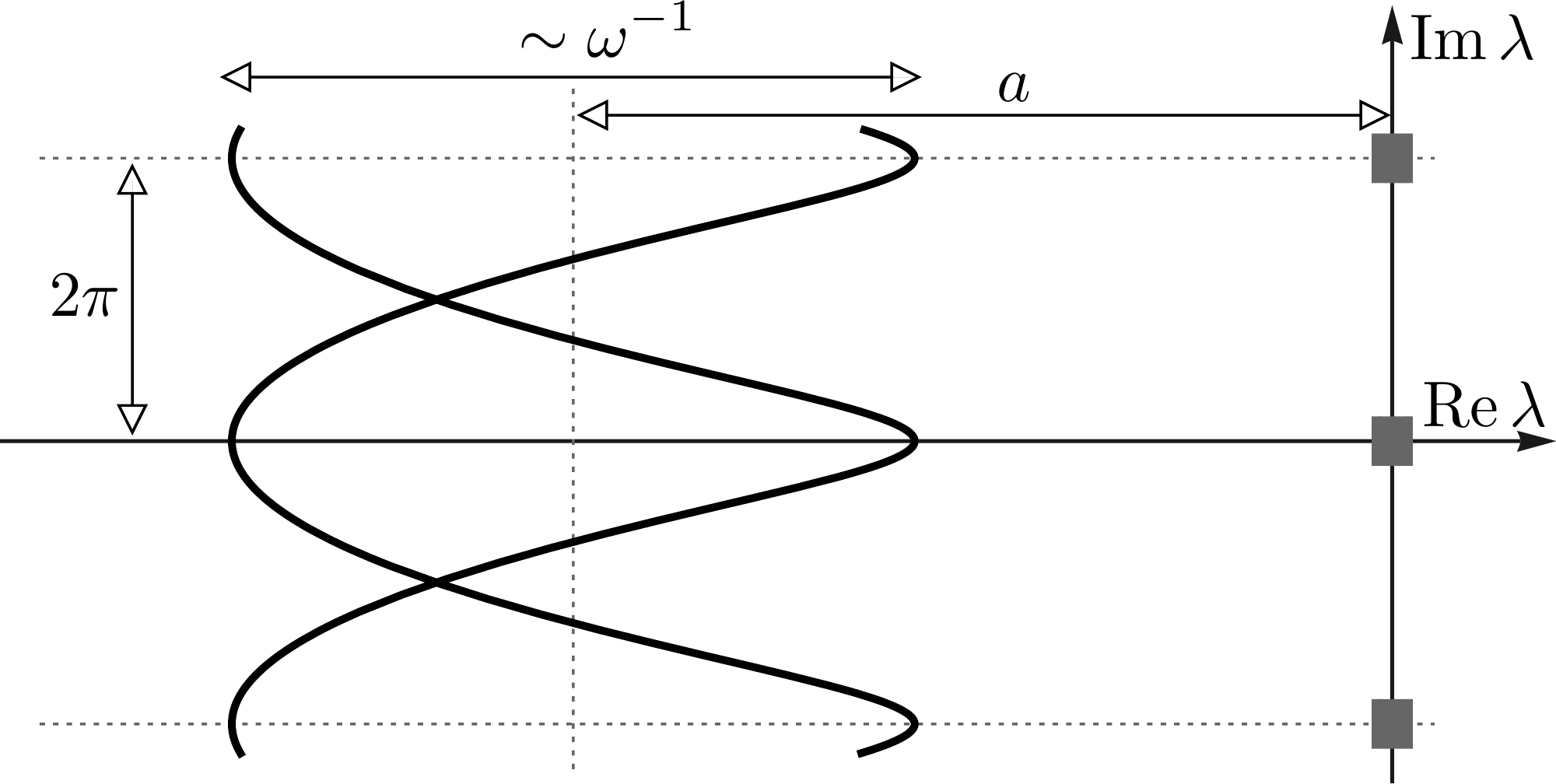}
} %
\caption{ %
On the $2\pi\iu$-periodic spectrum of the operator $\calL_\omega$. The smooth curves represent the essential spectrum as computed in the proof of Lemma~\ref{Lem:SpectralProperties} while the squares indicate the two-dimensional Jordan blocks  stemming from the shift symmetry. In Theorem~\ref{Thm:NoOtherEigenvalues} we show for large $\om$ that there is no further point spectrum in $\mathrm{Re}\,\la>-a$ but stable eigenvalues with $\mathrm{Re}\,\la< -a+\Do{1}$ might exist.} %
\label{Fig:Spectrum}%
\end{figure} %
  We finally derive a preliminary characterization of the spectrum of the operator $\calL_\omega$ from \eqref{Eqn:LinFPUOp}, see Figure~\ref{Fig:Spectrum} for an illustration. More precisely, following \cite{FP02,FP04a} we show $\at{i}$ that the spectrum of $\calL_\omega$ is periodic due to the spatial discreteness of the atomic chain,  $\at{ii}$ that the continuous spectrum has strictly negative real part due to the exponential weight, and $\at{iii}$ that the shift symmetry gives rise to a two dimensional Jordan block of neutral modes related to shifts and accelerations of the traveling wave with speed $\omega$. The key argument concerning linear stability is then to disprove the existence of unstable and other neutral eigenfunctions. This, however, cannot be concluded by adapting the KdV arguments from \cite{FP04b} but requires a different asymptotic analysis; cf. also the discussion in \S\ref{sect:LinODEs}.
\begin{lemma}[elementary properties of $\mathrm{spec}\,\calL_\omega$]
\label{Lem:SpectralProperties}
The $\fspaceL^2_a$-spectrum of $\calL_\omega$ is invariant under
the symmetry  transformation
\begin{align*}
\triple{\la}{S}{W}\rightsquigarrow\triple{\la+\iu 2\pi}{\calE_{\iu 2\pi}S}{\calE_{\iu 2\pi}W}
\end{align*}
and its essential part satisfies
\begin{align}
\notag
\mathrm{ess\,spec}\, \calL_\omega \subseteq \ccinterval{-a -C\omega^{-1}}{-a +C\omega^{-1}}\times \iu \Rset
\end{align}
for some constant $C$ which depends only on $a$. Moreover, the generalized nullspace of $\calL_\omega$ is --- for all sufficiently large $\om$ --- a two-dimensional Jordan block and spanned by the proper kernel function $\pair{S_{*,\,\omega}}{W_{*,\,\omega}}$ and the cyclic kernel function $\pair{S_{\#,\,\omega}}{W_{\#,\,\omega}}$.  These kernel functions   belong to
$\fspaceL_{-a}^2{\times}\fspaceL_{-a}^2\cap\fspaceL_{+a}^2{\times}\fspaceL_{+a}^2$ and  fulfil
\begin{align}
\label{Lem:SpectralProperties.Eqn2}
\delta_\omega^{m/2+1}\sigma\Bpair{\pair{S_{*,\,\omega}}{W_{*,\,\omega}}}{\pair{S_{\#,\,\omega}}{W_{\#,\,\omega}}}=\delta_\omega \partial_{\delta_\omega}h_\omega\quad\xrightarrow{\;\;\om\to\infty\;\;}\quad -\frac{m}{2}\,,
\end{align}
where the energy $h_\omega$ is defined in \eqref{Eqn:TotalEnergy}.
\end{lemma}
\begin{proof}
\emph{\ul{Essential spectrum}}: The operator $\calL_\omega$ can be written as
\begin{align*}
\calL_\omega = \calL_{1,\,\omega}+\calL_{2,\,\omega}
\end{align*}
with
\begin{align*}
 \calL_{1,\,\omega}\bpair{S}{W}:=
\bpair{
\partial_x S+\om^{-1}\nabla W}{
\partial_x W+\om^{-1}c\nabla S}\,,\qquad
 \calL_{2,\,\omega}\bpair{S}{W}:=
\bpair{0}{
\om^{-1}\nabla \at{C_\omega S}}\,,
\end{align*}
where the scalar $c$ and the coefficient function $C_\omega$ are defined by
\begin{align*}
c:=\Phi^{\prime\prime}\at{0}\,,\qquad C_\omega\at{x}=\Phi^{\prime\prime}\at{R_\omega\at{x}}-c\,.
\end{align*}
Since $C_\omega$ decays exponentially as $\abs{x}\to0$ according to Assumption~\ref{Ass.Pot} and Lemma~\ref{Lem:TailEstimates}, we   know   that $\calL_{2,\,\omega}$ is a compact linear operator from $\fspaceH_a^1{\times}\fspaceH^1_a$ to $\fspaceL_a^2{\times}\fspaceL^2_a$. Moreover, the operator $\calE_{+a}\calL_{1,\,\omega}\calE_{-a}$ diagonalizes in Fourier space and corresponds to the symbol matrix
\begin{align*}
L_{1,\,\omega}\at{\kappa}=
\begin{pmatrix}
-\iu\kappa-a&-2\omega^{-1}\sinh\at{\iu\ka/2+a/2}\\
-2\omega^{-1}c\sinh\at{\iu\ka/2+a/2}& -\iu\kappa-a
\end{pmatrix}\,,
\end{align*}
where $\kappa$ denotes the Fourier variable dual to $x$. Using this, we readily compute
\begin{align*}
\mathrm{spec}\, \calL_{1,\,\omega} =\mathrm{ess\,spec}\, \calL_{1,\,\omega} &= \{\la\in\Cset\;:\det \at{L_{1,\,\omega}\at{\kappa} - \la\,\mathrm{id}} =0 \;\text{for some $\kappa\in\Rset$}\;\}\\&=
 \{-\iu\kappa-a\pm\omega^{-1}c^{1/2}\sinh\at{\iu \ka/2+a/2}\;:\;\kappa\in\Rset\}
\end{align*}
and observe --- for large $\om$ --- that $\calL_{1,\,\omega}$ is a continuously invertible Fredholm operator with index $0$ from  $\fspaceH_a^1{\times}\fspaceH^1_a$ to  $\fspaceL_a^2{\times}\fspaceL^2_a$. This implies that  $\calL_{\omega}$ is also a Fredholm-operator with index $0$.
\par
\emph{\ul{Jordan structure of the eigenvalue $0$}}: In the remainder of this proof we abbreviate $U:=\pair{S}{W}$ and notice that the exponential decay of both $U_{*,\,\omega}$ and $U_{\#,\,\omega}$ follows --- thanks to the properties of the discrete differential operators and the convolution with the indicator function $\chi$ of the interval $\ccinterval{-1/2}{+1/2}$ --- from formula \eqref{Eqn:TWNonlIdentities.V2}, Assumption~\ref{Ass.Pot} and the tail estimates in Lemma~\ref{Lem:TailEstimates}. Straight forward computations relying on \eqref{Eqn:TWNonlIdentities.V1} and \eqref{Eqn:JordanFcts} provide the Jordan identities
\begin{align}
\label{Eqn:JordanIdentities}
\calL_\omega U_{*,\,\omega}=0\,,\qquad
\calL_\omega U_{\#,\,\omega}=\tfrac12 m \delta_\omega^{-1}
U_{*,\,\omega}\,,
\end{align}
so it remains to show that $U_{\#,\,\omega}$ does not belong to the image of $\calL_\omega$, which coincides --- thanks to the Fredholm alternative and by \eqref{Eqn:DualPairing} --- with the $\fspaceL^2$-orthogonal complement of the kernel of the dual operator   $\calL_\omega^*$ with
\begin{align*}
\calL_\omega^*\pair{S}{W}
=-
\Bpair{
\partial_x S+\om^{-1}\Phi^{\prime\prime}\at{R_\omega}\nabla W}{\partial_x W+\om^{-1}\nabla S}
\,.%
  \end{align*}
By direct computations we verify
\begin{align*}
\calL_\omega^* =-   \calD^{-1} \calL_\omega\calD^{+1} \,,\qquad \calD^{\pm 1}\pair{S}{W}:=\bpair{\nabla^{\pm1}W}{\nabla^{\pm1}S}\,,
\end{align*}
where we used the $\fspaceL^2$-antisymmetry of $\partial_x$ and $\nabla^{\pm 1}$ and the fact that these operators commute with each other, and conclude in view of \eqref{Eqn:JordanFcts} that  $\calD^{-1}U_{*,\,\omega}$ belongs to the kernel of $\calL_\omega^*$. On the other hand, the symplectic relation  \eqref{Lem:SpectralProperties.Eqn2} (which we derive in a moment) reveals that $U_{\#,\,\omega}$ and $\calD^{-1}U_{*,\,\omega}$ are not $\fspaceL^2$-perpendicular for large $\om$, so $U_{\#,\,\omega}$ cannot annihilate the kernel of $\calL_\omega^*$. We thus have shown that $1$ and $2$ are in fact the geometric and algebraic multiplicities of the eigenvalue $0$.
\par
\emph{\ul{Symplectic product}}: Using the definitions \eqref{Eqn:JordanFcts}, the nonlinear advance-delay-differential equations \eqref{Eqn:TWNonlIdentities.V1}, and the antisymmetry of $\nabla^{-1}$ we demonstrate
\begin{align}
\label{Lem:SpectralProperties.PEqn1}
\bskp{W_{*,\,\omega}}{\nabla^{-1}S_{\#,\,\omega}}&=\bskp{\partial_x V_\omega}{\nabla^{-1}\partial_{\delta_\omega}R_{\omega}}=-\bskp{\nabla^{-1}\partial_x V_\omega}{\partial_{\delta_\omega}R_{\omega}}=
\om^{-1}\bskp{\Phi^\prime\at{R_\omega}}{\partial_{\delta_\omega}R_\omega}\,,
\end{align}
and similarly we find
\begin{align}
\label{Lem:SpectralProperties.PEqn2}
\bskp{S_{*,\,\omega}}{\nabla^{-1}W_{\#,\,\omega}}&=\bskp{\partial_x R_\omega}{\nabla^{-1}\partial_{\delta_\omega}V_\omega}=
-\bskp{\nabla^{-1}\partial_x R_\omega}{\partial_{\delta_\omega}V_\omega}=\om^{-1}\bskp{V_\omega}{\partial_{\delta_\omega}V_\omega}
\\\notag&= %
-\om^{-2}\bskp{\chi\ast\Phi^\prime\at{R_\omega}}{\partial_{\delta_\omega}V_\omega}=-\om^{-2}\bskp{\Phi^\prime\at{R_\omega}}{\chi\ast\partial_{\delta_\omega}V_\omega}
\\\notag&= %
\om^{-2}\bskp{\Phi^\prime\at{R_\omega}}{\partial_{\delta_\omega}\at{\om R_\omega}}=
\om^{-1}\bskp{\Phi^\prime\at{R_\omega}}{\partial_{\delta_\omega}R_\omega}+  \om^{-2}\at{\partial_{\delta_\omega}\omega}\,
\bskp{\Phi^\prime\at{R_\omega}}{R_\omega}
\end{align}
where we also used the integrated traveling wave equation \eqref{Eqn:TWNonlIdentities.V2} as well as the symmetry of the convolution with the indicator function $\chi$. To estimate the different contributions, we first employ   the definitions in Lemma~\ref{Lem:DefAlpha} and \eqref{Eqn:DefDelta}+\eqref{Eqn:SpaceScalign}+\eqref{Eqn:DefTildeP}+\eqref{Eqn:DefCheckQ}   to replace $\Phi^\prime\at{R_\omega}$ and $\partial_{\delta_\omega} R_\omega$  by
$\tilde{K}_\omega$ and $\tilde{Q}_\omega$, respectively. This gives
\begin{align*}
\bskp{\Phi^\prime\at{R_\omega}}{\partial_{\delta_\omega}R_\omega}&= \int\limits_{\Rset}\Bat{\delta_\omega^{-m-1}\al_\omega^{-m-1}\tilde{K}_\omega\at{\tilde{x}}}\Bat{m \al_\omega\tilde{Q}_\omega\at{\tilde{x}}}\dint{\at{\delta_\omega \beta_\omega \tilde{x}}}\\&=
\frac{m}{\delta_\omega^{m}\al_\omega^{m/2-1}}\int\limits_\Rset
\tilde{K}_\omega\at{\tilde{x}}\tilde{Q}_\omega\at{\tilde{x}}\dint{\tilde{x}}\,,
\end{align*}
where   $\tilde{K}_\omega$ and $\tilde{Q}_\omega$ can in turn be approximated via $\bar{K}$ and $\breve{Q}_\omega$ from \eqref{Eqn:DefBarPZ} and \eqref{Eqn:DefCheckU} by the ODE solution $\bar{Y}$ defined in Lemma~\ref{Lem:AsympODE.Props}. More precisely,   from the approximation results in \S\ref{sect:NonlWaves} --- see Theorems~\ref{Thm:ExistenceNonlWaves} and \ref{Thm:SmoothnessNonlWaves}, as well as Lemmas~\ref{Lem:AsympCheckU} and~\ref{Lem:TailEstimates} --- we infer
\begin{align*}
\int\limits_\Rset
\tilde{K}_\omega\at{\tilde{x}}\tilde{Q}_\omega\at{\tilde{x}}\dint{\tilde{x}}
\quad\xrightarrow{\;\;\om\to\infty\;\;}\quad
\frac{1}{2m} \int\limits_\Rset
\bar{Y}^{\prime\prime} \at{\tilde{x}}\bar{Y}^\flat \at{\tilde{x}}\dint{\tilde{x}}=
-\frac{1}{2m} \int\limits_\Rset \tilde{x}
\bar{Y}^\prime\at{\tilde{x}} \bar{Y}^{\prime\prime}  \at{\tilde{x}}\dint{\tilde{x}}
\end{align*}
and hence
\begin{align}
\label{Lem:SpectralProperties.PEqn3}
\om^{-1}\bskp{\Phi^\prime\at{R_\omega}}{\partial_{\delta_\omega}R_\omega}=
\delta_\omega^{-m/2}\bat{-\tilde{C}+\Do{1}}
\end{align}
for some constant $\tilde{C}>0$ thanks to $\om=\delta_\omega^{-m/2}$.   Similarly, by \eqref{Eqn:DefTildeY}+\eqref{Eqn:SpaceScalign}+\eqref{Eqn:DefTildeP}   we have
\begin{align*}
\bskp{\Phi^\prime\at{R_\omega}}{R_\omega}&= \int_{\Rset}\Bat{\delta_\omega^{-m-1}\al_\omega^{-m-1}\tilde{K}_\omega\at{\tilde{x}}}\Bat{-1+\delta_\omega\al_\omega\tilde{Y}_\omega\at{\tilde{x}}}\dint{\at{\delta_\omega \beta_\omega \tilde{x}}}\,,
\end{align*}
and combining this with $\om^{-2}\partial_{\delta_\omega}\om=-\tfrac{m}{2}\delta_\omega^{m/2-1}$ and the approximation results from \S\ref{sect:NonlWaves} we arrive at
\begin{align}
\label{Lem:SpectralProperties.PEqn4}
  \om^{-2}\at{\partial_{\delta_\omega}\omega}\,
\bskp{\Phi^\prime\at{R_\omega}}{R_\omega}=
\delta_\omega^{-m/2-1}\bat{-C+\Do{1}}
\end{align}
with
\begin{align*}
-C &:= \at{\lim_{\om\to\infty}\al_\omega^{-m/2}}\cdot \frac{m}{2\at{m+1}}\cdot\int\limits_{\Rset}\frac{1}{\bar{Y}\at{\tilde{x}}^{m+1}}\dint{\tilde{x}}\\&= \frac{\sqrt{m}\sqrt{m+1}}{2}\cdot \frac{m}{2\at{m+1}}\cdot\at{\at{m+1}\bar{Y}^\prime\at{\infty}}=\frac{m}{2}
\,,
\end{align*}
see  \eqref{Lem:DefAlpha.Eqn2}, \eqref{Eqn:DefBarPZ}, and Lemma~\ref{Lem:AsympODE.Props}. The claim \eqref{Lem:SpectralProperties.Eqn2} is a direct consequence of \eqref{Lem:SpectralProperties.PEqn3} and \eqref{Lem:SpectralProperties.PEqn4},  where the identity
\begin{align}
\label{Eqn:NaturalJordanCondition}
\om\, \sigma\,\pair{U_{*,\,\omega}}{U_{\#,\,\omega}}=\partial_{\delta_\omega} h_\omega
\end{align}
is ensured by
parts of \eqref{Lem:SpectralProperties.PEqn1} and \eqref{Lem:SpectralProperties.PEqn2}.
\end{proof}
Lemma~\ref{Lem:SpectralProperties} reveals that the traveling wave energy $h_\omega$ is strictly increasing with respect to large $\om$. Such an energetic transversality condition is quite natural for traveling waves in spatially extended Hamiltonian systems and is also encoded by the assumptions of the Grillakis-Shatah-Strauss theory \cite{GSS87,GSS90} on orbital stability in dispersive PDEs,   see   \cite{Ang09} and references therein.   Recently \cite{CKVX17,XCKV18}     discussed sign changes  in \eqref{Eqn:NaturalJordanCondition} as a criterion for instability along a family of waves. The Friesecke-Pego theory for atomic chains generalizes the well-established framework for spatially continuous systems for it is not based on a second conserved quantity, which is usually used in the PDE context to rule out perturbations that accelerate the wave, cf. the related discussion in \cite{FP02}.   Notice that the orbital stability of standing waves in discrete nonlinear Sch\"odinger equations is also based on two conservation laws, see for instance \cite{Wei99}.
%
%

%
\subsection{Heuristic arguments for linear stability and overview}\label{sect:LinODEs}
%
%

In what follows we consider families of eigenvalues and eigenfunctions parameterized by $\om$ and write $\triple{\la_\omega}{S_\omega}{V_\omega}$ instead of $\triple{\la}{S}{W}$, where the eigenvalue is additionally split into its real and imaginary part according to
\begin{align}
\label{Eqn:EVSplit}
\la_\omega = \mu_\omega+\iu \nu_\omega\,.
\end{align}
We also observe that the velocity component can easily be eliminated in \eqref{Eqn:LinEVP} and that the distance component must satisfy the second-order advance-delay differential equation
\begin{align}
\label{Eqn:SecondOrderEVP}
\at{\partial_x - \la_\omega}^2 S_\omega = \Delta \bat{\omega^{-2}\Phi^{\prime\prime}\at{R_\omega}S_\omega}\,,
\end{align}
which is linear in ${S}_\omega$ but depends explicitly on $\omega$ and ${R}_\omega$. Since the right hand side in \eqref{Eqn:SecondOrderEVP} is singular due to the properties of $R_\omega$, we pass --- as in \S\ref{sect:NonlWaves} --- to the scaled space variable $\tilde{x}$ from \eqref{Eqn:SpaceScalign} and write
\begin{align}
\label{Eqn:DefTildeS}
\tilde{S}_\omega\at{\tilde{x}}:=
S_\omega\at{\delta_\omega\beta_\omega\tilde{x}}
\end{align}
but we also consider two functions $\tilde{G}_\omega$ and $\tilde{T}_\omega$ defined by
\begin{align}
\label{Eqn:DefTildeG}
\tilde{G}_\omega\at{\tilde{x}}:=\exp\bat{\at{a-\iu\nu_\omega}\delta_\omega\beta_\omega \tilde{x}}\tilde{S}_\omega\at{\tilde{x}}\,,
\end{align}
and
\begin{align}
\label{Eqn:DefTildeT}
\tilde{T}_\omega\at{\tilde{x}}:=
\exp\bat{-\la_\omega\delta_\omega\beta_\omega \tilde{x}}\tilde{S}_\omega\at{\tilde{x}}
=\exp\at{-\eps_\omega\tilde{x}}\tilde{G}_\omega\at{\tilde{x}}\,,
\end{align}
respectively. Here, $\delta_\omega$ and $\be_\omega$ are given as in  \eqref{Eqn:DefDelta} and \eqref{Eqn:DefBeta} and the  quantity
\begin{align}
\label{Eqn:DefEps}
\eps_\omega:=\delta_\omega \be_\omega\at{a+\mu_\omega}\,,
\end{align}
which is positive for $\mu_\omega>-a$, will later be identified as small. Direct computations transform \eqref{Eqn:SecondOrderEVP} into
\begin{align}
\label{Eqn:FPG.ODE}
\at{\partial_{\tilde{x}}-\eps_\omega}^2 \tilde{G}_\omega = \tilde{\Delta}_{-a+\iu\nu_\omega}\bat{\tilde{P}_\omega\tilde{G}_\omega}
\end{align}
and
\begin{align}
\label{Eqn:FPT.ODE}
\partial_{\tilde{x}}^2\tilde{T}_\omega=\tilde{\Delta}_{\la_\omega}\bat{\tilde{P}_\omega\tilde{T}_\omega}
\end{align}
with modified Laplacian as in \eqref{Eqn:DefModLapl}, and in what follows we investigate both equations in the limit $\omega\to\infty$. The introduction of $\tilde{G}_\omega$ and $\tilde{T}_\omega$ is motivated by the different needs in our asymptotic analysis. Since $\tilde{G}$ belongs to $\fspaceL^2$ by construction, we can investigate \eqref{Eqn:FPG.ODE} using nonlocal techniques such as Fourier transform and H\"older's inequality, and this finally provides a priori estimates for the eigenfunctions that are independent of any approximation. The equation for $\tilde{T}_\omega$, however, links the eigenvalue problem for $S_\omega$ to the linearized shape ODE and enables us to disprove the existence of unstable eigenfunctions in the symplectic complement of the neutral ones.

\par

{\bf{Approximation by local ODEs}}.
In order to explain our asymptotic ODE arguments on an informal level, we replace the coefficient profile $\tilde{P}_\omega$ in \eqref{Eqn:FPT.ODE} by its limit $\bar{P}$ from  \eqref{Eqn:DefBarPZ}, see also Theorem~\ref{Thm:ExistenceNonlWaves}, and
assume that $\tilde{T}_\omega$ grows at most linearly as $\tilde{x}\to\pm\infty$. This allows us --- similarly to the discussion in \S\ref{sect:Heuristics} ---  to neglect both the advance and the delay term in \eqref{Eqn:FPT.ODE} on the interval $I_\omega$, and provides
\begin{align}
\notag
\tilde{T}_\omega^{\prime\prime}\at{\tilde{x}}  \approx - 2\bar{P}\at{\tilde{x}}\tilde{T}_\omega\at{\tilde{x}} \qquad \text{for}\quad \tilde{x}\in\oointerval{-\xi_\omega}{+\xi_\omega}\,.
\end{align}
In other words, the transformed eigenfunction $\tilde{T}_\omega$ satisfies on the interval $I_\omega$ the linearized shape ODE \eqref{Lem:LinODE.Props.Eqn1} at least approximately, and we have
\begin{align}
\label{Eqn:LinHeu2}
\tilde{T}_\omega\at{\tilde{x}}  \approx \tilde{T}_\omega\at{0}\bar{T}_\even\at{\tilde{x}}+
\tilde{T}_\omega^\prime\at{0}\bar{T}_\odd\at{\tilde{x}}
 \qquad \text{for}\quad \tilde{x}\in\oointerval{-\xi_\omega}{+\xi_\omega}\,,
\end{align}
where $\bar{T}_\even$ and $\bar{T}_\odd$ are the two independent solutions from Lemma~\ref{Lem:LinODE.Props}. Moreover, on each of the two components of $3I_\omega\setminus I_\omega$ the dominant contribution to the right hand side of \eqref{Eqn:FPT.ODE} stems either from the advance or the delay term, while all terms are small outside of $3I_\omega$. In this way we justify the approximations
\begin{align*}
\tilde{T}_\omega^{\prime\prime}\at{\tilde{x}}  \approx \mhexp{+\la_\omega} \bar{P}\at{\tilde{x}+2\xi_\omega}\tilde{T}_\omega\at{\tilde{x}+2\xi_\omega}
 \approx -\tfrac12\mhexp{+\la_\omega}\tilde{T}^{\prime\prime}_\omega\at{\tilde{x}+2\xi_\omega}
\qquad \text{for}\quad \tilde{x}\in\oointerval{-3\xi_\omega}{-\xi_\omega}
\end{align*}
and
\begin{align}
\label{Eqn:LinHeu3}
\tilde{T}_\omega^{\prime\prime}\at{\tilde{x}}  \approx \mhexp{-\la_\omega} \bar{P}\at{\tilde{x}-2\xi_\omega}\tilde{T}_\omega\at{\tilde{x}-2\xi_\omega}
 \approx -\tfrac12\mhexp{-\la_\omega}\tilde{T}^{\prime\prime}_\omega\at{\tilde{x}-2\xi_\omega}  \qquad \text{for}\quad \tilde{x}\in\oointerval{+\xi_\omega}{+3\xi_\omega}\,,
\end{align}
as well as
\begin{align}
\label{Eqn:LinHeu4}
\tilde{T}_\omega^{\prime\prime}\bat{\tilde{x} }\approx 0\qquad \text{for}\quad
\abs{\tilde{x}}\geq 3\xi_\omega\,,
\end{align}
where the latter formula is consistent with our linear growth assumption.
\par
{\bf Matching conditions and decay constraint}.
The different local approximations for $\tilde{T}_\omega$ must be complemented by matching and decay conditions. On the one hand, $\tilde{T}_\omega\in\fspaceL^2_{\eps_\omega}$ implies
\begin{align}
\label{Eqn:LinHeu5}
\tilde{T}_\omega\at{3\xi_\omega}\approx0\,,\qquad
\tilde{T}_\omega^{\prime}\at{3\xi_\omega}\approx 0
\end{align}
because otherwise  $\tilde{T}_\omega$ would possess a nonvanishing, affine tail for $\tilde{x}\to+\infty$ due to \eqref{Eqn:LinHeu4}. On the other hand, $\tilde{T}_\omega$ and its derivative shall be continuous at $\xi_\omega$. In view of \eqref{Eqn:LinHeu5} and the local differential equation \eqref{Eqn:LinHeu3} we conclude that the continuity of $\tilde{T}_\omega^\prime$ is equivalent to
\begin{align*}
\tilde{T}_\omega^\prime{\at{+\xi_\omega}}-\tilde{T}_\omega^\prime{\at{-\xi_\omega}}=
\int\limits_{-\xi_\omega}^{+\xi_\omega}\tilde{T}_\omega^{\prime\prime}\at{\tilde{x}}\dint{\tilde{x}}\approx -2\mhexp{+\la_\omega}
\int\limits_{+\xi_\omega}^{+3\xi_\omega}\tilde{T}_\omega^{\prime\prime}\at{\tilde{x}}\dint{\tilde{x}}\approx +2\mhexp{+\la_\omega} \tilde{T}_\omega^\prime{\at{+\xi_\omega}},
\end{align*}
and combining this with \eqref{Eqn:LinHeu2} and the asymptotic properties from Lemma~\ref{Lem:LinODE.Props} we derive the first matching condition
\begin{align}
\label{Eqn:LinHeuM1}
\bat{1-\mhexp{+\la_\omega}} \tilde{T}_\omega\at{0}\approx 0
\end{align}
thanks to
\begin{align*}
\tilde{T}_\omega^\prime{\at{\pm\xi_\omega}} \approx \pm \tilde{T}_\omega\at{0}\bar{T}_\even^\prime\at{+\infty}\,.
\end{align*}
In the same manner we compute
\begin{align*}
\tilde{T}_\omega^\flat{\at{+\xi_\omega}}-\tilde{T}_\omega^\flat{\at{-\xi_\omega}}&\approx
\int\limits_{-\xi_\omega}^{+\xi_\omega}\tilde{x}\tilde{T}_\omega^{\prime\prime}\at{\tilde{x}}\dint{\tilde{x}}\approx -2\mhexp{+\la_\omega}
\int\limits_{+\xi_\omega}^{+3\xi_\omega}\at{\tilde{x}-2\xi_\omega}\tilde{T}_\omega^{\prime\prime}\at{\tilde{x}}\dint{\tilde{x}}
\\&\approx +2\mhexp{+\la_\omega}\Bat{ \tilde{T}_\omega^\flat{\at{+\xi_\omega}} -2\xi_\omega \tilde{T}_\omega^\prime\at{\xi_\omega}}
\end{align*}
as well as
\begin{align*}
\tilde{T}_\omega^\flat{\at{\pm\xi_\omega}} \approx \tilde{T}_\omega\at{0}\bar{T}_\even^\flat\at{+\infty}+\tilde{T}_\omega^\prime\at{0}\bar{T}_\odd^\flat\at{+\infty}\,,
\end{align*}
and obtain after rearranging terms the second matching condition
\begin{align}
\label{Eqn:LinHeuM2}
\tilde{T}_\omega^\prime\at{0}\bat{1-\mhexp{+\la_\omega}}\bar{T}_\odd^\flat\at{+\infty}
\approx \tilde{T}_\omega\at{0}\at{\mhexp{+\la_\omega}\bar{T}_\even^\flat\at{+\infty}
-
2\xi_\omega \bar{T}_\even^\prime\at{+\infty}}
\end{align}
for the continuity of both $\tilde{T}_\omega^\flat$ and  $\tilde{T}_\omega$.
\par
{\bf Discussion}. Equations \eqref{Eqn:LinHeuM1} and \eqref{Eqn:LinHeuM2} can be viewed as approximate constraints that couple the eigenvalue $\la_\omega$ via the linearized shape ODE and the required right-sided tail behavior to the initial data $\tilde{T}_\omega\at{0}$ and $\tilde{T}_\omega^\prime\at{0}$. The key observation is that these conditions are very restrictive
for large $\omega$. More specifically, recalling \eqref{Eqn:DefDelta} and \eqref{Eqn:DefXi} we easily check the asymptotic validity of the second matching condition in six typical cases and arrive at the following list, which indicates that the proper kernel functions from \eqref{Eqn:JordanFcts} are in fact the only eigenfunction of $\calL_\omega$:
\begin{center}
\begin{tabular}{c|c|c|c}
&even solution with & odd solution with &mixed solution with \\
&$\tilde{T}_\omega\at{0}\sim1$, $\tilde{T}_\omega^\prime\at{0}=0$&
$\tilde{T}_\omega\at{0}=0$, $\tilde{T}_\omega^\prime\at{0}\sim1$&
$\tilde{T}_\omega\at{0}\sim1$, $\tilde{T}_\omega^\prime\at{0}\sim1$\\
\hline
$\la_\omega=0 $&false & ok &false\\
$\la_\omega\sim1  $&false & false &false\\
\end{tabular}
\end{center}
Of course, this list is not exhaustive and covers neither almost odd eigenfunctions with small even part nor the limits $\la_\omega\to0$ and $\la_\omega\to\infty$. The asymptotic behavior in any of these cases depends crucially on the smallness of the error terms and cannot be characterized informally. Another challenge for the rigorous analysis is to replace the linear growth assumption concerning $\tilde{T}_\omega$ by weaker but more reliable tail estimates.
\par
{\bf Overview}. The remaining part of this chapter is organized as follows. In \S\ref{sect:StabilityEstimates} we study the nonlocal equation for $\tilde{G}_\omega$ in the framework of functional analysis and establish in Lemma~\ref{Lem:Compactness}  a priori bounds for the eigenfunctions. In \S\ref{sect:StabilityODEs} we return to the ODE point of view and show in Lemma~\ref{Lem:Convergence} that eigenvalues with $\mu_\omega>-a$ and $\abs{\nu_\omega}\leq\pi$ must in fact be small and that the corresponding eigenfunctions exhibit a certain asymptotic behavior which is basically independent of $\omega$ and $\la_\omega$. The proof exploits several approximation arguments but the key ingredient, see formulas \eqref{Lem:Convergence.Jump1}--\eqref{Lem:Convergence.Jump4},  is a rigorous analogue to the heuristic matching conditions from above. Our linear stability proof will be concluded in \S\ref{sect:PointSpectrum}. In particular, the proof of Theorem~\ref{Thm:NoOtherEigenvalues} guarantees that the symplectic product between any eigenfunction $\pair{S_\omega}{W_\omega}$ and the cyclic kernel function from \eqref{Eqn:JordanFcts} does not vanish, and hence that the proper kernel function is the only non-stable eigenfunction of $\calL_\omega$.

%
%
\subsection{A priori estimates for rescaled eigenfunctions}
\label{sect:StabilityEstimates} %
%
%

%
From now on we rely on the following standing assumption.
\begin{assumption}[family of eigenfunctions]
\label{Ass:Eigensystem} %
We consider families $\triple{\la_\omega}{S_\omega}{W_\omega}_{\omega}$ of solutions to the eigenvalue problem \eqref{Eqn:LinEVP}
with
\begin{align*}
S_\omega,\,W_\omega\in\fspaceL^2_a\,,\qquad\quad
\liminf_{\omega\to\infty}\mu_\omega>-a\,,\qquad\quad \nu_\omega\in\cointerval{-\pi}{+\pi}\,,
\end{align*}
where $\mu_\omega$ and $\nu_\omega$ abbreviate the real and imaginary part of $\la_\omega$ as in \eqref{Eqn:EVSplit}.
We also impose the normalization condition
\begin{align}
\label{Eqn:Normalization}
\bnorm{\tilde{\chi}_\omega\tilde{P}_\omega\tilde{G}_\omega}_1=1
\end{align}
which involves the functions   $\tilde{\chi}_\omega\in\fspaceL^2$,  $\tilde{P}_\omega\in\fspaceL^\infty$, and $\tilde{G}_\omega\in\fspaceL^2$ from \eqref{Eqn:DefChiAndI},
\eqref{Eqn:DefTildeP}, and \eqref{Eqn:DefTildeG}.
\end{assumption}
The first step of our rigorous analysis is to convert \eqref{Eqn:FPG.ODE} into a fixed point equation for $\tilde{G}_\omega$ by means of the following auxiliary result, which relies on Fourier techniques and is hence independent of initial conditions.

\begin{figure}[ht!] %
\centering{ %
\includegraphics[width=0.7\textwidth]{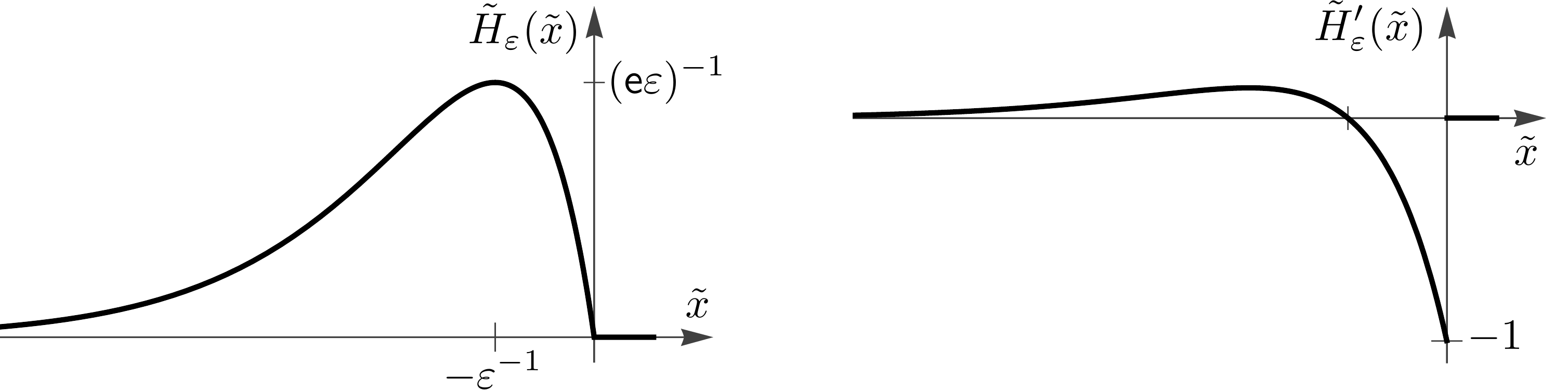} %
} %
\caption{%
The fundamental solution $\tilde{H}_\eps$ from \eqref{Lem:EV.FundSol.Eqn2} to the auxiliary problem \eqref{Lem:EV.FundSol.Eqn1}, which converges for $\eps\to0$ locally uniformly to the piecewise linear function $\tilde{H}_0$ with $\tilde{H}_0\at{\tilde{x}}=\max\{-\tilde{x},0\}$.
} %
\label{Fig:FundSol} %
\end{figure} %
\begin{lemma}[a fundamental solution]
\label{Lem:EV.FundSol}%
For any $\tilde{F}\in\fspaceL^2$ and all $\eps>0$ there exists a unique solution $\tilde{U}\in\fspaceL^2$ to
\begin{align}
\label{Lem:EV.FundSol.Eqn1}
\bat{\partial_{\tilde{x}}-\eps}^2\tilde{U}=\tilde{F}\,,
\end{align}
which can be written as
\begin{align}
\label{Lem:EV.FundSol.Eqn2}
\tilde{U}= \tilde{H}_{\eps}\ast \tilde{F}\qquad\quad \text{with}\qquad\quad \tilde{H}_{\eps}\at{\tilde{x}}:=\left\{\begin{array}{ccl}-\tilde{x}\exp\at{\eps\tilde{x}}&\text{for }&x\leq 0\\
0&\text{for}& x>0\end{array}\right.\,,
\end{align}
see Figure~\ref{Fig:FundSol}. Moreover, we have
\begin{align}
\label{Lem:EV.FundSol.Eqn3}
\bnorm{\tilde{H}_{\eps}}_p\leq C\eps^{-1-1/p}\,,\qquad
\bnorm{\tilde{H}_{\eps}^\prime}_p\leq C\eps^{-1/p}
\end{align}
for all $p\in\ccinterval{1}{\infty}$ (where we set $1/\infty =0$) and a universal constant $C$.
\end{lemma}
\begin{proof}
The operator $\bat{\partial_{\tilde{x}}-\eps}^2$ diagonalizes in Fourier space and its symbol function
\begin{align*}
\tilde{h}_{\eps}\nat{\tilde{k}}=\bat{\iu\tilde{k}+\eps}^2
\end{align*}
has no zeros on the real axis. In particular, we have
\begin{align*}
\bnorm{1/\tilde{h}_{\eps}}_\infty = \sup\limits_{\tilde{k}\in\Rset}{\babs{\iu\tilde{k}+\eps}^{-2}}=\eps^{-2}\,.
\end{align*}
and the existence and uniqueness of $\tilde{U}$ as in \eqref{Lem:EV.FundSol.Eqn1} follows immediately from Parseval's Theorem. Moreover,  by direct computations we verify $\tilde{H}_{\eps} \in\fspaceL^p$ and $\tilde{H}_{\eps}^\prime \in\fspaceL^p$ with \eqref{Lem:EV.FundSol.Eqn3}, as well as
\begin{align*}
\bat{\partial_{\tilde{x}}-\eps}^2\tilde{H}_{\eps}=\mathrm{Dirac}\,,
\end{align*}
and this implies the representation formula \eqref{Lem:EV.FundSol.Eqn2}.
\end{proof}

Thanks to Lemma~\ref{Lem:EV.FundSol}, the linear advance-delay-differential equation  \eqref{Eqn:FPG.ODE} is equivalent to the nonlocal integral equation
\begin{align}
\label{Eqn:FPG.Int}
 \tilde{G}_\omega = \tilde{\Delta}_{-a+\iu\nu_\omega}\bat{\tilde{H}_{\eps_\omega}\ast\nat{\tilde{P}_\omega\tilde{G}_\omega}}\,,
\end{align}
which gives rise to the  following estimates.
\begin{lemma}[a priori estimates for $\tilde{G}_\omega$ and smallness of $\eps_\omega$]
\label{Lem:Compactness}
For all sufficiently  large $\omega$, the function $\tilde{G}_\omega$ belongs to $\fspaceL^p\at\Rset$ for any $p\in\ccinterval{1}{\infty}$ and satisfies
\begin{align}
\notag
\bnorm{\tilde{G}_\omega}_p\leq C\eps_\omega ^{-1-1/p}
\end{align}
as well as
\begin{align}
\label{Lem:Compactness.Eqn2}
\babs{\tilde{G}_\omega\at{0}}+\bnorm{\tilde{G}_\omega^\prime}_\infty\leq C
\end{align}
for some constant $C$ independent of $\omega$. Moreover, we have  $\eps_\omega\to0$  as $\omega\to\infty$.
\end{lemma}
\begin{proof}
\ul{$L^p$-estimates}:
We define an approximation $\bar{G}_\omega$ of $\tilde{G}_\omega$ by
\begin{align}
\label{Lem:Compactness.PEqn2}
\bar{G}_\omega := \tilde{\Delta}_{-a+\iu\nu_\omega}\tilde{H}_{\eps_\omega}\ast\bat{\tilde{\chi}_\omega\tilde{P}_\omega\tilde{G}_\omega}
\end{align}
and deduce from Young's inequality and the uniform $\fspaceL^p$-continuity of $\tilde{\Delta}_{-a+\iu\nu_\omega}$ the estimates
\begin{align*}
\bnorm{\bar{G}_\omega}_1\leq C\bnorm{\tilde{H}_{\eps_\omega}}_1\bnorm{\tilde{\chi}_\omega\tilde{P}_\omega\tilde{G}_\omega}_1\leq C\eps_\omega^{-2}\,,\qquad
\bnorm{\bar{G}_\omega}_\infty\leq C\bnorm{\tilde{H}_{\eps_\omega}}_\infty\bnorm{\tilde{\chi}_\omega\tilde{P}_\omega\tilde{G}_\omega}_1\leq C\eps_\omega^{-1}\,,
\end{align*}
  see also \eqref{Eqn:Normalization} and \eqref{Lem:EV.FundSol.Eqn3}.   This implies $\bar{G}_\omega\in\fspaceL^p\at\Rset$ with
\begin{align}
\label{Lem:Compactness.PEqn5}
\bnorm{\bar{G}_\omega}_p\leq C \eps_\omega^{-1-1/p}\leq C \delta_\omega^{-1-1/p}
\end{align}
by interpolation and due to \eqref{Eqn:DefEps}. In view of \eqref{Eqn:FPG.Int} and \eqref{Lem:Compactness.PEqn2} we further write
\begin{align*}
\tilde{G}_\omega = \bar{G}_\omega +\tilde{\calR}_\omega\bato{ \tilde{G}_\omega}\,,
\end{align*}
where the linear operator $\tilde{\calR}_\omega$ is defined by
\begin{align*}
\tilde{\calR}_\omega  \bato{\tilde{U}} :=\tilde{\Delta}_{-a+\iu\nu_\omega}\tilde{H}_\omega\ast \bat{\nat{1-\tilde{\chi}_\omega}\tilde{P}_\omega \tilde{U}}
\end{align*}
and satisfies
\begin{align}
\label{Lem:Compactness.PEqn4}
\bnorm{\tilde{\calR}_\omega\nato{\tilde{U}}}_p\leq   C  \bnorm{\tilde{H}_{\eps_\omega}}_1\bnorm{\nat{1-\tilde{\chi}_\omega}\tilde{P}_\omega}_\infty\bnorm{\tilde{U}}_p\leq C\eps_\omega^{-2}\delta_\omega^{m+2}\bnorm{\tilde{U}}_p \leq C\delta_\omega^m \bnorm{\tilde{U}}_p
\end{align}
by Young's and H\"older's inequalities, where we also used Theorem~\ref{Thm:ExistenceNonlWaves} as well as \eqref{Eqn:DefEps}.
In particular, $\tilde{\calR}_\omega$ maps  the Hilbert space $\fspaceL^2$ contractively into itself for large $\omega$. The Neumann formula hence ensures
\begin{align}
\label{Lem:Compactness.PEqn1}
\tilde{G}_\omega =   \Bat{\mathrm{Id} -\tilde{\calR}_\omega}^{-1}\ato{\bar{G}_\omega} =\Bat{\mathrm{Id} +\tilde{\calR}_\omega+\tilde{\calR}_\omega^2+\tdots}\ato{\bar{G}_\omega}\,,
\end{align}
and we obtain $\tilde{G}_\omega\in\fspaceL^p$ with
\begin{align*}
\bnorm{\tilde{G}_\omega}_p\leq \sum_{k=0}^\infty \at{C\delta_\omega^m}^k\bnorm{\bar{G}_\omega}_p\leq C \bnorm{\bar{G}_\omega}_p\leq C \eps_\omega^{-1-1/p}
\end{align*}
by employing \eqref{Lem:Compactness.PEqn5} and \eqref{Lem:Compactness.PEqn4} once again.
\par
\ul{\emph{Further estimates}}: Young's inequality further provides
\begin{align*}
\bnorm{\bar{G}_\omega^\prime}_\infty\leq C\bnorm{\tilde{H}_{\eps_\omega}^\prime}_\infty \bnorm{\tilde{\chi}_\omega\tilde{P}_\omega\tilde{G}_\omega}_1\leq C
\end{align*}
as well as
\begin{align*}
\bnorm{\nat{\tilde{\calR}_\omega \nato{\tilde{U}}}^\prime}_\infty\leq  C  \bnorm{\tilde{H}_{\eps_\omega}^\prime}_1\bnorm{\nat{1-\tilde{\chi}_\omega}\tilde{P}_\omega}_\infty\bnorm{\tilde{U}}_\infty\leq C\eps_\omega^{-1}\delta_\omega^{m+2}\bnorm{\tilde{U}}_\infty
\leq C \delta_\omega^{m+1}\bnorm{\tilde{U}}_\infty\,,
\end{align*}
and combining this with \eqref{Lem:Compactness.PEqn4} and \eqref{Lem:Compactness.PEqn1} we find via
\begin{align*}
\bnorm{\tilde{G}_\omega^\prime}_\infty\leq \bnorm{\bar{G}_\omega^\prime}_\infty + C\delta_\omega^{m+1}\sum_{k=1}^\infty \bnorm{\tilde{\calR}_\omega^{k-1} \nato{\bar{G}_\omega}}_\infty\leq C+C\delta_\omega^{m+1}\sum_{k=1}^\infty \at{C\delta_\omega^{m}}^{k-1}C\delta_\omega^{-1}  \leq C
\end{align*}
the desired uniform upper bound for $\nnorm{\tilde{G}_\omega^\prime}_\infty$.
From this we infer
\begin{align*}
\babs{\tilde{G}_\omega\at{0}}\leq C
\end{align*}
because otherwise we would obtain via Theorem~\ref{Thm:ExistenceNonlWaves} and the estimate
\begin{align*}
\int\limits_{-1}^{+1}\babs{\tilde{P}_\omega\at{\tilde{x}}\tilde{G}_\omega\at{\tilde{x}}}\dint{\tilde{x}}\geq
  \int\limits_{-1}^{+1}\babs{\tilde{P}_\omega\at{\tilde{x}}}\Bat{\babs{\tilde{G}_\omega\at{0}}-C} \dint{\tilde{x}}
\geq
C^{-1} \babs{\tilde{G}_\omega\at{0}}-C
\end{align*}
a contradiction to the normalization condition \eqref{Eqn:Normalization}. In particular, by \eqref{Lem:Compactness.Eqn2} we have
\begin{align}
\label{Lem:Compactness.PEqn3}
\babs{\tilde{G}_\omega\at{\tilde{x}}}\leq C \bat{1+\abs{\tilde{x}}} \leq C\bar Y\at{\tilde{x}}
\end{align}
for all $\tilde{x}\in\Rset$, and combining this with \eqref{Eqn:Normalization}, Lemma~\ref{Lem:AsympODE.Props}, Theorem~\ref{Thm:ExistenceNonlWaves}, and \eqref{Lem:Compactness.PEqn5} we further estimate
\begin{align*}
1\leq\bnorm{\tilde{\chi}_\omega \nat{\tilde{P}_\omega-\bar{P}}\tilde{G}_\omega}_1+
\bnorm{\tilde{\chi}_\omega\bar{P}\tilde{G}_\omega}_1\leq C\delta_\omega^{\min\{k,\,m\}}+C\bnorm{\tilde{G}_\omega}_\infty\leq  C\at{\delta_\omega^{\min\{k,\,m\}}+\eps_\omega^{-1}}\,,
\end{align*}
which implies $\eps_\omega\leq C$. Similarly,
\begin{align}
\label{Lem:Compactness.PEqn6}
\bnorm{\bar{P}\tilde{G}_\omega}_1\quad\xrightarrow{\;\;\omega\to\infty\;\;}\quad 1
\end{align}
holds by \eqref{Eqn:Normalization} and Theorem~\ref{Thm:ExistenceNonlWaves}.
\par
\ul{\emph{Convergence of $\eps_\omega$}}: %
Suppose for contradiction that $\liminf_{\omega\to\infty}\eps_\omega>0$. By \eqref{Lem:Compactness.Eqn2} and the Arzel\'a-Ascoli Theorem there exists a (not relabeled) subsequence for $\omega\to\infty$ such that
\begin{align*}
\eps_\omega\quad\xrightarrow{\;\;\omega\to\infty\;\;}\quad \eps_\infty >0\qquad\text{and}\qquad
\tilde{G}_\omega\xrightarrow{\;\;\omega\to\infty\;\;}\tilde{G}_\infty\quad\text{in}\quad \fspace{BC}_\loc
\end{align*}
for some limit $\tilde{G}_\infty$. Moreover, the pointwise affine bound \eqref{Lem:Compactness.PEqn3} ensures that the limit $\tilde{G}_\infty$ is nontrivial as it  satisfies $\nnorm{\bar{P}\tilde{G}_\infty}_1=1$ due to \eqref{Lem:Compactness.PEqn6} and the Dominated Convergence Theorem. Testing the differential fixed point equation  \eqref{Eqn:FPG.ODE} with a smooth and compactly supported $\phi$ gives
\begin{align*}
\skp{\tilde{G}_\omega}{\bat{\partial_{\tilde{x}}+\eps_\omega}^2\phi}=\skp{\tilde{P}_\omega\tilde{G}_\omega}{\tilde{\Delta}_{+a-\iu\nu_\omega}\phi}\,,
\end{align*}
where $\skp{\,}{}$ abbreviates the  $\fspaceL^2$-inner product and $\tilde{\Delta}_{+a-\iu\nu_\omega}$ is the dual operator to $\tilde{\Delta}_{-a+\iu\nu_\omega}$.
Splitting
\begin{align*}
\tilde{P}_\omega\tilde{G}_\omega=\tilde{\chi}_\omega\bar{P}\tilde{G}_\omega+
\tilde{\chi}_\omega \bat{\tilde{P}_\omega-\bar{P}}\tilde{G}_\omega+\nat{1-\tilde{\chi}}\tilde{P}_\omega\tilde{G}_\omega
\end{align*}
and using the support properties of $\tilde{\Delta}_{+a-\iu\nu_\omega}\phi$ as well as the approximation result from Theorem~\ref{Thm:ExistenceNonlWaves} we verify
\begin{align*}
\skp{\tilde{P}_\omega\tilde{G}_\omega}{\tilde{\Delta}_{+a-\iu\nu_\omega}\phi}\quad\xrightarrow{\;\;\omega\to\infty\;\;}\quad\skp{\bar{P}\tilde{G}_\infty}{-2\phi}\,,
\end{align*}
and conclude that the function $\tilde{T}_\infty$ with $\tilde{T}_\infty\at{\tilde{x}}:=\exp\at{-\eps_\infty\tilde{x}}\tilde{G}_\infty\at{\tilde{x}}$ satisfies the asymptotic ODE \eqref{Lem:LinODE.Props.Eqn1}. In other words, the nontrivial limit $\tilde{G}_\infty$ satisfies
\begin{align*}
\tilde{G}_\infty\at{\tilde{x}}=\exp\at{+\eps_\infty\tilde{x}}\Bat{
d_\even\bar{T}_\even\at{\tilde{x}}+d_\odd\bar{T}_\odd\at{\tilde{x}}}
\end{align*}
for some coefficients $\pair{d_\even}{d_\odd}\neq\pair{0}{0}$ and grows hence exponentially as $\tilde{x}\to+\infty$, see Lemma~\ref{Lem:LinODE.Props}. This, however, contradicts the affine bound \eqref{Lem:Compactness.PEqn3} for $\omega=\infty$.
\end{proof}
Lemma~\ref{Lem:Compactness} implies that the families $\nat{\tilde{G}_\omega}_\omega$ and $\nat{\tilde{T}_\omega}_\omega$ from \eqref{Eqn:DefTildeG} and \eqref{Eqn:DefTildeT} are compact with respect to local uniform convergence and have the same set of accumulation points. Moreover, the last step in the proof of Lemma~\ref{Lem:Compactness} reveals that any accumulation point solves the linearized  shape ODE \eqref{Lem:LinODE.Props.Eqn1}, but at this moment we do not yet know that the limit is uniquely determined by the normalization condition \eqref{Eqn:Normalization}. We thus refine our asymptotic analysis by means of ODE matching arguments as outlined in \S\ref{sect:LinODEs}.

%
%
%
\subsection{Asymptotic analysis of the eigenproblem}
\label{sect:StabilityODEs} %
%
%

A particular challenge is to exclude the case of large $\mu_\omega$, in which we would have $0<\delta_\omega\ll \eps_\omega\ll 1$. At the moment, however, we have to take into account a hypothetical scale separation between $\delta_\omega$ and $\eps_\omega$ and define a function $\breve{G}_\omega$ by
\begin{align}
\label{Eqn:DefBreveG}
\breve{G}_\omega:= \tilde{\Delta}_{-a+\iu\nu}\Bat{\tilde{H}_{\eps_\omega}\ast\bat{\breve{\chi}_\omega\tilde{P}_\omega\tilde{G}_\omega}}\,,
\end{align}
where the right hand side differs from   the   fixed point relation \eqref{Eqn:FPG.Int} by the cut-off function
\begin{align}
\label{Eqn:DefBreveQuant1}
\breve{\chi}_\omega:=\left\{\begin{array}{cl}1&\text{for $\tilde{x}\in J_\omega$,}\\
0&\text{else}
\end{array}
\right.
\end{align}
where
\begin{align}
\label{Eqn:DefBreveQuant2}
J_\omega:=\ccinterval{-\zeta_\omega}{\zeta_\omega}
\,,\qquad\quad
\zeta_\omega :=\xi_\omega \min\big\{1,\,\frac{a}{a+\mu_\om} \big\}=\min\big\{\xi_\omega,\,\frac{a}{2\eps_\omega}\big\}\,.
\end{align}
In consistency with \eqref{Eqn:DefTildeT} we also introduce
\begin{align}
\label{Eqn:DefBreveT}
\breve{T}_\omega\at{\tilde{x}}:=\exp\at{-\eps_\omega\tilde{x}}\breve{G}_\omega\at{\tilde{x}}\,,
\end{align}
which satisfies due \eqref{Eqn:DefBreveG} and Lemma~\ref{Lem:EV.FundSol} the differential relation
\begin{align}
\label{Eqn:LawBreveT}
\breve{T}_\omega^{\prime\prime}=\tilde{\Delta}_{\la_\omega}\bat{\breve{\chi}_\omega\tilde{P}_\omega\tilde{T}_\omega}
\end{align}
with the right hand side depending on $\tilde{T}_\om$. Notice that $0<\zeta_\omega\leq \xi_\omega$ and $J_\omega\subseteq I_\omega$ hold by construction and that $\nabs{\tilde{G}_\omega}$ bounds $\breve{T}_\omega$ on the interval $J_\omega$ but, at this moment, not necessarily on $I_\omega$.
\par
Our aim is to show that  $\breve{G}_\omega$ is a good approximation of $\tilde{G}_\omega$. This not only implies quite accurate approximation formulas for $\tilde{G}_\omega$ but finally enables us to prove that   $\mu_\delta$ and $\abs{\nu_\delta}$ are asymptotically very negative and small, respectively.
\begin{figure}[ht!] %
\centering{ %
\includegraphics[width=0.8\textwidth]{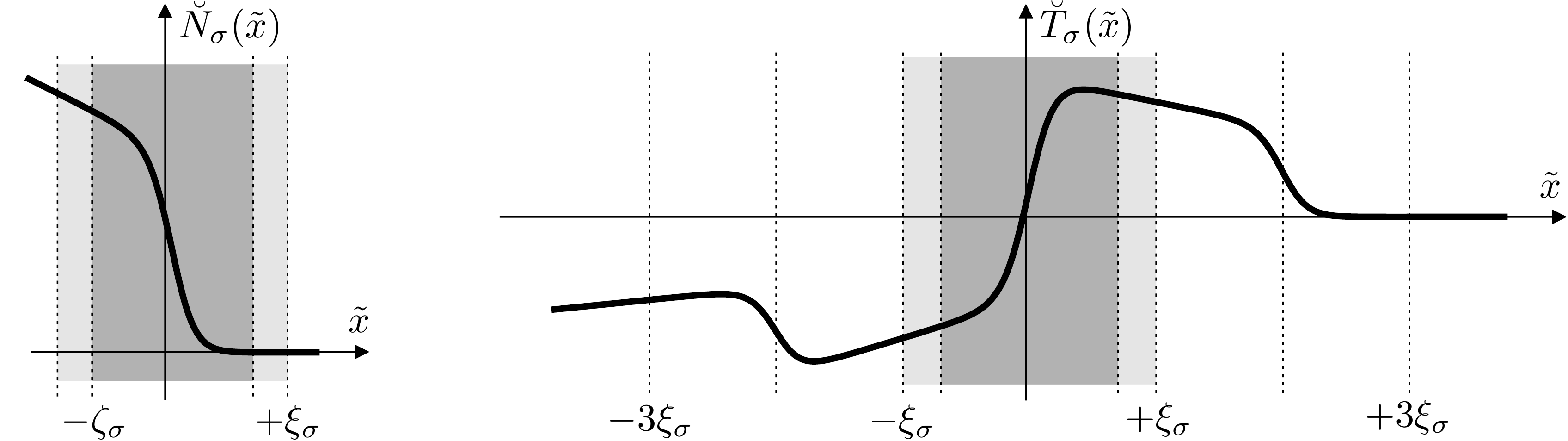} %
} %
\caption{ %
Schematic representation of the functions $\breve{N}_\omega$ and $\breve{T}_\omega$ as used in the proof of Lemma~\ref{Lem:Convergence}, see \eqref{Lem:Convergence.PEqn42} and \eqref{Lem:Convergence.FrmB}, with gray boxes indicating the intervals $J_\omega$ and $I_\omega$. By construction, $\breve{N}_\omega$ vanishes for $\tilde{x}>+\zeta_\omega$ and is affine for $\tilde{x}<-\zeta_\omega$ but our asymptotic analysis guarantees that it is basically constant as $\tilde{x}\to-\infty$ and hence that $\breve{T}_\omega$ equals $\tilde{S}_{*,\,\omega}$ up to small error terms.
} %
\label{Fig:ScaledEigenFcts} %
\end{figure} %
\begin{lemma}[asymptotic smallness of $\la_\omega$ and local asymptotics of $\Delta^{-1}S_\omega$]
\label{Lem:Convergence}
We have
\begin{align*}
\la_\omega\quad\xrightarrow{\quad\omega\to\infty\quad}\quad0
\end{align*}
and there exists a family of constants $\at{b_\omega}_\omega\subset\Rset$ such that $0<\lim_{\omega\to\infty}\abs{b_\omega}<\infty$  and
\begin{align}
\label{Lem:Convergence.Eqn1}
\delta_\omega \int\limits_{-3\xi_\omega}^{0}
\babs{\tilde{B}_\omega^\prime\at{\tilde{x}}}+
\babs{\tilde{B}_\omega\at{\tilde{x}}-b_\omega}\dint{\tilde{x}}+
\delta_\omega \int\limits_0^{+3\xi_\omega}
\babs{\tilde{B}_\omega^\prime\at{\tilde{x}}}+
\babs{\tilde{B}_\omega\at{\tilde{x}}}\dint{\tilde{x}}\quad\xrightarrow{\;\omega\to\infty\;}\quad 0\,,
\end{align}
where $\tilde{B}_\omega:=\tilde{\Delta}_{0}^{-1}\tilde{S}_\omega$ and
$\tilde{S}_\omega$ as in \eqref{Eqn:DefTildeS}.
\end{lemma}
\begin{proof}
Within this proof we set $l:=\min\{k,\,m\}$ and always assume that $\omega$ is sufficiently large.
\par
\ul{\emph{Preliminary error bounds}}: Young's inequality gives
\begin{align}
\label{Lem:Convergence.PEqn22}
\bnorm{\breve{G}_\omega}_\infty\leq \bnorm{\tilde{H}_{\eps_\omega}}_\infty\bnorm{\breve{\chi}_\omega\tilde{P}_\omega\tilde{G}_\omega}_1\leq C\eps_\omega^{-1}
\end{align}
thanks to the $J_\omega\subseteq I_\omega$, Lemma~\ref{Lem:EV.FundSol}, and the normalization condition \eqref{Eqn:Normalization}. Moreover, we define
\begin{align}
\label{Lem:Convergence.PEqn48}
\breve{M}_\omega := \tilde{H}_{\eps_\omega}\ast\bat{\breve{\chi}_\omega\tilde{P}_\omega \tilde{G}_\omega}\,,\qquad
\tilde{M}_\omega := \tilde{H}_{\eps_\omega}\ast\bat{\tilde{P}_\omega \tilde{G}_\omega}
\end{align}
and find
\begin{align*}
\breve{M}_\omega-\tilde{M}_\omega=\tilde{H}_{\eps_\omega}\ast\at{\sum_{i=1}^3\breve{F}_{i,\,\omega}}
\,,\qquad
\breve{G}_\omega-\tilde{G}_\omega=\tilde{\Delta}_{-a+\iu\nu_\omega}\Bat{\breve{M}_\omega-\tilde{M}_\omega}
\end{align*}
with  error terms
\begin{align*}
\breve{F}_{1,\,\omega}:=
\nat{1-\tilde{\chi}_\omega}\tilde{P}_\omega\tilde{G}_\omega
\,,\qquad
\breve{F}_{2,\,\omega}:=
\nat{\tilde{\chi}_\omega-\breve{\chi}_\omega}\nat{\tilde{P}_\omega-\bar{P}}\tilde{G}_\omega
\,,\qquad
\breve{F}_{3,\,\omega}:=
\nat{\tilde{\chi}_\omega-\breve{\chi}_\omega}\bar{P}\tilde{G}_\omega\,.
\end{align*}
Theorem~\ref{Thm:ExistenceNonlWaves} combined with Lemma~\ref{Lem:AsympODE.Props}, \eqref{Eqn:DefEps}, and \eqref{Lem:EV.FundSol.Eqn3} implies
\begin{align*}
\bnorm{\tilde{H}_{\eps_\omega}\ast\breve{F}_{1,\,\omega}}_\infty\leq \bnorm{\tilde{H}_{\eps_\omega}}_1\bnorm{\breve{F}_{1,\,\omega}}_\infty\leq C\eps_\omega^{-2}\delta_\omega^{m+2}\bnorm{\tilde{G}_\omega}_\infty\leq C\delta_\omega^m\bnorm{\tilde{G}_\omega}_\infty
\end{align*}
as well as
\begin{align*}
\bnorm{\tilde{H}_{\eps_\omega}\ast\breve{F}_{2,\,\omega}}_\infty\leq \bnorm{\tilde{H}_{\eps_\omega}}_\infty\bnorm{\breve{F}_{2,\,\omega}}_1\leq C\eps_\omega^{-1}\delta_\omega^{l}\bnorm{\nat{\tilde{\chi}_\omega-\breve{\chi}_\omega}\bar{Y}^{-1}}_\infty\bnorm{\tilde{G}_\omega}_\infty\leq C\delta_\omega^l\bnorm{\tilde{G}_\omega}_\infty\,,
\end{align*}
and
\begin{align*}
\bnorm{\tilde{H}_{\eps_\omega}\ast\breve{F}_{3,\,\omega}}_\infty\leq \bnorm{\tilde{H}_{\eps_\omega}}_\infty\bnorm{\breve{F}_{3,\,\omega}}_1\leq C\eps_\omega^{-1}\bnorm{\nat{\tilde{\chi}_\omega-\breve{\chi}_\omega}\bar{P}}_1\bnorm{\tilde{G}_\omega}_\infty\leq C \eps_\omega^m\bnorm{\tilde{G}_\omega}_\infty\,.
\end{align*}
In particular, we have
\begin{align}
\label{Lem:Convergence.PEqn43}
\bnorm{\tilde{M}_\omega-\breve{M}_\omega}_\infty \leq C\sum_{i=1}^3\bnorm{\tilde{H}_{\eps_\omega}\ast\breve{F}_{i,\,\omega}}_\infty\leq C \eps_\omega^l\bnorm{\tilde{G}_\omega}_\infty\,,
\end{align}
and with analogous estimates we show that
\begin{align}
\label{Lem:Convergence.PEqn44}
\bnorm{\tilde{M}_\omega^\prime-\breve{M}_\omega^\prime}_\infty \leq C\sum_{i=1}^3\bnorm{\tilde{H}_{\eps_\omega}^\prime\ast\breve{F}_{i,\,\omega}}_\infty\leq C \eps_\omega^{l+1}\bnorm{\tilde{G}_\omega}_\infty
\end{align}
because the Lebesgue estimates $\tilde{H}_{\eps_\omega}^\prime$ from Lemma~\ref{Lem:EV.FundSol} are better that those for $\tilde{H}_{\eps_\omega}$. The uniform $\fspaceL^p$-continuity of the discrete Laplacian $\tilde{\Delta}_{-a+\iu\nu_\omega}$ finally gives
\begin{align}
\label{Lem:Convergence.PEqn21}
\bnorm{\tilde{G}_\omega-\breve{G}_\omega}_\infty \leq C \eps_\omega^l\bnorm{\breve{G}_\omega}_\infty
\end{align}
after elementary transformations. Moreover, \eqref{Lem:Convergence.PEqn22} guarantees
\begin{align}
\label{Lem:Convergence.PEqn23}
C \eps_\omega^l\bnorm{\breve{G}_\omega}_\infty\leq C\eps_\omega^{l-1}\quad\xrightarrow{\;\;\omega\to\infty\;\;}\quad0
\end{align}
but we improve this bound below.
\par
\emph{\ul{Algebraic relations for $\breve{T}_\omega$}}:
As illustrated in Figure~\ref{Fig:ScaledEigenFcts}, we set
\begin{align}
\label{Lem:Convergence.PEqn42}
\breve{N}_\omega:=\tilde{H}_0\ast \bat{\breve{\chi}_\omega\tilde{P}_\omega\tilde{T}_\omega}\,,\qquad
\tilde{N}_\omega:=\tilde{H}_0\ast \bat{\tilde{P}_\omega\tilde{T}_\omega}\,.
\end{align}
The first formula reads
\begin{align}
\label{Lem:Convergence.FrmA}
\breve{N}_\omega\at{\tilde{x}}=
\int\limits_{\tilde{x}}^{+\infty} %
\bat{\tilde{y}-\tilde{x}}\breve{\chi}_\omega\at{\tilde{y}}\tilde{P}_\omega\at{\tilde{y}}
\tilde{T}_\omega\at{\tilde{y}}\dint\tilde{y}\,,
\end{align}
and since $\breve{\chi}_\omega$ is supported in the interval $J_\omega$ we find
\begin{align}
\label{Lem:Convergence.FrmC}
\breve{N}_\omega\at{\tilde{x}}=m_{1,\,\omega}-m_{0,\,\omega}\tilde{x}\quad \text{for}\quad \tilde{x}\leq- \zeta_\omega\,,\qquad
\breve{N}_\omega\at{\tilde{x}}=0\quad \text{for}\quad \tilde{x}\geq +\zeta_\omega\,,
\end{align}
where $m_{i,\omega}$ is shorthand for the moment integral
\begin{align}
\label{DefMoments}
m_{i,\,\omega}:=\int\limits_{J_\omega}\tilde{y}^i\tilde{P}_\omega\at{\tilde{y}}\tilde{T}_\omega\at{\tilde{y}}\dint{\tilde{y}}\,.
\end{align}
In other words, the function $\breve{N}_\omega$ is affine on the left and vanishes on the right of the compact interval $J_\omega$.   Since $\breve{T}_\omega$ also vanishes identically for large $\tilde{x}$ according to \eqref{Lem:EV.FundSol.Eqn2}+\eqref{Eqn:DefBreveG}+\eqref{Eqn:DefBreveT}, and because \eqref{Eqn:LawBreveT}+\eqref{Lem:Convergence.FrmA} imply  $\breve{T}^{\prime\prime}_\omega=\tilde{\Delta}_{\la_\omega}\breve{N}^{\prime\prime}_\omega$ on $\Rset$, we find
\begin{align}
\label{Lem:Convergence.FrmB}
\breve{T}_\omega =\tilde{\Delta}_{\la_\omega}\breve{N}_\omega
\end{align}
and compute
\begin{align}
\label{Lem:Convergence.PEqn1}
\begin{split}
\tfrac12\Bat{\breve{T}_\omega^\prime\nat{+\zeta_\omega}-
\breve{T}_\omega^\prime\nat{-\zeta_\omega}}&=-m_{0,\,\omega}\,,\\
\tfrac12\Bat{\breve{T}_\omega^\prime\nat{+\zeta_\omega}+
\breve{T}_\omega^\prime\nat{-\zeta_\omega}}&=\bat{1-\mhexp{-\la_\omega}}m_{0,\,\omega}\,.
\end{split}
\end{align}
Similarly we verify
\begin{align*}
\tfrac12\Bat{\breve{T}_\omega\nat{+\zeta_\omega}-
\breve{T}_\omega\nat{-\zeta_\omega}}
&= %
m_{1,\,\omega}+\bat{1-\mhexp{-\la_\omega}}\zeta_\omega m_{0,\,\omega}\,,
\\
\tfrac12\Bat{\breve{T}_\omega\nat{+\zeta_\omega}+\breve{T}_\omega\nat{-\zeta_\omega}}
&=
-\bat{1-\mhexp{-\la_\omega}}m_{1,\,\omega}
+\bat{2\mhexp{-\la_\omega}\tilde{\xi}_\omega-\zeta_\omega}m_{0,\,\omega}\,,
\end{align*}
which in turn gives
\begin{align}
\label{Lem:Convergence.PEqn2}
\begin{split}
\tfrac12\Bat{\breve{T}_\omega^\flat\nat{+\zeta_\omega}-\breve{T}_\omega^\flat\nat{-\zeta_\omega}}&=-m_{1,\,\omega}\,,\\
\tfrac12\Bat{\breve{T}_\omega^\flat\nat{+\zeta_\omega}+\breve{T}_\omega^\flat\nat{-\zeta_\omega}}&=\bat{1-\mhexp{-\la_\omega}}m_{1,\,\omega}-2\mhexp{-\la_\omega}\tilde{\xi}_\omega m_{0,\,\omega}
\end{split}
\end{align}
thanks to \eqref{Lem:Convergence.PEqn1}.
\par
\emph{\ul{Asymptotic formula for $\breve{T}_\omega$ on $J_\omega$}}: The formulas \eqref{Eqn:DefBreveT}+\eqref{Eqn:LawBreveT}+\eqref{Lem:Convergence.FrmA}+\eqref{Lem:Convergence.FrmB} guarantee
\begin{align}
\label{Lem:Convergence.FromG}
\breve{T}_\omega^{\prime\prime}\at{\tilde{x}}=-2 \breve{N}_\omega^{\prime\prime}\at{\tilde{x}}=-2\tilde{P}_\omega\at{\tilde{x}}\tilde{T}_\omega\at{\tilde{x}}
=-2\bar{P}\at{\tilde{x}}\breve{T}_\omega\at{\tilde{x}}+
\breve{E}_\omega\at{\tilde{x}}\qquad \text{for}\quad \tilde{x}\in J_\omega\,,
\end{align}
where the error terms
\begin{align*}
\breve{E}_\omega:= 2\bat{\bar{P}-\tilde{P}_\omega}\breve{T}_\omega+
2\tilde{P}_\omega\bat{\breve{T}_\omega-\tilde{T}_\omega}
\end{align*}
are pointwise bounded by
\begin{align*}
\babs{\breve{E}_\omega\at{\tilde{x}}}&\leq
C\bnorm{\tilde{G}_\omega}_\infty\babs{\bar{P}\at{\tilde{x}}-\tilde{P}_\omega\at{\tilde{x}}}
+
C\eps_\omega^{l}\bnorm{\breve{G}_\omega}_\infty\babs{\tilde{P}_\omega\at{\tilde{x}}}
\\&\leq  C\bnorm{\breve{G}_\omega}_\infty\at{
\babs{\bar{P}\at{\tilde{x}}-\tilde{P}_\omega\at{\tilde{x}}}+\eps_\omega^{l}\babs{\bar{P}\at{\tilde{x}}}}
\end{align*}
because \eqref{Eqn:ApproxEstSbreve} and \eqref{Lem:Convergence.PEqn21} imply
\begin{align}
\label{Eqn:ApproxEstSbreve}
\sup_{x\in J_\omega}\babs{\breve{T}_\omega\at{\tilde{x}}-\tilde{T}_\omega\at{\tilde{x}}}\leq C\eps_\omega^{l}\bnorm{\breve{G}_\omega}_\infty
\end{align}
since $\exp\at{\eps_\omega\tilde{x}}$ is uniformly bounded for $\tilde{x}\in J_\omega$. The function  $\breve{T}_\omega$ restricted to the interval $J_\omega$ can be regarded as an approximate solution to the linearized asymptotic ODE \eqref{Lem:LinODE.Props.Eqn1}.   Using   the Duhamel principle as well as the properties of $\bar{T}_\even$ and $\bar{T}_\odd$ --- see Lemma~\ref{Lem:LinODE.Props} and notice that the Wronski determinant is constant and has been normalized --- we arrive at the representation formula
\begin{align}
\label{Lem:Convergence.PEqn3a}
\begin{pmatrix}\breve{T}_\omega\at{\tilde{x}}\\\breve{T}_\omega^\prime\at{\tilde{x}}
\end{pmatrix}
=\begin{pmatrix}
\bar{T}_\even\at{\tilde{x}}&\bar{T}_\odd\at{\tilde{x}}\\
\bar{T}_\even^\prime\at{\tilde{x}}&\bar{T}_\odd^\prime\at{\tilde{x}}\\
\end{pmatrix}
\cdot
\Bigg(
\begin{pmatrix}
\breve{T}_\omega\at{0}\\\breve{T}_\omega^\prime\at{0}
\end{pmatrix}
+\int\limits_{0}^{\tilde{x}}\begin{pmatrix}
+\bar{T}_\odd^\prime\at{\tilde{y}}&-\bar{T}_\odd\at{\tilde{y}}\\
-\bar{T}_\even^\prime\at{\tilde{y}}&+\bar{T}_\even\at{\tilde{y}}\\
\end{pmatrix}
\cdot
\begin{pmatrix}
0\\\breve{E}_\omega\at{\tilde{y}}
\end{pmatrix}
\dint{\tilde{y}}\Bigg) \,.%
\end{align}
Moreover, the approximation results in Theorem~\ref{Thm:ExistenceNonlWaves} guarantee
\begin{align}
\label{Lem:Convergence.PEqn3b}
\begin{split}
\int\limits_0^{\tilde{x}} \babs{\breve{E}_\omega\at{\tilde{y}}}\Bat{
\babs{\bar{T}_\even\at{\tilde{y}}}+
\babs{\bar{T}_\odd\at{\tilde{y}}}
}\dint{\tilde{y}}&\leq C  \int\limits_{I_\omega}\babs{\breve{E}_\omega\at{\tilde{y}}} \bar{Y}\at{\tilde{y}}\dint{y}\\
&\leq C\bnorm{\tilde{G}_\omega}_\infty \at{
\bnorm{\tilde{\chi}_\omega\bat{\tilde{P}_\omega-\bar{P}}\bar{Y}}_1+\eps_\omega^l
\bnorm{\tilde{\chi}_\omega\bar{P}\bar{Y}}_1}\\&\leq C\eps_\omega^l\bnorm{\tilde{G}_\omega}_\infty \,,
\end{split}
\end{align}
and we get
\begin{align}
\label{Lem:Convergence.PEqn4}
\breve{T}_\omega^\at{\cdot}\at{\tilde{x}}
= %
\bar{T}_\even^\at{\cdot}\at{\tilde{x}}\Bat{\breve{T}_\omega\at{0}+
\bDO{\eps_\omega^l\bnorm{\breve{G}_\omega}_\infty}}
+ %
\bar{T}_\odd^\at{\cdot}\at{\tilde{x}}\Bat{\breve{T}_\omega^\prime\at{0}+\bDO{\eps_\omega^l\bnorm{\breve{G}_\omega}_\infty}}\,,
\end{align}
where $^\at{\cdot}$ stands for either $^\prime$ or $^\flat$. We thus reformulate the four algebraic equations in  \eqref{Lem:Convergence.PEqn1}, \eqref{Lem:Convergence.PEqn2} as
\begin{align}
\label{Lem:Convergence.Jump1} %
-m_{0,\,\omega} &= \bar{T}^\prime_\even\bat{\zeta_\omega}\breve{T}_\omega\at{0}+\bDO{\eps_\omega^l\bnorm{\breve{G}_\omega}_\infty}\,,\\ %
\label{Lem:Convergence.Jump2} %
\bat{1-\mhexp{-\la_\omega}}m_{0,\,\omega} &=\bar{T}_\odd^\prime\bat{\zeta_\omega}\breve{T}_\omega^\prime\at{0}+\bDO{\eps_\omega^l\bnorm{\breve{G}_\omega}_\infty}
\,,\\
\label{Lem:Convergence.Jump3} %
-m_{1,\,\omega}&=\bar{T}_\odd^{\flat}\bat{\zeta_\omega} \breve{T}_\omega^\prime\at{0}+\bDO{\eps_\omega^l\bnorm{\breve{G}_\omega}_\infty}
\,,\\%
\label{Lem:Convergence.Jump4} %
\bat{1-\mhexp{-\la_\omega}}m_{1,\,\omega}-2\xi_\omega\mhexp{-\la_\omega}m_{0,\,\omega}&=
\bar{T}_\even^{\flat}\bat{\zeta_\omega}\breve{T}_\omega\at{0}+
\bDO{\eps_\omega^l\bnorm{\breve{G}_\omega}_\infty}\,,
\end{align}
where Lemma~\ref{Lem:LinODE.Props} ensures
\begin{align}
\label{Lem:Convergence.FrmE} %
\bar{T}_\even^{\prime}\bat{\zeta_\omega}=  -\sqrt{\frac{m}{m+1}}  +\nDO{\eps_\omega^m}\,,\qquad
\bar{T}_\even^{\flat}\bat{\zeta_\omega}=  \gamma +\nDO{\eps_\omega^{m-1}}
\end{align}
and
\begin{align}
\label{Lem:Convergence.FrmF} %
\bar{T}_\odd^{\prime}\bat{\zeta_\omega}=0+\nDO{\eps_\omega^{m+1}}\,,\qquad
\bar{T}_\odd^{\flat}\bat{\zeta_\omega}=  -\sqrt{\frac{m+1}{m}} +\nDO{\eps_\omega^{m}}\,,
\end{align}
  here $\gamma$ depends on $m$ only.   Combining the above Duhamel formula with \eqref{Lem:Compactness.Eqn2}+\eqref{Eqn:DefBreveT}+\eqref{Lem:Convergence.PEqn23} we first deduce
\begin{align}
\label{Lem:Convergence.PEqn5a}
\limsup\limits_{\omega\to\infty}\Bat{\babs{\breve{T}_\omega\at{0}}+\babs{\breve{T}_\omega^\prime\at{0}}}<\infty\,,
\end{align}
and afterwards
\begin{align}
\label{Lem:Convergence.PEqn5b}
\liminf_{\omega\to\infty}\Bat{\babs{\breve{T}_\omega\at{0}}+\babs{\breve{T}_\omega^\prime\at{0}}}>0
\end{align}
because otherwise both $\breve{G}_\omega$ and $\tilde{G}_\omega$ were converging locally uniform to $0$ and hence violating the normalization condition \eqref{Eqn:Normalization}.
\par
\ul{\emph{Refined error bounds}}:
We already observed in \eqref{Lem:Convergence.FromG} that $-2\breve{N}_\omega^{\prime\prime}\at{\tilde{x}}=\breve{T}_\omega^{\prime\prime}\at{\tilde{x}}$ holds for $\tilde{x}\in J_\omega$, so \eqref{Lem:Convergence.FrmC} and \eqref{Lem:Convergence.PEqn4} provide
\begin{align*}
-2\breve{N}_\omega^{\prime}\at{\tilde{x}}&= %
\breve{T}_\omega\at{0}\Bat{\bar{T}_\even^{\prime}\at{\tilde{x}}-
\bar{T}_\even^{\prime}\nat{\zeta_\omega}}
+ %
\breve{T}_\omega^\prime\at{0}\Bat{\bar{T}_\odd^\prime\at{\tilde{x}}-
\bar{T}_\odd^\prime\nat{\zeta_\omega}}+\bDO{\eps_\omega^l\bnorm{\breve{G}_\omega}_\infty}\,.
\end{align*}
This gives
\begin{align}
\label{Lem:Convergence.FrmG}
\sup\limits_{\tilde{x}\in J_\omega}\babs{\breve{N}_\omega\at{\tilde{x}}}
=\bDO{\eps_\omega^{-1}\babs{m_{0,\,\omega}}}+\nDO{1}+\bDO{\eps_\omega^{l-1}\bnorm{\breve{G}_\omega}_\infty}
\end{align}
after integration, where we also used \eqref{Lem:Convergence.FrmC}+\eqref{Lem:Convergence.Jump1}+\eqref{Lem:Convergence.PEqn5a} as well as
\begin{align*}
\int\limits_{J_\omega}\babs{\bar{T}_\even^\prime\bat{\tilde{x}}-\bar{T}_\even^\prime\bat{\zeta_\omega}}\dint{\tilde{x}}
=\nDO{\eps_\omega^{-1}}\,,\qquad
\int\limits_{J_\omega}\babs{\bar{T}_\odd^\prime
\bat{\tilde{x}}-\bar{T}_\odd^\prime\bat{\zeta_\omega}}\dint{\tilde{x}}=
\nDO{1}\,,
\end{align*}
which is a consequence of Lemma~\ref{Lem:AsympODE.Props}. Furthermore, by construction we have
\begin{align*}
\breve{M}_\omega\at{\tilde{x}}=\exp\at{\eps_\omega\tilde{x}}\breve{N}_\omega\at{\tilde{x}}\,,\qquad
\breve{G}_\omega=\tilde{\Delta}_{-a+\iu\nu_\omega}\breve{M}_\omega
\end{align*}
and verify with \eqref{Lem:Convergence.FrmA}+\eqref{Lem:Convergence.Jump3}+\eqref{Lem:Convergence.PEqn5a}+\eqref{Lem:Convergence.PEqn5b}+\eqref{Lem:Convergence.FrmG} the estimate
\begin{align*}
\bnorm{\breve{M}_\omega}_\infty &\leq
\babs{m_{0,\,\omega}}  \sup_{\tilde{x}\leq
-\zeta_\omega}\abs{\tilde{x}}\exp\bat{\eps_\omega\tilde{x}}
+ %
\babs{m_{1,\,\omega}}  \sup_{\tilde{x}\leq
-\zeta_\omega}\exp\bat{\eps_\omega{\tilde{x}}} + \sup\limits_{\tilde{x}\in J_\omega}\babs{\breve{M}_\omega\at{\tilde{x}}}
\\
&\leq
C\eps_\omega^{-1}\babs{m_{0,\,\omega}}+C\babs{m_{1,\,\omega}}
+C\sup\limits_{\tilde{x}\in J_\omega}\babs{\breve{N}_\omega\at{\tilde{x}}}
\\&=\bDO{\eps_\omega^{-1}\babs{m_{0,\,\omega}}}+\bDO{1}
+\bDO{\eps_\omega^{l-1}\bnorm{\breve{G}_\omega}_\infty} \,,
\end{align*}
which finally implies
\begin{align}
\label{Lem:Convergence.PEqn8}
\bnorm{\breve{G}_\omega}_\infty \leq C \bnorm{\breve{M}_\omega}_\infty \leq
\bDO{\eps_\omega^{-1}|m_{0,\,\omega}|}+\DO{1}+\bDO{\eps_\omega^{l-1}\bnorm{\breve{G}_\omega}_\infty}\,.
\end{align}
On the other hand, multiplying \eqref{Lem:Convergence.Jump4} with $\beta_\omega\delta_\omega$, subtracting \eqref{Lem:Convergence.Jump2}, eliminating $\breve{T}_\omega\at{0}$ and $\breve{T}_\omega^\prime\at{0}$ by means of \eqref{Lem:Convergence.Jump1} and \eqref{Lem:Convergence.Jump3}, and inserting \eqref{Lem:Convergence.FrmE}+\eqref{Lem:Convergence.FrmF} we obtain
\begin{align}
\notag%
\beta_\omega\delta_\omega\bat{1-\mhexp{-\la_\omega}}m_{1,\,\omega}-m_{0,\,\omega}&=\bDO{\delta_\omega |m_{0,\,\omega}|}+
\bDO{\eps_\omega^{m+1}}+\bDO{\eps_\omega^{l}\bnorm{\breve{G}_\omega}_\infty}\,,
\end{align}
and hence
\begin{align}
\label{Lem:Convergence.PEqn9}
\eps_\omega^{-1}\babs{m_{0,\,\omega}}=
\bDO{\eps_\omega^{l-1}\bnorm{\breve{G}_\omega}_\infty}+\DO{1}
\end{align}
thanks to
\begin{align*}
\abs{\eps_\omega^{-1}\delta_\omega\at{1-\mhexp{-\la_\omega}}}\,\leq C\frac{1+\mhexp{-\mu_\omega}}{a+\mu_\omega}\, \,{\leq}\,\,C\,.
\end{align*}
Combining \eqref{Lem:Convergence.PEqn8} and \eqref{Lem:Convergence.PEqn9} we finally arrive at
\begin{align}
\label{Lem:Convergence.PEqn17}
\bnorm{\breve{G}_\omega}_\infty=\DO{1}\,,
\end{align}
which is better than \eqref{Lem:Convergence.PEqn22} and reveals that $\eps_\omega^l\bnorm{\breve{G}_\omega}_\infty=\DO{\eps_\omega^l}$.
\par
\ul{\emph{Smallness of $\la_\omega$ and $m_{0,\,\omega}$}}:
From \eqref{Lem:Convergence.Jump1}+\eqref{Lem:Convergence.PEqn9}+\eqref{Lem:Convergence.PEqn17}  we infer
\begin{align*}
\babs{m_{0,\,\omega}}+\babs{\breve{T}_\omega\at{0}}= \nDO{\eps_\omega}\,,
\end{align*}
so \eqref{Lem:Convergence.Jump3}+\eqref{Lem:Convergence.PEqn5a}+\eqref{Lem:Convergence.PEqn5b} provide
\begin{align}
\label{Lem:Convergence.PEqn14}
0<\liminf_{\om\to\infty}\babs{m_{1,\,\omega}}\,,\qquad \limsup_{\om\to\infty}\babs{m_{1,\,\omega}}<\infty
\end{align}
as well as a similar result for $\breve{T}^\prime_\omega\at{0}$. Moreover, \eqref{Lem:Convergence.Jump1}+\eqref{Lem:Convergence.Jump4}+\eqref{Lem:Convergence.PEqn17}+\eqref{Lem:Convergence.PEqn14} imply
\begin{align}
\label{Lem:Convergence.PEqn6}
\beta_\omega\delta_\omega\bat{1-\mhexp{-\la_\omega}}m_{1,\,\omega} - \mhexp{-\la_\omega}m_{0,\,\omega}=\bDO{\delta_\omega\babs{m_{0,\,\omega}}}+\bDO{\delta_\omega\eps_\omega^l}\,.
\end{align}
Now suppose for contradiction that $\liminf_{\omega\to\infty}\abs{\la_\omega}>0$ and recall that $\abs{\nu_\omega}\leq\pi$ holds by Assumption~\ref{Ass:Eigensystem}. Then there exists a subsequence for $\omega\to\infty$ such that the right hand side in \eqref{Lem:Convergence.PEqn6}, which is of order $\nDo{\delta_\omega}$, can be subsumed into the first term on the left hand side, see \eqref{Lem:Convergence.PEqn14}. By \eqref{Lem:Convergence.Jump2} and \eqref{Lem:Convergence.PEqn6} we therefore get
\begin{align}
\label{Lem:Convergence.PEqn11}
\delta_\omega\abs{\frac{1-\mhexp{-\la_\omega}}{\mhexp{-\la_\omega}}}\leq C\babs{m_{0,\,\omega}}\,,\qquad \babs{m_{0,\,\omega}}\leq C\frac{\eps_\omega^l}{\abs{1-\mhexp{-\la_\omega}}}
\end{align}
 and rearranging terms we find
\begin{align}
\label{Lem:Convergence.PEqn12}
\delta_\omega\babs{\mhexp{\mu_\omega+\iu\nu_\omega}+\mhexp{-\mu_\omega-\iu\nu_\omega}-2}\leq C\eps_\omega^l = C\bat{\delta_\omega a+\delta_\omega\mu_\omega}^l\,,
\end{align}
where the estimates $l>1$, $\liminf\mu_\omega >-a $, and $\limsup_{\omega\to\infty}\abs{\nu_\omega}\leq \pi$ hold by Assumptions \ref{Ass.Pot} and \ref{Ass:Eigensystem}. Assuming that either $\lim_{\omega\to\infty}\la_\omega=\la_\infty=\mu_\infty+\iu\nu_\infty\neq0$ or
$\lim\inf_{\omega\to\infty}\mu_\omega=\infty$ holds for a subsequence
we can simplify \eqref{Lem:Convergence.PEqn12} to
\begin{align*}
1\leq C\delta_\omega^{l-1}\qquad\text{or}\qquad
\frac{\mhexp{\mu_\omega}}{\mu_\omega^l}\leq C\delta_\omega^{l-1}\,,
\end{align*}
respectively, and obtain in any of these two cases a contradiction along the chosen subsequence. In particular, we have now shown that
\begin{align}
\label{Lem:Convergence.PEqn13}
\babs{\la_\omega}=\nDo{1}\,,\qquad \eps_\omega=\delta_\omega\bat{a+\nDo{1}}\,.
\end{align}
We finally distinguish the following two cases: For $\nabs{\la_\omega}=\DO{\delta_\omega^l}$ we have $\delta_\omega\nabs{1-\mhexp{-\la_\omega}}\nabs{m_{1,\,\omega}}=\DO{\delta_\omega^{l+1}}$, so \eqref{Lem:Convergence.PEqn6} combined with \eqref{Lem:Convergence.PEqn13} immediately ensures that  $ \nabs{m_{0,\delta }}\,{=}\,\nDO{\delta_\omega^{l+1}}$. Otherwise we can  --- letting the second term on the right hand side of \eqref{Lem:Convergence.PEqn6} be consumed by the first on the left hand side and the term involving $m_{0,\,\omega}$ by the second on the right hand side--- still rely on $\eqref{Lem:Convergence.PEqn11}_2$ to estimate $|m_{0,\,\omega}|$  and obtain after simplification with \eqref{Lem:Convergence.PEqn13} and elementary computations
\begin{align}
\label{Lem:Convergence.PEqn41}
\babs{\la_\omega}=\bDO{\delta_\omega^{\at{l-1}/2}}\,,\qquad \babs{m_{0,\,\omega}}=\bDO{\delta_\omega^{\at{l+1}/2}}\,.
\end{align}
These formulas cover also the first case and imply
\begin{align}
\label{Lem:Convergence.PEqn31}
\frac{\zeta_
\omega}{\xi_\omega}=1+\bDO{\delta_\omega^{\at{l-1}/2}}\,,\qquad \babs{\breve{T}_\omega\at{0}}=\bDO{\delta_\omega^{\at{l+1}/2}}\,,
\end{align}
where the first equation follows from \eqref{Eqn:DefBreveQuant2}, and the second one from \eqref{Lem:Convergence.Jump1}+\eqref{Lem:Convergence.FrmE}+\eqref{Lem:Convergence.PEqn17} and $\at{l+1}/2<l$.
\par
\emph{\ul{Asymptotics of $\breve{T}_\omega$}}:
Since $\babs{\tilde{G}_\omega}\leq C\bar{Y}$ holds by Lemma~\ref{Lem:AsympODE.Props},   Lemma~\ref{Lem:Compactness} and \eqref{Lem:Convergence.PEqn13}, we deduce from \eqref{Thm:ExistenceNonlWaves.Eqn1} the estimate
\begin{align*}
\abs{\bnorm{\tilde{\chi}_\omega \tilde{P}_\omega\tilde{G}_\omega }_1-
\bnorm{\tilde{\chi}_\omega \bar{P}\tilde{G}_\omega }_1}\leq C \bnorm{\tilde{\chi}_\omega \bat{\tilde{P}_\omega-\bar{P}}\bar{Y} }_1\leq C\delta_\omega^l\,,
\end{align*}
and in view of \eqref{Lem:Compactness.Eqn2}+\eqref{Eqn:DefBreveQuant1}+\eqref{Lem:Convergence.PEqn13} we find
\begin{align*}
\abs{\bnorm{\tilde{\chi}_\omega \bar{P}\tilde{G}_\omega }_1-\bnorm{\breve{\chi}_\omega \bar{P}\tilde{G}_\omega }_1}\leq C\babs{\xi_\omega-\zeta_\omega}\sup_{\zeta_\omega\leq \babs{\tilde{x}}\leq \xi_\omega}\abs{\bar{P}\at{\tilde{x}}\bar{Y}\at{\tilde{x}}}\leq C \frac{\babs{\xi_\omega-\zeta_\omega}\xi_\omega}{\zeta_\omega^{m+2}}\leq C\delta_\omega^m\,.
\end{align*}
Moreover, \eqref{Lem:Convergence.PEqn21} and \eqref{Lem:Convergence.PEqn17} ensure
\begin{align*}
\abs{
\bnorm{\breve{\chi}_\omega \bar{P}\tilde{G}_\omega }_1-
\bnorm{\breve{\chi}_\omega \bar{P}\breve{G}_\omega }_1
}\leq \bnorm{\tilde{\chi}_\omega \bar{P}}_1\bnorm{\tilde{G}_\omega - \breve{G}_\omega }_\infty\leq C\delta_\omega^l
\end{align*}
while a simply Taylor argument yields
\begin{align*}
\abs{
\bnorm{\breve{\chi}_\omega \bar{P}\breve{G}_\omega }_1-
\bnorm{\breve{\chi}_\omega \bar{P}\breve{T}_\omega }_1
}\leq C\eps_\omega \int\limits_{-\zeta_\omega}^{+\zeta_\omega}\babs{\tilde{x}\bar{P}\at{\tilde{x}}\breve{G}\at{\tilde{x}}}\dint{\tilde{x}}\leq C\delta_\omega\int\limits_{\Rset}\nabs{\tilde{x}}^2\bar{P}\at{\tilde{x}}\dint{\tilde{x}}\leq C\delta_\omega\,.
\end{align*}
Inserting all these partial results into the normalization condition \eqref{Eqn:Normalization} we finally get
\begin{align}
\label{Lem:Convergence.PEqn32}
\bnorm{\breve{\chi}_\omega \bar{P}\breve{T}_\omega}_1=1+\DO{\delta_\omega}\,.
\end{align}
On the other hand, \eqref{Lem:Convergence.PEqn21}+\eqref{Lem:Convergence.PEqn3a}+\eqref{Lem:Convergence.PEqn3b}+\eqref{Lem:Convergence.PEqn17} provide the pointwise estimate
\begin{align*}
\breve{T}_\omega\at{\tilde{x}}
= %
\bar{T}_\even\at{\tilde{x}}\Bat{\breve{T}_\omega\at{0}+
\bDO{\delta_\omega^l}}
+ %
\bar{T}_\odd\at{\tilde{x}}\Bat{\breve{T}_\omega^\prime\at{0}+\bDO{\delta_\omega^l}}
\end{align*}
for all $\tilde{x}\in J_\omega$, and in view of the asymptotic tail behavior of all involved functions --- see Lemmas~\ref{Lem:AsympODE.Props}   and \ref{Lem:LinODE.Props} --- we finally arrive at
\begin{align}
\label{Lem:Convergence.PEqn33}
\bnorm{\breve{\chi}_\omega \bar{P}\breve{T}_\omega}_1=\babs{\breve{T}_\omega\at{0}} \bnorm{\bar{P}\bar{T}_\even}_1 +
\babs{\breve{T}_\omega^\prime\at{0}} \bnorm{\bar{P}\bar{T}_\odd}_1+\nDO{\delta_\omega^l}\,.
\end{align}
Combining \eqref{Lem:Convergence.PEqn31}+\eqref{Lem:Convergence.PEqn32}+\eqref{Lem:Convergence.PEqn33} and \eqref{Lem:Convergence.Jump3}+\eqref{Lem:Convergence.FrmF} reveals
\begin{align*}
\babs{\breve{T}_\omega^\prime\at{0}}=\bnorm{\bar{P}\bar{T}_\odd}_1^{-1}+\DO{\delta_\omega}\,,\qquad
\abs{m_{1,\,\omega}}=c_\odd+\DO{\delta_\omega}\,,\qquad c_\odd := \bnorm{\bar{P}\bar{T}_\odd}_1^{-1}\bar{T}_\odd^\flat\at{\infty}>0
\,,
\end{align*}
and we conclude that the restriction of $\breve{T}_\omega$ to the interval $J_\omega$ is --- up to small error terms and the undefined sign --- a certain multiple of $\bar{T}_\odd$.
\par
\emph{\ul{Estimates for $\tilde{B}_\omega$ on the interval $3I_\omega$}}:
Using \eqref{Lem:Convergence.FrmA}, the support properties of $\breve{N}_\omega$ from \eqref{Lem:Convergence.FrmC}, and
the bound
\begin{align*}
\sup_{\abs{\tilde{x}}\leq 3\xi_\omega}\babs{\tilde{T}_\omega\at{\tilde{x}}}=
\sup_{\abs{\tilde{x}}\leq 3\xi_\omega}\babs{\exp\at{-\eps_\omega\tilde{x}}\tilde{G}_\omega\at{\tilde{x}}}\leq C\exp\at{3\eps_\omega\xi_\omega}\bnorm{\tilde{G}_\omega}_\infty\leq C\
\end{align*}
we find
\begin{align*}
\int\limits_{0}^{+3\xi_\omega}\abs{\breve{N}_\omega\at{\tilde{x}}}\dint{\tilde{x}}\leq
\int\limits_{0}^{+\zeta_\omega} %
\babs{\tilde{P}_\omega\at{\tilde{y}}\tilde{T}_\omega\at{\tilde{y}}}
\int\limits_{0}^{\tilde{y}}\bat{\tilde{y}-\tilde{x}}\dint{\tilde{x}}
\dint{\tilde{y}}
\leq C \int\limits_{0}^{+\zeta_\omega}\tilde{y}^2\babs{\tilde{P}_\omega\at{\tilde{y}}}\dint{\tilde{y}}
\leq C \int\limits_{0}^{\infty}\tilde{y}^2\bar{P}\at{\tilde{y}}\dint{\tilde{y}}
\leq C
\end{align*}
due to \eqref{Eqn:DefBarPZ} and Theorem~\ref{Thm:ExistenceNonlWaves}, as well as
\begin{align*}
\int\limits_{-3\xi_\omega}^0\abs{\breve{N}_\omega\at{\tilde{x}}-m_{1,\,\omega}+m_{0,\,\omega}\tilde{x}}\dint{\tilde{x}}&\leq
\int\limits_{-\zeta_\omega}^0\int\limits_{-\zeta_\omega}^{\tilde{x}}
\at{\tilde{x}-\tilde{y}}\babs{\tilde{P}_\omega\at{\tilde{y}}\tilde{T}_\omega\at{\tilde{y}}}
\dint{\tilde{y}}
\dint{\tilde{x}}
\\&\leq
C\int\limits_{-\zeta_\omega}^0
\babs{\tilde{P}_\omega\at{\tilde{y}} }\int\limits_{\tilde{y}}^{0}
\at{\tilde{x}-\tilde{y}}
\dint{\tilde{x}
}\dint{\tilde{y}}\leq C
\end{align*}
thanks to \eqref{Lem:Convergence.FrmC}. Similarly, we estimate
\begin{align*}
\int\limits_{0}^{+3\xi_\omega}\abs{\breve{N}_\omega^\prime\at{\tilde{x}}}\dint{\tilde{x}}\leq
\int\limits_{0}^{+\zeta_\omega}\int\limits_{\tilde{x}}^{+\zeta_\omega} %
\babs{\tilde{P}_\omega\at{\tilde{y}}\tilde{T}_\omega\at{\tilde{y}}}
\dint{\tilde{y}}\dint{\tilde{x}}
\leq C \int\limits_{0}^{+\zeta_\omega}\babs{\tilde{y}\tilde{P}_\omega\at{\tilde{y}}}\dint{\tilde{y}}
\leq C
\end{align*}
as well as
\begin{align*}
\int\limits_{-3\xi_\omega}^0\abs{\breve{N}_\omega^\prime\at{\tilde{x}}+m_{0,\,\omega}}\dint{\tilde{x}}&\leq
\int\limits_{-\zeta_\omega}^0\int\limits_{-\zeta_\omega}^{\tilde{x}}
\babs{\tilde{P}_\omega\at{\tilde{y}}\tilde{T}_\omega\at{\tilde{y}}}
\dint{\tilde{y}}
\dint{\tilde{x}}
\leq C \int\limits_{-\zeta_\omega}^0\babs{\tilde{y}\tilde{P}_\omega\at{\tilde{y}}}\dint{\tilde{y}}
\leq C\,.
\end{align*}
Combining these estimates with setting $b_\omega:=m_{1,\,\omega}$ we arrive at
\begin{align}
\label{Lem:Convergence.PEqn45}
\delta_\omega \int\limits_{-3\xi_\omega}^{0}
\babs{\breve{N}_\omega^\prime\at{\tilde{x}}}+
\babs{\breve{N}_\omega\at{\tilde{x}}-b_\omega}\dint{\tilde{x}}+
\delta_\omega \int\limits_0^{+3\xi_\omega}
\babs{\breve{N}_\omega^\prime\at{\tilde{x}}}+
\babs{\breve{N}_\omega\at{\tilde{x}}}\dint{\tilde{x}}\quad\xrightarrow{\;\omega\to\infty\;}\quad 0\,,
\end{align}
where we used that \eqref{Lem:Convergence.PEqn41} implies
\begin{align*}
\delta_\omega \int\limits_{-3\xi_\omega}^0\abs{m_{0,\,\omega}}+\abs{m_{0,\,\omega}\tilde{x}}\dint{\tilde{x}}\leq C\delta_\omega\delta_\omega^{\at{l+1}/2}\Bat{\delta_\omega^{-1}+\delta_\omega^{-2}}\quad
\xrightarrow{\;\omega\to\infty\;}\quad 0\,.
\end{align*}
On the other hand,   writing the convolutions in \eqref{Lem:Convergence.PEqn48} and \eqref{Lem:Convergence.PEqn42} explicitly and using \eqref{Eqn:DefBreveT} we verify
\begin{align*}
\breve{N}_\omega\at{\tilde{x}}=\exp\at{-\eps_\omega\tilde{x}}\breve{M}_\omega\at{\tilde{x}}\,,\qquad \tilde{N}_\omega\at{\tilde{x}}=\exp\at{-\eps_\omega\tilde{x}}\tilde{M}_\omega\at{\tilde{x}}\,,
\end{align*}
so \eqref{Lem:Convergence.PEqn43}+\eqref{Lem:Convergence.PEqn44}+\eqref{Lem:Convergence.PEqn17}+\eqref{Lem:Convergence.PEqn13} provide
\begin{align*}
\sup_{\abs{\tilde{x}}\leq 3\xi_\omega}\abs{\breve{N}_\omega\at{\tilde{x}}-\tilde{N}_\omega\at{\tilde{x}}}\leq \exp\at{3\eps_\omega\xi_\omega}\bnorm{\breve{M}_\omega-\tilde{M}_\omega}_\infty\leq C\delta_\omega^l
\end{align*}
as well as
\begin{align*}
\sup_{\tilde{x}\in 3 I_\omega}\abs{\breve{N}_\omega^\prime\at{\tilde{x}}-\tilde{N}_\omega^\prime\at{\tilde{x}}}\leq \eps_\omega\bnorm{\breve{M}_\omega-\tilde{M}_\omega}_\infty+
\exp\at{3\eps_\omega\xi_\omega}\bnorm{\breve{M}_\omega^\prime-\tilde{M}_\omega^\prime}_\infty\leq C\delta_\omega^{l+1}\,.
\end{align*}
In particular, \eqref{Lem:Convergence.PEqn45} holds also with $\tilde{N}_\omega$ instead of $\breve{N}_\omega$ and this   implies
\begin{align*}
\delta_\omega \int\limits_{-3\xi_\omega}^{+3\xi_\omega}\babs{\tilde{N}_\omega\at{\tilde{x}}}\dint{x}\leq C\,.
\end{align*}
Moreover, our definitions also   ensure
\begin{align*}
\tilde{B}_\omega\at{\tilde{x}}=\exp\at{\la_\omega\delta_\omega\beta_\omega\tilde{x}}\tilde{N}_\omega\at{\tilde{x}}
\end{align*}
and combining this with the Taylor estimate
\begin{align*}
\sup_{\abs{\tilde{x}}\leq 3\xi_\omega}\babs{\exp\at{\la_\omega\delta_\omega\beta_\omega\tilde{x}}-1}
\leq C\abs{\la_\omega\delta_\omega\tilde{x}}\leq C\abs{\la_\omega}
\end{align*}
we deduce that
\begin{align*}
\delta_\omega \int\limits_{-3\xi_\omega}^{+3\xi_\omega}\babs{\tilde{B}_\omega\at{\tilde{x}}-\tilde{N}_\omega\at{\tilde{x}}}\dint{x}\leq C\abs{\la_\omega} \delta_\omega
\int\limits_{-3\xi_\omega}^{+3\xi_\omega}\babs{\tilde{N}_\omega\at{\tilde{x}}}\dint{x}\leq C\abs{\la_\omega}\quad
\xrightarrow{\;\omega\to\infty\;}\quad 0
\end{align*}
as well as
\begin{align*}
\delta_\omega \int\limits_{-3\xi_\omega}^{+3\xi_\omega}\babs{\tilde{B}_\omega^\prime\at{\tilde{x}}-\tilde{N}_\omega^\prime\at{\tilde{x}}}\dint{x}\leq
C\delta_\omega \int\limits_{-3\xi_\omega}^{+3\xi_\omega}
\abs{\la_\omega}\delta_\omega\babs{\tilde{N}_\omega\at{\tilde{x}}}+
\abs{\la_\omega}
 \babs{\tilde{N}_\omega^\prime\at{\tilde{x}}}\dint{\tilde{x}}\quad
\xrightarrow{\;\omega\to\infty\;}\quad 0\,.
\end{align*}
The claim \eqref{Lem:Convergence.Eqn1} now follows immediately and the proof is complete.
\end{proof}
Notice that the proof of Lemma~\ref{Lem:Convergence} provides accurate approximation formulas for the eigenfunctions but the rather rough results in \eqref{Lem:Convergence.Eqn1} are sufficient for showing that \eqref{Eqn:JordanFcts} provides the only eigenfunctions with $\mu_\omega>-a$ and $|\nu_\omega|\leq\pi$ for $\omega$ large enough. Moreover, the asymptotic behavior of the kernel functions can be deduced from Theorem~\ref{Thm:SmoothnessNonlWaves} and Lemma~\ref{Lem:AsympCheckU}.
%
%
\subsection{Computation of the point spectrum and conclusion}\label{sect:PointSpectrum}
%
%
  We are now able to prove our main results in this chapter concerning the orbital stability of solitary waves with high-energy.

\begin{theorem}[non-existence of unstable eigenfunctions]
\label{Thm:NoOtherEigenvalues} %
Let  $\triple{\la_\omega}{S_\omega}{W_\omega}_{\omega}$ be a family of eigenvalues as in Assumption~\ref{Ass:Eigensystem}
and let $\omega$ be sufficiently large. Then $\pair{S_\omega}{W_\omega}$ cannot be symplectically orthogonal the Jordan block from Lemma~\ref{Lem:SpectralProperties}, i.e., we have $\la_\omega=0$ and $\pair{S_\omega}{W_\omega}$ is a multiple of the symmetry eigenfunction $\pair{S_{*,\,\omega}}{W_{*,\,\omega}}$.
\end{theorem}
\begin{proof}
As in the proof of Lemma~\ref{Lem:Convergence}, we abbreviate $U=\pair{S}{W}$ and always assume that $\omega$ is sufficiently large.
\par
\emph{\ul{Symplectic product between $U_{\omega}$ and $U_{\#,\,\omega}$}}:
Using the antisymmetry of $\partial_x$ as well as $\nabla^{\pm1}$,   the   symmetry of $\Delta^{\pm1}$, and the identity
$\nabla^{-1}=\Delta^{-1}\nabla$ we verify
\begin{align*}
\sigma\pair{U_\omega}{U_{\#,\,\omega}}&=\bskp{S_\omega}{\nabla^{-1}W_{\#,\,\omega}}
+\bskp{W_\omega}{\nabla^{-1}S_{\#,\,\omega}}
\\&=\bskp{\Delta^{-1}S_\omega}{\nabla W_{\#,\,\omega}}
-\bskp{\nabla W_\omega}{\Delta^{-1}S_{\#,\,\omega}}\,.
\end{align*}
Eliminating the velocity components via the eigenvalue equation \eqref{Eqn:LinEVP} and the Jordan identity \eqref{Eqn:JordanIdentities} we thus obtain
\begin{align*}
\sigma\pair{U_\omega}{U_{\#,\,\omega}}&=\bskp{\Delta^{-1}S_\omega}{\tfrac12m\, \omega\delta_\omega^{-1}S_{*,\,\omega}-\omega\partial_x S_{\#,\,\omega}}
-\bskp{\omega\la_\omega S_\omega - \omega\partial_x S_\omega}{\Delta^{-1}S_{\#,\,\omega}}\,,
\end{align*}
and rearranging terms yields
\begin{align*}
\delta_\omega^{1+m/2}\sigma\pair{U_\omega}{U_{\#,\,\omega}}&=\psi_{1,\,\omega}+\psi_{2,\omega}+ \psi_{3,\,\omega}
\end{align*}
with
\begin{align*}
\psi_{1,\,\omega}:=\frac{m}{2}\bskp{\Delta^{-1}S_\omega}{ S_{*,\,\omega}}\,,\qquad
\psi_{2,\,\omega}:=2\bskp{\partial_x\Delta^{-1}S_\omega}{ \delta_\omega S_{\#,\,\omega}}\,,\quad
\psi_{3,\,\omega}:=-\la_\omega\bskp{  \Delta^{-1}S_\omega}{\delta_\omega S_{\#,\,\omega}}\,.
\end{align*}
Our next goal is to compute the leading order terms for $\psi_{j,\omega}$.
\par
\emph{\ul{Bounds for the tail contributions}}:
By construction --- see \eqref{Eqn:SpaceScalign}+\eqref{Eqn:DefTildeG}+\eqref{Eqn:FPG.Int} --- we have
\begin{align*}
\bat{\Delta^{-1}S_\omega}\at{x}=\exp\bat{-\at{a-\iu \nu_\omega}x}  \tilde{M}_\omega  \bat{\delta_\omega ^{-1}\beta_\omega^{-1}x}\,,
\end{align*}
where   $\tilde{M}_\omega$ has been defined in \eqref{Lem:Convergence.PEqn48} and   satisfies
\begin{align*}
\bnorm{  \tilde{M}_\omega  }_\infty\leq
\bnorm{\tilde{H}_{\eps_\omega}}_\infty\bnorm{\tilde{\chi}_\omega\tilde{P}_\omega\tilde{G}_\omega}_1+
\bnorm{\tilde{H}_{\eps_\omega}}_1\bnorm{\at{1-\tilde{\chi}_\omega}\tilde{P}_\omega}_\infty\bnorm{\tilde{G}_\omega}_\infty
\end{align*}
by Young's inequality.  Using the estimates from Theorem~\ref{Thm:ExistenceNonlWaves}, Lemma~\ref{Lem:EV.FundSol}, and Lemma~\ref{Lem:Compactness}  we demonstrate
\begin{align*}
\bnorm{  \tilde{M}_\omega   }_\infty\leq
C\delta_\omega^{-1}\cdot 1+
C\delta_\omega^{-2}\cdot\delta_\omega^{m+2}\cdot\delta_\omega^{-1}\leq C\delta_\omega^{-1}\,,
\end{align*}
and similarly we derive
\begin{align*}
\delta_\omega^{-1}\bnorm{\partial_{\tilde{x}}  \tilde{M}_\omega  }_\infty&\leq
\bnorm{\delta_\omega^{-1}\partial_{\tilde{x}}\tilde{H}_{\eps_\omega}}_\infty\bnorm{\tilde{\chi}_\omega\tilde{P}_\omega\tilde{G}_\omega}_1+
\bnorm{\delta_\omega^{-1}\partial_{\tilde{x}}\tilde{H}_{\eps_\omega}}_1\bnorm{\at{1-\tilde{\chi}_\omega}\tilde{P}_\omega}_\infty\bnorm{\tilde{G}_\omega}_\infty\leq C\delta_\omega^{-1}\,.
\end{align*}
In particular, the pointwise estimate
\begin{align}
\label{Thm:NoOtherEigenvalues.PEqn1}
\babs{\bat{\Delta^{-1}S_\omega}\at{x}}+\babs{\bat{\partial_x \Delta^{-1}S_\omega}\at{x}}\leq C\delta_\omega^{-1}\exp\at{-ax}
\end{align}
holds for all $x\in\Rset$, where we used that $\delta_\omega\partial_x =\partial_{\tilde{x}}$ commutes with $\Delta^{-1}$ according to Lemma~\ref{Lem:DiscrDiffOperators}. On the other hand, Lemma~\ref{Lem:TailEstimates} ensures for fixed $c>a$ the tails estimate
\begin{align*}
\babs{S_{*,\,\omega}\at{x}}+ \babs{\delta_\omega S_{\#,\,\omega}\at{x}}\leq \delta_\omega^m\exp\at{-c\abs{x}}\abs{x} \qquad \text{for}\quad \abs{x}>\tfrac32\,,
\end{align*}
and combining this with \eqref{Thm:NoOtherEigenvalues.PEqn1} we conclude that the contributions to $\psi_{j,\,\omega}$ that stem from $\abs{x}\geq 3/2$ are of order $\DO{\delta_\omega^{m-1}}=\Do{1}$.
\par
\emph{\ul{Leading order contributions}}:
Employing the tails estimates from above, the scaling \eqref{Eqn:SpaceScalign}, and Lemma~\ref{Lem:Convergence} we verify
\begin{align*}
\psi_{1,\,\omega}=\frac{m\,\delta_\omega}{2} \int\limits_{-3\xi_\omega}^{+3\xi_\omega}
\tilde{B}_\omega\at{\tilde{x}}\tilde{S}_{*,\,\omega}\at{\tilde{x}}\dint{\tilde{x}}+\Do{1}=\frac{m\,b_\omega}{2\beta_\omega} \int\limits_{-3/2}^{0}{S}_{*,\,\omega}\at{x}\dint{x}
+\Do{1}\,,
\end{align*}
so the approximation formulas from Theorem~\ref{Thm:ExistenceNonlWaves}   ensure   that
\begin{align*}
\psi_{1,\,\omega}\quad\xrightarrow{\;\omega\to\infty\;}\quad  c\neq0\,.
\end{align*}
On the other hand, in view of \eqref{Eqn:DefCheckQ} and \eqref{Eqn:JordanFcts} and due to Theorem~\ref{Thm:SmoothnessNonlWaves}, Lemma~\ref{Lem:AsympCheckU}, and Lemma~\ref{Lem:Convergence} we conclude that
\begin{align*}
\abs{\psi_{2,\,\omega}}&\leq C\delta_\omega \int\limits_{-3\xi_\omega}^{+3\xi_\omega}
\babs{\tilde{B}_\omega^\prime\at{\tilde{x}}}\babs{\tilde{S}_{\#,\,\omega}\at{\tilde{x}}}\dint{\tilde{x}}+\Do{1}
\leq \delta_\omega \int\limits_{-3\xi_\omega}^{+3\xi_\omega}
C\babs{\tilde{B}_\omega^\prime\at{\tilde{x}}}\dint{\tilde{x}}+\Do{1}= \Do{1}
\end{align*}
as well as
\begin{align*}
\abs{\psi_{3,\,\omega}}&\leq  \abs{\la_\omega}\delta_\omega\at{\delta_\omega \int\limits_{-3\xi_\omega}^{+3\xi_\omega}
C\babs{\tilde{B}_\omega\at{\tilde{x}}}\dint{\tilde{x}}+\Do{1}}
\leq C \abs{\la_\omega} \delta_\omega
= \Do{1}\,.
\end{align*}
In summary, we have
\begin{align*}
\babs{\sigma\pair{U_\omega}{U_{\#,\,\omega}}}=\delta_\omega^{-m/2-1}\bat{c+\Do{1}}
\end{align*}
for some constant $0<c<\infty$, so $U_\omega$ cannot be symplectically orthogonal to $U_{\#,\,\omega}$.
\end{proof}
  We are now able to deduce the nonlinear orbital stability from the work by Friesecke, Mizumachi, and Pego.

\begin{corollary}[nonlinear orbital stability]
\label{Cor:NonlStability}
Suppose that the wave speed $\omega$ is sufficiently large. Then any high-energy wave $\pair{R_\omega}{V_\omega}$ from \S\ref{sect:NonlWaves} is orbitally stable   both in the sense of Friesecke-Pego (Main Result \ref{res:Stab} and Theorem~\ref{Thm:FPcrit}) and Mizumachi (Main Result \ref{res:StabEner}),
\end{corollary}
\begin{proof}
We check the Friesecke-Pego criteria as follows. Condition (P1) holds according to Theorem~\ref{Thm:ExistenceNonlWaves}, while (P2) and (P3) are provided by Lemma~\ref{Lem:Regularity} and Lemma~\ref{Lem:SpectralProperties}, respectively. Moreover, (S1)
is a consequence of the decay results in Lemma~\ref{Lem:TailEstimates} and Theorem~\ref{Thm:NoOtherEigenvalues} is just a reformulation of (S2).
\par
  Mizumachi formulates in \cite[Sec.4]{Miz09} four conditions (P1)-(P4) for the nonlinear orbital stability with respect to the energy space, where the first three ones coincide with the conditions  (P1)-(P3)  in Theorem \ref{Thm:FPcrit} and have already been validated. The fourth condition (P4) is equivalent to the condition (L) from in \cite{FP04a} and follows from the Friesecke-Pego criterion (S2) as it is shown in \cite[Thm. 1.2]{FP04a}.  In particular, the validity of (P4) for high-energy waves is granted by Theorem \ref{Thm:NoOtherEigenvalues}.  
\end{proof}
We finally mention that our stability results in Corollary \ref{Cor:NonlStability} are not uniform with respect to $\omega$. Calling upon the hard-sphere interpretation we believe that solitary waves  are getting more stable as $\omega\to\infty$ in the sense of an increasing basin of attraction. A rigorous statement, however, would require to quantify the $\omega_0$-dependence of the constants $\eta_0$ and $C_0$ in Theorem~\ref{Thm:FPcrit}, and lies beyond the scope of this paper.
%

%
\section*{Acknowledgement}

Both authors are grateful for the support by the \emph{Deutsche Forschungsgemeinschaft} (DFG individual grant HE 6853/2-1) and the \emph{London Mathematical Society} (LMS Scheme 4 Grant, Ref~41326). KM also thanks the University of M\"unster for the kind hospitality during his sabbatical stay.
\appendix %
%
\section{List of symbols}
%
%
\begin{tabular}{lll}
$\Phi$, $m$, $k$& potential, parameters for singular behavior& Assumption~\ref{Ass.Pot}
\\
$t$, $j$, $x$& time, particle index, space in comoving frame&
\\
$\chi$& indicator function for $x\in\ccinterval{-\tfrac12}{+\tfrac12}$&
\\
$a$& exponential weight parameter in \S\ref{sect:Stability} & Assumption~\ref{Def:PrmA}
\\
$C$& generic constant, depends only on $\Phi$ and $a$ &
\bigskip
\\
$\omega$ &wave speed, free parameter&
\\
$\delta_\omega$ &natural small quantity, power of $\omega$&\eqref{Eqn:DefDelta}
\\
$R_\omega$, $V_\omega$ &distance and velocity profile as function of $x$&\eqref{Eqn:TWAnsatz}
\\
$h_\omega$ & total energy of the traveling wave& \eqref{Eqn:TotalEnergy}
\bigskip %
\\
$\bar{Y}$ &solution to the nonlinear shape ODE & \eqref{Lem:AsympODE.Props.Eqn1} and Lemma~\ref{Lem:AsympODE.Props}
\\
$\bar{T}_\even$, $\bar{T}_\odd$ &base solutions to the linearized shape ODE &  \eqref{Lem:LinODE.Props.Eqn1} and Lemma~\ref{Lem:LinODE.Props}
\\
$\bar{K}$, $\bar{P}$ & limits of $\tilde{K}_\omega$ and $\tilde{P}_\omega$, given in terms of $\bar{Y}$ &\eqref{Eqn:DefBarPZ}
\bigskip
\\
$\al_\beta$, $\beta_\omega$&scaling parameters& \eqref{Eqn:DefBeta} and Lemma~\ref{Lem:DefAlpha}
\\
$\tilde{x}$& scaled space variable&\eqref{Eqn:SpaceScalign}
\\
$\xi_\omega$& value of $\tilde{x}$ corresponding to $x=1/2$&\eqref{Eqn:DefXi}
\\
$I_\omega$, $\tilde{\chi}_\omega$& interval $\ccinterval{-\xi_\omega}{+\xi_\omega}$ and its indicator function&\eqref{Eqn:DefChiAndI}
\\
$^\flat$& special differential operator &\eqref{Eqn:FlatOperator}
\bigskip %
\\
$\tilde{R}_\omega$, $\tilde{V}_\omega$ & scaled wave profiles as functions of $\tilde{x}$
&\eqref{Eqn:ScaledProfiles}
\\
$\breve{R}_\omega$, $\breve{V}_\omega$ & approximations to $\tilde{R}_\omega$ and $\tilde{V}_\omega$
&Lemma~\ref{Lem:DefBreveR} and Lemma~\ref{Lem:ApproxVelProfile}
\\
$\tilde{Y}_\omega $&analogue to $\bar{Y}$ in terms of $\tilde{R}_\omega$&\eqref{Eqn:DefTildeY}
\\
$\tilde{K}_\omega$, $\tilde{P}_\omega$ & scaled and normalized variants of $\Phi^\prime\nat{\tilde{R}_\omega}$,
$\Phi^{\prime\prime}\nat{\tilde{R}_\omega}$ &\eqref{Eqn:DefTildeP}
\\
$\tilde{Q}_\omega$, $\breve{Q}_\omega$ & scaled multiple of $\partial_{\delta_\omega} R_\omega$ and its approximation&
\eqref{Eqn:DefCheckQ} and Lemma~\ref{Lem:AsympCheckU}
\bigskip
\\
$\nabla$&centered difference operator, with respect to $x$&\eqref{Eqn:DefNabla}
\\
$\Delta=\nabla^2$&standard discrete Laplacian, with respect to $x$&\eqref{Eqn:DefLapl}
\\
$\tilde{\Delta}_{w}$&modified Laplacian with respect to $\tilde{x}$&\eqref{Eqn:DefModLapl}
\bigskip
\\
$\calL_\omega$ &  linearized traveling wave operator& \eqref{Eqn:LinFPUOp}
\\
$\la_\omega=\mu_\omega+\iu\nu_\omega$ &  eigenvalue of $\calL_\omega$ &  Assumption~\ref{Ass:Eigensystem}
\\
$S_\omega$, $W_\omega$ &  components of the eigenfunction to $\la_\omega$ &
\\
$S_{*,\,\omega}$, $W_{*,\,\omega}$&proper kernel function of $\calL_\omega$&\eqref{Eqn:JordanFcts}
and Lemma~\ref{Lem:SpectralProperties}
\\
$S_{\#,\,\omega}$, $W_{\#,\,\omega}$&cyclic kernel function of $\calL_\omega$&
\bigskip
\\
$\fspaceL^2_{\pm a}$, $\fspaceH^1_{\pm a}$, $\ell^2_{\pm a}$ &  function spaces with exponential weight& Section \S\ref{sect:FPCrit}
\\
$\sigma$ &  symplectic product & \eqref{Eqn:SympProd}
\bigskip
\\
$\eps_\omega$&further real scalar depending on $\la_\omega$& \eqref{Eqn:DefEps}
\\
$\zeta_\omega$, $J_\omega$, $\breve{\chi}_\omega$&analogues to $\xi_\omega$, $I_\omega$, $\tilde{\chi}_\omega$ with $\eps_\omega$ instead of $\delta_\omega$& \eqref{Eqn:DefBreveQuant1} and \eqref{Eqn:DefBreveQuant2}
\bigskip
\\
$\tilde{S}_\omega$ &  scaled eigenfunction $S_\omega$ &  \eqref{Eqn:DefTildeS}
\\
$\tilde{G}_\omega$, $\tilde{T}_\omega$ &  variants of $\tilde{S}_\omega$ with exponential factors &  \eqref{Eqn:DefTildeG}, \eqref{Eqn:DefTildeT}
\\
$\breve{G}_\omega$, $\breve{T}_\omega$ &  approximations to $\tilde{G}_\omega$, $\tilde{T}_\omega$ &  \eqref{Eqn:DefBreveG}, \eqref{Eqn:DefBreveT}
\bigskip
\\
$\tilde{B}_\omega$, $\tilde{M}_\omega$, $\tilde{N}_\omega$ &
$\tilde{\Delta}^{-1}_0 \tilde{S}_\omega$ and analogues for $\tilde{G}_\omega$, $\tilde{T}_\omega$
 &  Proof of Lemma \ref{Lem:Convergence}
\\
$\breve{M}_\omega$, $\breve{N}_\omega$ &
analogues to $\tilde{\Delta}^{-1}_0 \tilde{S}_\omega$ for $\breve{G}_\omega$, $\breve{T}_\omega$  &
\\
$m_{0,\,\omega}$, $m_{1,\,\omega}$ &  moment integrals of $\tilde{P}_\omega\tilde{T}_\omega$ on $J_\omega$&\eqref{DefMoments}
\bigskip
\\
$U$, $\tilde{U}$, $\tilde{F}$, \tdots& different, local meanings
\end{tabular}%
%
%
%
\bibliographystyle{alpha}
\bibliography{paper}
\end{document}